\documentclass[12pt,pagebackref=true]{article}
\usepackage{fullpage}

\usepackage{amsmath,amsfonts,bm}
\usepackage{pifont}

\usepackage{nicefrac}

\def\eqref#1{equation~\ref{#1}}

\def\1{\bm{1}}

\def\ra{{\textnormal{a}}}

\DeclareMathAlphabet{\mathsfit}{\encodingdefault}{\sfdefault}{m}{sl}
\SetMathAlphabet{\mathsfit}{bold}{\encodingdefault}{\sfdefault}{bx}{n}

\DeclareMathOperator*{\argmin}{arg\,min}

\newcommand{\circledOne}{\text{\ding{172}}}
\newcommand{\circledTwo}{\text{\ding{173}}}
\newcommand{\circledThree}{\text{\ding{174}}}
\newcommand{\circledFour}{\text{\ding{175}}}
\newcommand{\circledFive}{\text{\ding{176}}}
\newcommand{\circledSix}{\text{\ding{177}}}
\newcommand{\circledSeven}{\text{\ding{178}}}
\newcommand{\circledEight}{\text{\ding{179}}}
\newcommand{\circledNine}{\text{\ding{180}}}
\newcommand{\circledTen}{\text{\ding{181}}}

\usepackage[table]{xcolor}

\usepackage{multirow}
\usepackage[cp1250]{inputenc}
\usepackage[T1]{fontenc}
\usepackage{calligra}
\usepackage{layout}

\usepackage{amsmath}
\usepackage{amsthm}
\usepackage{amssymb}
\usepackage{amsfonts}
\usepackage{mathtools}

\usepackage{url}
\usepackage{color}
\usepackage{graphicx}
\usepackage{algorithmic}
\usepackage{algorithm}
\usepackage{verbatim}
\usepackage{appendix}

\usepackage{accents}
\newcommand{\ubar}[1]{\underaccent{\bar}{#1}}

\newcommand{\norm}[1]{\left\| #1 \right\|}
\newcommand{\normeu}[1]{\left\| #1 \right\|_2}
\newcommand{\normstar}[1]{\left\| #1 \right\|_{\star}}

\newcommand{\inp}[2]{\left\langle#1,#2\right\rangle} 

\newcommand{\cD}{\mathcal{D}}

\newcommand{\cL}{\mathcal{L}}
\newcommand{\cO}{\mathcal{O}}

\newcommand{\cN}{\mathcal{N}}

\newcommand{\cU}{\mathcal{U}}

\newcommand{\del}[1]{}
\let\la=\langle
\let\ra=\rangle

\definecolor{junglegreen}{rgb}{0.16, 0.67, 0.53}
\definecolor{lasallegreen}{rgb}{0.03, 0.47, 0.19}

\newcommand{\R}{\mathbb{R}} 

\newcommand{\eqdef}{:=} 

\newcommand{\Exp}[1]{{\rm E}\left[#1\right]}

\newcommand{\ExpSub}[2]{{\rm E}_{#1}\left[#2\right]}

\usepackage{hyperref} 
\usepackage{cleveref} 

\theoremstyle{definition}

\newtheorem{assumption}{Assumption}
\newtheorem{lemma}{Lemma}

\newtheorem{theorem}{Theorem}

\usepackage[many]{tcolorbox}

\tcolorboxenvironment{theorem}{
	colback=gray!10,
	boxrule=0pt,
	boxsep=1pt,
	left=2pt,right=2pt,top=2pt,bottom=2pt,
	oversize=2pt,
	sharp corners,
	before skip=\topsep,
	after skip=\topsep,
}

\tcolorboxenvironment{lemma}{
	colback=gray!10,
	boxrule=0pt,
	boxsep=1pt,
	left=2pt,right=2pt,top=2pt,bottom=2pt,
	oversize=2pt,
	sharp corners,
	before skip=\topsep,
	after skip=\topsep,
}

\tcolorboxenvironment{corollary}{
	colback=gray!10,
	boxrule=0pt,
	boxsep=1pt,
	left=2pt,right=2pt,top=2pt,bottom=2pt,
	oversize=2pt,
	sharp corners,
	before skip=\topsep,
	after skip=\topsep,
}

\tcolorboxenvironment{algorithms}{
	colback=gray!10,
	boxrule=0pt,
	boxsep=1pt,
	left=2pt,right=2pt,top=2pt,bottom=2pt,
	oversize=2pt,
	sharp corners,
	before skip=\topsep,
	after skip=\topsep,
}

\tcolorboxenvironment{remark}{
	colback=gray!10,
	boxrule=0pt,
	boxsep=1pt,
	left=2pt,right=2pt,top=2pt,bottom=2pt,
	oversize=2pt,
	sharp corners,
	before skip=\topsep,
	after skip=\topsep,
}

\definecolor{red}{rgb}{1.0, 0.01, 0.24}
\definecolor{orange}{rgb}{1.0, 0.43, 0.29}

\usepackage[round]{natbib}

\usepackage{graphicx}
\usepackage{subcaption}

\usepackage{hyperref}
\usepackage{url}

\usepackage{todonotes}

\title{\textbf{Better LMO-based Momentum Methods with Second-Order Information}}

\author{
Sarit Khirirat\thanks{ Correspondence to: \texttt{sarit.khirirat@kaust.edu.sa}} \qquad Abdurakhmon Sadiev \qquad Yury Demidovich \\
 \vspace {0.5cm}  Peter Richt{\'a}rik \\
\phantom{x}
    \\
    King Abdullah University of Science and Technology (KAUST) \\
    Thuwal, Saudi Arabia
}
\date{}

\begin{document}

\maketitle

\begin{abstract}
The use of momentum in stochastic optimization algorithms has shown empirical success across a range of machine learning tasks. 
Recently, a new class of stochastic momentum algorithms has emerged within the Linear Minimization Oracle (LMO) framework--leading to state-of-the-art methods, such as Muon, Scion, and Gluon, that effectively solve deep neural network training problems.
However, traditional stochastic momentum methods  offer convergence guarantees no better than the $\mathcal{O}(1/K^{1/4})$ rate. 
While several approaches--such as Hessian-Corrected Momentum (HCM)--have aimed to improve this rate, their theoretical results are generally restricted to the Euclidean norm setting. 
This limitation hinders their applicability in problems,  where arbitrary norms are often required.
In this paper, we extend the LMO-based framework by integrating HCM, and provide convergence guarantees under relaxed smoothness and arbitrary norm settings. 
We establish improved convergence rates of $\mathcal{O}(1/K^{1/3})$ for HCM, which can adapt to the geometry of the problem and achieve a faster rate than traditional momentum. 
Experimental results on training Multi-Layer Perceptrons (MLPs) and Long Short-Term Memory (LSTM) networks verify our theoretical observations.
\end{abstract}

\section{Introduction}

Stochastic momentum methods have been widely used for  solving data-intensive machine learning problems.
They are employed as the default Stochastic Gradient Descent (SGD) optimizer in open-sourced software libraries for deep neural network training, such as Pytorch~\citep{paszke2019pytorch}.
Originating from Polyak's heavy-ball method~\citep{polyak1964some}, stochastic momentum methods have been analyzed extensively in the literature~\citep{yan2018unified,yu2019linear,liu2020improved,cutkosky2020momentum,hubler2024gradient}.

Recent studies have extended the convergence analysis of stochastic momentum methods to arbitrary norm settings by employing the Linear Minimization Oracle (LMO) framework~\citep{pethick2025training,kovalev2025understanding,riabinin2025gluon}.
This framework encompasses many state-of-the-art optimization algorithms specifically used in deep neural network training, such as Muon~\citep{jordan2024muon} and Scion~\citep{pethick2025training}.
However, existing analyses of LMO-based momentum algorithms typically rely on standard smoothness assumptions, which limits their applicability to a wider range of machine learning problems.

Many modern learning problems, including distributionally robust optimization~\citep{chen2023generalized} and deep neural network training, often violate the standard $L$-smoothness assumptions. 
Empirical studies have shown that the Hessian norm  is not uniformly upper-bounded, as required by $L$-smoothness, for neural network models, including LSTM by~\citet{zhang2019gradient}, ResNet20 by ~\citet{zhang2019gradient}, and transformers by~\citet{crawshaw2022robustness}.  
These observations have motivated the development of relaxed smoothness conditions that more accurately capture the behavior of modern learning problems.
Key examples include the $(L_0,L_1)$-smoothness~\citep{zhang2019gradient,gorbunov2024methods,chen2023generalized,vankov2024optimizing}, the $(L_0,L_1,\alpha)$-smoothness~\citep{chen2023generalized}, and  the $\ell$-smoothness~\citep{li2023convex}.

Gluon, proposed by \citet{riabinin2025gluon}, is a layer-wise implementation of LMO-based algorithms using momentum, including Muon and Scion, for minimizing relaxed smooth functions.
However, its convergence under arbitrary norm settings--analogous to stochastic momentum methods under Euclidean norm settings--does not surpass the $\mathcal{O}(1/K^{1/4})$ rate in terms of the expected gradient norm.
Thus, the current analysis does not provide theoretical advantage from using momentum.

To improve the convergence of stochastic algorithms using momentum under the Euclidean setting, many works have focused on modifying the gradient estimators used in the momentum updates. 
\citet{arnold2019reducing} introduced implicit gradient transportation (IGT) as a novel gradient estimator, and subsequently, \citet{cutkosky2020momentum} established the $\cO(1/K^{2/7})$ convergence for IGT-based momentum methods. 
This convergence was further improved by recent variants, such as  STORM~\citep{cutkosky2019momentum} and MARS~\citep{yuan2024mars}, which employ stochastic gradient differences as gradient estimators. 
Other variants by~\citet{zhang2020one,salehkaleybar2022momentum,tran2022better}
leverage second-order information to construct more accurate estimators.  
Both variants ensure the improved rate of $\cO(1/K^{1/3})$, which is proven to be optimal for minimizing smooth nonconvex functions under mild  conditions~\citep{arjevani2023lower}. 
However, these momentum methods have been studied only under the standard smoothness and Euclidean norm setting.

\section{Contributions}

In this paper, we aim to generalize two variants of second-order momentum methods~\citep{salehkaleybar2022momentum,tran2022better} for minimizing nonconvex functions under relaxed smoothness assumptions and under the arbitrary norm setting.
We make the following contributions:

First, we incorporate two second-order momentum  variants into the class of LMO-based optimization algorithms~\citep{pethick2025training,kovalev2025understanding}. Our methods generalize the second-order momentum techniques originally analyzed under the Euclidean setting by~\citet{salehkaleybar2022momentum,tran2022better} to broader geometries and norm choices associated with the  problems.

Second, we establish an $\cO(1/K^{1/3})$ convergence  in the expected gradient norm for LMO-based methods using two second-order momentum variants under relaxed smoothness assumptions in the arbitrary norm setting. This improves upon the previously known $\cO(1/K^{1/4})$ rate for LMO-based methods using Polyak momentum by~\citet{pethick2025training,pethick2025generalized,riabinin2025gluon,kovalev2025understanding}. Also, our results match the known rates for second-order momentum methods under the standard $L$-smoothness and Euclidean setting~\citep{salehkaleybar2022momentum,tran2022better,sadiev2025second}. A comparison of theoretical guarantees is summarized in~\Cref{tab:TheoreticalComparisons}.

Third, we evaluate LMO-based methods with second-order momentum and other momentum updates on solving three nonconvex problems: logistic regression problems, Multi-Layer Perceptron (MLP) training problems, and Long Short-Term Memory (LSTM) network training problems. Our results validate our theory, illustrating that LMO-based methods using second-order momentum outperform those using Polyak momentum and IGT momentum.

\definecolor{LightCyan}{rgb}{0.9,0.9,0.9}
\begin{table}[h!]
\begin{center}
\renewcommand{\arraystretch}{1.4}
\resizebox{\textwidth}{!}{%
\begin{tabular}{cccccc}
\bf Methods & \bf Rate & \bf GS & \bf HS & \bf MSS  & \bf Norm Type \\ \hline\hline
\begin{tabular}{c} Unconstrained SCG \\ \citet{pethick2025training} \end{tabular} & $\cO\left( \nicefrac{1}{K^{1/4}} \right)$ & $L$ & -- & -- & Arbitrary   \\ 
\begin{tabular}{c} GGNC \\ \citet{pethick2025generalized} \end{tabular} & $\cO\left( \nicefrac{1}{K^{1/4}} \right)$ & $(L_0,L_1)$ & -- & -- & Arbitrary   \\ 
\begin{tabular}{c} Gluon \\ \citet{riabinin2025gluon} \end{tabular} & $\cO\left( \nicefrac{1}{K^{1/4}} \right)$ & $(L_0,L_1)$ & -- & -- & Arbitrary   \\ 
\begin{tabular}{c} Extrapolated Momentum \\ \citet{kovalev2025understanding} \end{tabular} &  $\cO\left( \nicefrac{1}{K^{2/7}} \right)$  & $L$ & $M$& -- & Arbitrary  \\
\begin{tabular}{c} Momentum Variance Reduction  \\ \citet{kovalev2025non} \end{tabular} &  $\cO\left( \nicefrac{1}{K^{1/3}} \right)$  & $L$ & --& $\cL$  & Arbitrary  \\
\begin{tabular}{c} LiMuon\\ \citet{huang2025limuon} \end{tabular} &  $\cO\left( \nicefrac{1}{K^{1/3}} \right)$  & $(L_0,L_1)$ & --& $(\cL_0,\cL_1)$ & Euclidean  \\
\begin{tabular}{c} Second-Order Momentum (Variant 1)  \\ \citet{salehkaleybar2022momentum} \end{tabular} &  $\cO\left( \nicefrac{1}{K^{1/3}} \right)$  & $L$ & --& --  & Euclidean  \\
\begin{tabular}{c} Second-Order Momentum (Variant 2)  \\ \citet{tran2022better} \end{tabular} &  $\cO\left( \nicefrac{1}{K^{1/3}} \right)$  & $L$ & $M$& --  & Euclidean  \\ \hline
 \rowcolor{LightCyan} \begin{tabular}{c} Extrapolated Momentum  \\ {\bf NEW} (\Cref{thm:IGT_relaxed_smoothness}) \end{tabular} & $\cO\left( \nicefrac{1}{K^{2/7}} \right)$ & $(L_0,L_1)$ & $(M_0,M_1)$ & -- & Arbitrary   \\
 \rowcolor{LightCyan} \begin{tabular}{c} Momentum Variance Reduction  \\ {\bf NEW} (\Cref{thm:MVR_relaxed_smoothness}) \end{tabular} & $\cO\left( \nicefrac{1}{K^{1/3}} \right)$ & $(L_0,L_1)$ & -- & $(\cL_0, \cL_1)$ & Arbitrary   \\  \hline 
 \rowcolor{LightCyan} \begin{tabular}{c} Second-Order Momentum (Variant 1)  \\  {\bf NEW} (\Cref{lemma:V1_relaxed_smoothness}) \end{tabular}
&  $\cO\left( \nicefrac{1}{K^{1/3}} \right)$ & $(L_0,L_1)$ & --& --  & Arbitrary   \\
 \rowcolor{LightCyan} \begin{tabular}{c} Second-Order Momentum (Variant 2)  \\ {\bf NEW} (\Cref{thm:V2_relaxed_smoothness}) \end{tabular} & $\cO\left( \nicefrac{1}{K^{1/3}} \right)$ & $(L_0,L_1)$ & $(M_0,M_1)$ & -- & Arbitrary   \\
 \hline \hline
\end{tabular}
}
\end{center}
\caption{Comparisons of convergence for
LMO-based methods that utilize Polyak momentum like Gluon, and second-order momentum methods, and extrapolated momentum (implicit gradient transportation), and momentum variance reduction. We use the following abbreviations: {\bf GS} = Gradient Smoothness, {\bf HS} = Hessian Smoothness, and {\bf MSS} = Mean-Squared Smoothness.
}
\label{tab:TheoreticalComparisons}
\end{table}

\section{Related Works}

In this section, we review the existing literature on stochastic momentum methods, Muon, and relaxed smoothness, all of which are closely related to our work.

\textbf{Stochastic momentum methods.}
Momentum is widely used  to expedite the training process in stochastic optimization methods.
In the Euclidean norm setting, stochastic momentum methods enjoy the $\cO(1/K^{1/4})$ convergence under standard smoothness~\citep{yan2018unified,yu2019linear,liu2020improved,cutkosky2020momentum}, and recently under relaxed smoothness~\citep{hubler2024gradient}. 
Moreover, momentum has been incorporated to enhance the performance of other methods, such as distributed methods with communication compression for communication-efficient learning~\citep{fatkhullin2023momentum,khirirat2024error}, and LMO-based optimization methods for deep neural network training~\citep{jordan2024muon,pethick2025training}.
To improve upon the $\cO(1/K^{1/4})$ convergence rate, several novel momentum updates have been proposed. 
For instance, stochastic methods employing IGT momentum (or extrapolated momentum)~\citep{arnold2019reducing} have been shown by~\citet{cutkosky2020momentum} to achieve an improved rate of $\cO(1/K^{2/7})$. Furthermore, the methods based on STORM~\citep{cutkosky2019momentum}, MARS~\citep{yuan2024mars}, and Hessian-corrected momentum~\citep{salehkaleybar2022momentum,tran2022better,zhang2020improved} attain the convergence rate of $\cO(1/K^{1/3})$, which matches the known lower bound established by~\citet{arjevani2023lower}.
However, the convergence guarantees of these novel stochastic momentum methods were restricted to the Euclidean norm setting. 
In this paper, we will incorporate Hessian-corrected momentum  into the LMO-based optimization methods to guarantee the $\cO(1/K^{1/3})$ convergence under the arbitrary norm setting, which improves upon the $\cO(1/K^{1/4})$ convergence achieved by LMO-based methods using Polyak momentum.

\textbf{Muon.}
Muon optimizers~\citep{jordan2024muon} update the model parameters (with their inherent matrix structure) with orthogonalized gradient momentum.
Empirical studies, e.g. by~\citet{jordan2024muon,liu2025muon}, highlight the superior performance of Muon  over Shampoo~\citep{gupta2018shampoo}, SOAP~\citep{vyas2024soap}, and AdamW~\citep{loshchilov2017decoupled}, for language model training.
From these encouraging observations, the convergence of Muon  was  analyzed initially by~\citet{li2025note} under the standard smoothness with respect to the Frobenius norm. 
Later, LMO-based optimization methods with momentum were proposed by~\citet{pethick2025training} to capture optimizers, including Muon, Scion, and stochastic Euclidean  momentum methods with momentum.  
Under the arbitrary norm setting, the methods were shown to enjoy the $\cO(1/K^{1/4})$ convergence under the standard smoothness~\citep{pethick2025training,kovalev2025understanding}, and under the relaxed smoothness~\citep{pethick2025generalized,riabinin2025gluon}.

\textbf{Relaxed smoothness.}
Training deep neural networks, especially large-scale language models, poses significant challenges due to their highly non-convex functions. To better capture this, a variety of relaxed smoothness conditions on the functions have been proposed. One early contribution by~\citet{zhang2019gradient} introduced a function class in which the Hessian norm grows linearly with the gradient norm, providing a more flexible alternative to classical smoothness assumptions. Building upon this, several works have introduced broader classes of functions, such as the symmetric $(L_0, L_1, \alpha)$-smoothness condition~\citep{chen2023generalized}, the $\ell$-smoothness condition~\citep{li2023convex}, and the $(\rho, K_0, K_\rho)$-smoothness condition~\citep{liu2025theoretical}. These conditions aim to more accurately reflect the optimization landscape encountered in deep learning. Under such relaxed assumptions, the convergence behavior of various gradient-based methods has been extensively studied~\citep{crawshaw2022robustness, hubler2024gradient, zhang2019gradient, koloskova2023revisiting, chezhegov2025convergence, gorbunov2024methods, vankov2024optimizing, khirirat2024error, faw2023beyond,wang2023convergence,zhao2021convergence,hubler2024parameter}, providing theoretical insights that align more closely with empirical observations in neural network  training.

\section{Problem Formulation}
Throughout this paper, we consider the following unconstrained stochastic optimization problems
\begin{eqnarray}\label{eqn:problem}
    \underset{x\in\R^d}{\min} \ f(x) := \ExpSub{\xi \sim \cD}{f_{\xi}(x)},
\end{eqnarray}
where $f_\xi(x)$ is a differentiable but possibly nonconvex function,  $x\in\R^d$ is the $d$-dimensional vector of model parameters, and $\xi$ is a random vector, which represents a training data sample drawn from an unknown data distribution $\cD$.

Often, we  access a stochastic oracle to compute the stochastic gradient and Hessian of $f$ at $x$ with respect to $\xi\sim \cD$ denoted by $\nabla f_\xi(x)$ and $\nabla^2 f_\xi(x)$, respectively. 
We  assume that the stochastic gradient and stochastic Hessian satisfy the following unbiased and variance-bounded properties: 
\begin{assumption}\label{assum:Hessian_boundedvariance_v2}
 $\nabla f_\xi(x)$ and $\nabla^2 f_{\xi}(x)$ is an unbiased and variance-bounded estimator of $\nabla f(x)$ and $\nabla^2 f(x)$, respectively, i.e. for  all $x,w\in\R^d$,
\begin{align*}
& \ExpSub{\xi}{\nabla f_\xi(x)}= \nabla f(x), \quad  \ExpSub{\xi}{\normeu{\nabla f_\xi(x)-\nabla f(x)}^2} \leq \sigma^2_g, \quad  \\  &
\ExpSub{\xi}{\nabla^2 f_\xi(x)}= \nabla^2 f(x), \quad \text{and} \quad \ExpSub{\xi}{\normeu{(\nabla^2 f_\xi(x)-\nabla^2 f(x))w}^2} \leq \sigma^2_H \normeu{w}^2. 
\end{align*}
\end{assumption}

\Cref{assum:Hessian_boundedvariance_v2} is commonly used for analyzing the convergence of optimization algorithms using stochastic gradients~\citep{ghadimi2013stochastic,cutkosky2019momentum,gorbunov2020linearly} and stochastic Hessians~\citep{tran2022better,jiang2024stochastic}.

We further impose the following standard assumptions on objective functions.

\begin{assumption}\label{assum:lower_bound}
The function $f:\R^d\rightarrow \R$ is bounded from below, i.e., $f^{\inf}=\inf_{x\in\R^d} f(x) > -\infty$.    
\end{assumption}

\begin{assumption}\label{assum:Relaxed_Lipschitz_Grad}
The function $f:\R^d\rightarrow\R$ is differentiable, and its gradient $\nabla f(x)$ is symmetrically $(L_0,L_1)$-Lipschitz continuous with respect to the norm $\norm{\cdot}$, i.e. for all $x,y\in\R^d$, 
\begin{eqnarray*}
    \normstar{\nabla f(x) - \nabla f(y)} \leq (L_0 + L_1 \underset{\theta \in [0,1]}{\sup} \normstar{\nabla f(\theta x + (1-\theta)y)}) \norm{x-y}.
\end{eqnarray*}
\end{assumption}

\begin{assumption}\label{assum:Relaxed_Lipschitz_Hessian}
The function $f:\R^d\rightarrow \R$ is twice differentiable, and its Hessian $\nabla^2 f(x)$ is symmetrically $(M_0,M_1)$-Lipschitz continuous with respect to the norm $\norm{\cdot}$, i.e. for all $x,y\in\R^d$, 
\begin{eqnarray*}
    \normstar{\nabla^2 f(x) - \nabla^2 f(y)} \leq (M_0 + M_1 \underset{\theta \in [0,1]}{\sup} \normstar{ \nabla f(\theta x + (1-\theta)y)}) \norm{x-y}.
\end{eqnarray*}
\end{assumption}

Assumptions~\ref{assum:lower_bound},~\ref{assum:Relaxed_Lipschitz_Grad}, and~\ref{assum:Relaxed_Lipschitz_Hessian} ensure the lower-bound of the objective function $f$, the symmetric relaxed Lipschitz continuity of its gradient $\nabla f$, and the symmetric relaxed Lipschitz continuity of its Hessian $\nabla^2 f$, respectively.
On the one hand, \Cref{assum:Relaxed_Lipschitz_Grad} is a generalization of the symmetric $(L_0,L_1)$-Lipschitz continuity of $\nabla f(\cdot)$ with respect to the Euclidean norm by~\citet{gorbunov2024methods,chen2023generalized}. 
Also, \Cref{assum:Relaxed_Lipschitz_Grad}  recovers $L$-Lipschitz continuity with respect to the arbitrary norm by~\citet{kovalev2025understanding,pethick2025training} when we let $L_0=L$ and $L_1=0$.
On the other hand, \Cref{assum:Relaxed_Lipschitz_Hessian} is a generalization of  the asymmetric $(M_0,M_1)$-Lipschitz continuity of the Hessian with respect to the Euclidean norm by~\citet[Assumption 3]{xie2024trust}. Also, \Cref{assum:Relaxed_Lipschitz_Hessian}  obtains the $L_H$-Lipschitz continuity of the Hessian by~\citet{carmon2018accelerated,nesterov2006cubic}, when we let $M_0=L_H$ and $M_1 = 0$.

Finally, note that the variance (\Cref{assum:Hessian_boundedvariance_v2}) is measured with respect to the Euclidean norm, while the Lipschitz continuity of the gradient (\Cref{assum:Relaxed_Lipschitz_Grad})  and Hessian (\Cref{assum:Relaxed_Lipschitz_Hessian}) with respect to the arbitrary norm and its dual norm. 
Hence, we describe a connection between these norms using the following inequality: 
\begin{eqnarray}\label{eqn:norm_diff}
{\color{red}\ubar{\rho}}\normeu{x} \leq \normstar{x} \leq {\color{red}\bar\rho} \normeu{x}, \quad \text{and} \quad {\color{blue}\ubar{\theta}} \normeu{x} \leq  \norm{x} \leq {\color{blue}\bar \theta} \normeu{x}.
\end{eqnarray}

\subsection{LMO-based  Methods with Momentum}\label{sec:LMO}

To solve the problem in~\eqref{eqn:problem} for neural network training tasks, we focus on the LMO-based optimization methods using momentum~\citep{pethick2025training,kovalev2025understanding}, which update the iterates $\{x_k\}_{k \geq 0}$ as follows: Given $x_0,m_0 \in\R^d$, 
\begin{eqnarray}\label{eqn:LMO_momentum}
    x_{k+1} = x_k + \text{lmo}(m_k), \quad \text{and} \quad m_{k+1} = (1-\alpha_k)m_k + \alpha_k \nabla f_{\xi_{k+1}}(x_{k+1}).
\end{eqnarray}
Here, $\eta_k>0$ is a stepsize, $\alpha_k \in (0,1)$ is a momentum parameter,  $\nabla f_{\xi}(x)$ is the stochastic gradient, and $\text{lmo}(m_k) \eqdef {\text{argmin}}_{ \norm{x} \leq \eta_k } \inp{m_k}{x}$. 
The methods in~\eqref{eqn:LMO_momentum} encompass a class of stochastic momentum methods under specific norm choices. 
For instance, they recover normalized stochastic momentum methods~\citep{cutkosky2020momentum} when we let $\norm{\cdot} = \normeu{\cdot}$, and    sign stochastic momentum methods~\citep{jiang2025improved} when we let $\norm{\cdot} = \norm{\cdot}_{\infty}$.

The LMO-based methods using momentum in~\eqref{eqn:LMO_momentum} were  shown by~\citet{pethick2025training,kovalev2025understanding,riabinin2025gluon} to attain the $\cO\left( 1/K^{1/4} \right)$ convergence rate in the gradient norm  under the arbitrary norm setting, The obtained rate is the same as that of stochastic momentum methods under the Euclidean norm setting, e.g. by~\citet{cutkosky2020momentum,jiang2025improved,zhang2020improved,hubler2024parameter}.

To further improve the iteration complexity of the momentum methods in~\eqref{eqn:LMO_momentum}, we may use three momentum variants: extrapolated momentum, Stochastic Recursive Momentum (STORM), and second-order momentum.

First, extrapolated momentum~\citep{arnold2019reducing,cutkosky2020momentum}   updates the momentum vector $m_k$ according to: 
\begin{eqnarray}\label{eqn:extrapolated_m}
   m_{k+1} = (1-\alpha_k)m_k + \alpha_k \nabla f_{\xi_{k+1}} \left( y_{k+1} \right),
\end{eqnarray}
where $y_{k+1} =  x_{k+1} + (\nicefrac{(1-\alpha_k)}{\alpha_k} )(x_{k+1} - x_{k})$. 
LMO-based methods using extrapolated momentum enjoys the $\cO(1/K^{2/7})$ convergence rate under the standard smoothness and arbitrary norm setting~\citet{kovalev2025understanding}, which matches the rate under the Euclidean norm setting by~\citet[Theorem 3]{cutkosky2020momentum}, ~\citet[Theorem 5]{cutkosky2021high} in the bounded variance case.

To further improve the convergence of stochastic methods using extrapolated momentum, we may consider STORM~\citep{cutkosky2019momentum}. The key technique  is to use the stochastic gradient difference $\nabla f_{\xi_{k+1}}(x_{k+1}) - \nabla f_{\xi_{k+1}}(x_k)$ as the correction term to expedite the training process. In STORM, the momentum vector $m_k$ is updated via: 
\begin{eqnarray}
m_{k+1} = (1-\alpha_k)(m_k + [\nabla f_{\xi_{k+1}}(x_{k+1}) - \nabla f_{\xi_{k+1}}(x_k)]) + \alpha_k \nabla f_{\xi_{k+1}} \left( x_k \right).    
\end{eqnarray}
LMO-based methods using STORM have been recently shown to achieve the $\cO(1/K^{1/3})$ convergence rate under the relaxed smoothness and Frobenius norm setting recently by~\citet{huang2025limuon}, which recovers the rate under the standard smoothness and Euclidean norm setting by~\citet{cutkosky2019momentum}.
To better control the gradient correction term in STORM,~\citet{yuan2024mars} introduce a scaling parameter $\beta_k \in (0,1)$. 
The resulting algorithm, named MARS, modifies STORM by adjusting the update rule for the momentum vector $m_k$ as follows: 
\begin{eqnarray}
\label{eq:MARS}
m_{k+1} = (1-\alpha_k) \left(m_k + {\color{orange} \frac{\beta_k}{1-\alpha_k}}[\nabla f_{\xi_{k+1}}(x_{k+1}) - \nabla f_{\xi_{k+1}}(x_k)] \right) + \alpha_k \nabla f_{\xi_{k+1}} \left( x_k \right).   
\end{eqnarray}
By tuning the scaling factor $\beta_k$ properly, MARS can provide the finer control over the gradient difference term, thus leading to improved convergence performance. 
Note that MARS with $\beta_k =1-\alpha_k$ becomes STORM.

Finally, other variants that achieve the same  rate as STORM are second-order momentum~\citep{salehkaleybar2022momentum,tran2022better,zhang2020one} (or Hessian-corrected momentum). 
They utilize the Hessian-vector product as the correction term, and perform  the following update: 
\begin{eqnarray}\label{eqn:second_order_momentum_update}
    m_{k+1} = (1-\alpha_k)(m_k + \nabla^2 f_{\xi_{k+1}}(\hat x_{k+1})(x_{k+1}-x_k)) + \alpha_k \nabla f_{\xi_{k+1}}(x_{k+1}).
\end{eqnarray}
The second-order momentum update in~\eqref{eqn:second_order_momentum_update} becomes the one proposed by~\citet{tran2022better} when $\hat x_{k+1} = x_{k+1}$, and the one by~\citet{salehkaleybar2022momentum} when $\hat x_{k+1} = b_k x_{k+1} + (1-b_k)x_k$ with  $b_k \in \R$  generated by the uniform distribution $\cU(0,1)$.
Both STORM and second-order momentum  ensures the $\cO(1/K^{1/3})$ convergence, which is faster than the $\cO(1/K^{2/7})$ convergence of extrapolated momentum and the $\cO(1/K^{1/4})$ convergence of Polyak momentum.

\section{LMO-based Methods with Second-order Momentum}

Now, we consider LMO-based methods~\citep{pethick2025training,kovalev2025understanding} using  second-order momentum~\citep{salehkaleybar2022momentum,tran2022better,zhang2020one}.
The methods leverage the LMO-based oracle in~\eqref{eqn:LMO_momentum},  
second-order momentum updates in~\eqref{eqn:second_order_momentum_update}, and the scaling factor $\beta_k$ introduced in MARS~\citep{yuan2024mars}.
See \Cref{alg:2_order_momentum} for the detailed description of the methods. 

In~\Cref{alg:2_order_momentum}, the second-order momentum  incorporates the Hessian-vector product term into its update to accelerate the convergence of the LMO-based momentum methods. 
The Hessian-vector product requires roughly the same computational time as the gradient~\citep{pearlmutter1994fast}.
This can be achieved by using the automatic differentiation package to compute $\nabla^2 f(x) v$ for $x,v\in\R^d$ by evaluating $\nabla h(x)$, where  $h(x) = \inp{\nabla f(x)}{v}$~\citep{tran2022better}.

\begin{algorithm}
\caption{LMO-based Optimization Methods with Second-order Momentum}
\label{alg:2_order_momentum}
\begin{algorithmic}[1]
\STATE \textbf{Input:} Tuning parameters $\eta_k>0$ and $\alpha_k , {\color{orange} \beta_k} \in (0,1)$ for $k=0,1,\ldots$; Initial points $x_0 , m_0 \in \R^d$; Total number of iterations $K>0$
\FOR{each iteration $k = 0, 1, \dots, K$}
        \STATE Update $x_{k+1} = x_k + \text{lmo}(m_k)$, where $\text{lmo}(m_k) \eqdef \underset{ \norm{x} \leq \eta_k }{\text{argmin}} \inp{m_k}{x}$, and sample $\xi_{k+1} \sim \cD$
        \IF{Variant 1}
        \STATE Set $\hat x_{k+1}=b_{k+1} x_{k+1} + (1-b_{k+1})x_k $, where $b_{k+1} \sim \cU(0,1)$
        \STATE Update $m_{k+1} =  (1-\alpha_k)\left(m_k + {\color{orange} \frac{\beta_k}{1-\alpha_k}}\nabla^2 f_{\xi_{k+1}}(  \hat x_{k+1}) (x_{k+1} - x_k) \right) + \alpha_k \nabla f_{\xi_{k+1}}\left(x_{k+1} \right)$
        \ELSIF{Variant 2}
        \STATE Update $m_{k+1} =  (1-\alpha_k)\left( m_k + {\color{orange} \frac{\beta_k}{1-\alpha_k}}\nabla^2 f_{\xi_{k+1}}( x_{k+1}) (x_{k+1} - x_k) \right) + \alpha_k \nabla f_{\xi_{k+1}}\left(x_{k+1} \right)$
        \ENDIF
\ENDFOR
\end{algorithmic}
\end{algorithm}

To this end, we provide the convergence theorems for~\Cref{alg:2_order_momentum} with $\beta_k=1-\alpha_k$.

\begin{theorem}[Variant 1 of~\Cref{alg:2_order_momentum}]\label{lemma:V1_relaxed_smoothness} Consider the problem of minimizing $f(x)=\ExpSub{\xi \sim \cD}{f_{\xi}(x)}$. Let $f$ be twice differentiable, and let   Assumption~\ref{assum:Hessian_boundedvariance_v2},~\ref{assum:lower_bound}, and~\ref{assum:Relaxed_Lipschitz_Grad} hold. 
Then, the iterates $\{x_k\}$ generated by Variant 1 of \Cref{alg:2_order_momentum} with $\beta_k=1-\alpha_k$ and with
\begin{eqnarray*}
  \alpha_k =  \alpha = \frac{1}{(K+1)^{ \nicefrac{2}{3}}}, \quad \text{and} \quad  \eta_k= \eta = \frac{\hat \eta}{(K+1)^{  \nicefrac{2}{3}}},
\end{eqnarray*}
where $\hat \eta =\frac{1}{80L_1} \left(\nicefrac{{\color{red}\bar\rho}}{{\color{red}\ubar\rho}}\right)^{-1}$ satisfy  
\begin{eqnarray*}
&& \min_{k \in \{0,1,\ldots,K\}} \Exp{\normstar{\nabla f(x_k)}}
\leq  \frac{2 (f(x_0) - f_{\inf})}{\hat \eta (K+1)^{\nicefrac{1}{3}}} +  \frac{4\left( \Exp{\normstar{e_0}} + {\color{red}\bar\rho}\sigma_g\right)}{(K+1)^{\nicefrac{1}{3}}} +  \frac{\hat \eta L_0 \exp(L_1 \hat\eta)}{(K+1)^{\nicefrac{2}{3}}}  \\
&&  \hspace{1.0cm} + 12\frac{{\color{red}\bar\rho}}{{\color{red} \ubar \rho}} \frac{  \hat\eta L_0 }{(K+1)^{ \nicefrac{1}{3} }}   \left(\exp( L_1 \hat\eta) + \sqrt{L_1 \hat \eta} \exp( L_1 \hat\eta) + 2\right)   + 8\frac{ {\color{red}\bar\rho}}{{\color{blue}\ubar\theta}} \frac{ \hat\eta \sigma_H}{(K+1)^{ \nicefrac{1}{3} }} .
\end{eqnarray*}
\end{theorem}

\begin{theorem}[Variant 2 of~\Cref{alg:2_order_momentum}]\label{thm:V2_relaxed_smoothness}
 Consider the problem of minimizing $f(x)=\ExpSub{\xi \sim \cD}{f_{\xi}(x)}$. Let $f$ be twice differentiable, and let   Assumptions~\ref{assum:Hessian_boundedvariance_v2},~\ref{assum:lower_bound},~\ref{assum:Relaxed_Lipschitz_Grad}, and~\ref{assum:Relaxed_Lipschitz_Hessian} hold. 
Then, the iterates $\{x_k\}$ generated by Variant 2 of \Cref{alg:2_order_momentum} with $\beta_k=1-\alpha_k$ and with 
\begin{eqnarray*}
  \alpha_k =  \alpha = \frac{1}{(K+1)^{\nicefrac{2}{3}}}, \quad \text{and} \quad  \eta_k= \eta = \frac{\hat \eta}{(K+1)^{\nicefrac{2}{3}}},
\end{eqnarray*}
where $\hat \eta =  \frac{1}{3}\min\left\{ \frac{1}{L_1} , \frac{1}{\sqrt{M_1 } }  \right\}$  satisfy  
\begin{eqnarray*}
    && \min_{k \in \{0,1,\ldots,K\}} \Exp{\normstar{\nabla f(x_k)}} 
    \leq   \frac{2(f(x_0) - f_{\inf})}{\hat\eta(K+1)^{\nicefrac{1}{3}}} + \frac{4\left( \Exp{\normstar{e_0}} + {\color{red}\bar\rho} \sigma_g\right)}{ (K+1)^{\nicefrac{1}{3}}}  + \frac{\hat\eta L_0 \exp(L_1 \hat \eta) }{(K+1)^{\nicefrac{2}{3}}}  \\
    && \hspace{1.0cm}+ \frac{2\hat \eta^2}{(K+1)^{\nicefrac{2}{3}}}\left(M_0 + \frac{2}{3}M_1 \hat \eta L_0 \exp(L_1\hat \eta)\right)  + 4\frac{{\color{red}\bar\rho}}{\color{blue} \ubar{\theta}} \frac{\hat \eta \sigma_H }{(K+1)^{\nicefrac{1}{3}}}  . 
\end{eqnarray*}
\end{theorem}

From Theorems~\ref{lemma:V1_relaxed_smoothness} and~\ref{thm:V2_relaxed_smoothness}, \Cref{alg:2_order_momentum} using two second-order momentum variants with $\beta_k=1-\alpha_k$ achieves the $\cO(1/K^{1/3})$ convergence in the expected gradient norm for minimizing relaxed smooth functions under the arbitrary norm setting.
Our methods match the convergence rates of second-order momentum methods established for standard smooth functions under the Euclidean norm setting by~\citet{salehkaleybar2022momentum,tran2022better} and under bounded variance assumption by~\citet{sadiev2025second}.
Our results do not require large minibatch sizes for computing stochastic gradients, and do not impose restrictive assumptions, such as  bounded stochastic gradient norms by~\citet{tran2022better}.
Moreover, our methods converge faster than LMO-based optimization methods using Polyak momentum analyzed by~\citet{pethick2025training, pethick2025generalized,kovalev2025understanding,riabinin2025gluon}. 
Also, note that Variant 2 of \Cref{alg:2_order_momentum} assumes the symmetric smoothness on the Hessian, which Variant 1 of \Cref{alg:2_order_momentum} does not. 
Moreover, although our theoretical results are established for \Cref{alg:2_order_momentum} with $\beta_k=1-\alpha_k$, the analysis can be extended to the case of general $\beta_k \in (0,1]$. In such cases, the convergence bounds include an additional term depending on the choice of $\beta_k$. 
Nonetheless, determining an optimal choice of $\beta_k$ remains unclear. We leave a tight convergence analysis under appropriate $\beta_k$ selection as an open direction for future work.

\textbf{Extension to extrapolated momentum.}
We can leverage our analysis to establish  the convergence of LMO-based methods in~\eqref{eqn:LMO_momentum}, which incorporate extrapolated momentum in~\eqref{eqn:extrapolated_m}, for minimizing relaxed smooth functions in the arbitrary norm setting (see~\Cref{app:IGT_relaxed_smoothness}). 
The resulting convergence matches the $\cO(1/K^{2/7})$ rate obtained by \citet[Corollary 5]{kovalev2025understanding} under traditional smoothness.

\textbf{Extension to momentum variance reduction.}
Our analysis can also be used to derive the convergence of  
the LMO-based methods in~\eqref{eqn:LMO_momentum} that incorporate momentum variance reduction (MVR/STORM/MARS), as defined in~\eqref{eq:MARS}. We derive the results for minimizing relaxed smooth functions in an arbitrary norm setting (see~\Cref{app:MVR_relaxed_smoothness}), assuming generalized mean-squared smoothness in an arbitrary norm setting (see~\Cref{assum:Relaxed_Mean_Squared_Lipschitz_Grad}). 
Notably, our derived convergence rate matches the $\cO(1/K^{1/3})$ rate obtained by \citet[Theorem 3]{kovalev2025non} under traditional matrix-type smoothness.

\section{Numerical Experiments}

We conduct numerical experiments to investigate the performance of LMO-based  methods using Polyak momentum, extrapolated momentum~\citep{cutkosky2020momentum}, and two second-order momentum variants (\Cref{alg:2_order_momentum}). 
Specifically, the methods using two second-order momentum variants with $\beta_k=1-\alpha_k$ are referred to as SOM-V1 and SOM-V2, while those with general $\beta_k$ values are denoted by $\beta$-SOM-V1 and $\beta$-SOM-V2. 
We implemented these algorithms using PyTorch~\citep{paszke2019pytorch}, and benchmarked them for two nonconvex problems that satisfy symmetric $(L_0,L_1)$ smoothness: problems of training Multi-Layer Perceptrons (MLPs) and Long Short-Term Memory (LSTM) networks. We reported our results for logistic regression problems with nonconvex regulaization in~\Cref{sec:exp}. 
For all experiments, each element of the initial point $x_0\in \R^d$ was generated from the standard normal distribution $\cN(0,1)$, the random seed was fixed, and also the learning rate $\eta_k$ and the momentum parameter $\alpha_k$ were chosen as follows: (1) $\alpha_k = \frac{1}{\sqrt{k+1}}$ and $\eta_k = \frac{\eta_0}{(k+1)^{3/4}}$ for Polyak momentum, (2) $\alpha_k = \frac{1}{(k+1)^{4/7}}$ and $ \eta_k = \frac{\eta_0}{(k+1)^{5/7}}$ for extrapolated momentum, and $\alpha_k = \frac{1}{(k+1)^{2/3}}$ and $\eta_k = \frac{\eta_0}{(k+1)^{2/3}}$ for all second-order momentum variants.

\subsection{MLP}

We evaluated the algorithms for binary classification tasks using the MLP model with two hidden layers over  the \texttt{splice} dataset from the \texttt{libsvm} library~\citep{chang2011libsvm}. 
The dataset comprises $1000$ training samples and $60$ features. 
We minimize the objective function $f(x) = \cL(x) + R(x)$, where $\cL(x)$ is the binary cross-entropy loss with logits, and $R(x)$ is the nonconvex Welsch regularizer defined by $R(x) = \lambda \sum_{i=1}^d x_i^2/(1+x_i^2)$ with $\lambda = 0.01$.




From {\color{black}Figure~\ref{fig:momentum_comparison_main_text}}, LMO-based methods using second-order momentum outperform those using Polyak or extrapolated momentum. In particular, Variant 2 of second-order momentum achieves superior convergence in training loss, compared to other momentum variants. Furthermore, {\color{black} from Figure~\ref{fig:hvp_comparison_main_text}}, the scaling factor $\beta_k$ further enhances the convergence achieved by second-order momentum. In particular, Variant 2 with the scaling factor $\beta_k$ exhibits the most consistent and strongest convergence performance in training loss, outperforming Variant 1 with the same scaling factor.
Additionally, {\color{black}Figure~\ref{fig:batch_comparison_main_text}} shows that increasing the mini-batch size for computing stochastic gradients improves the convergence performance of LMO-based methods using second-order momentum. Specifically, the methods using Variant 2 with the scaling factor $\beta_k$ achieve higher solution accuracy with a mini-batch size of $32$ compared to sizes $1$ and $16$.

\begin{figure}[h!]
    \centering
    \includegraphics[width=0.47\textwidth]{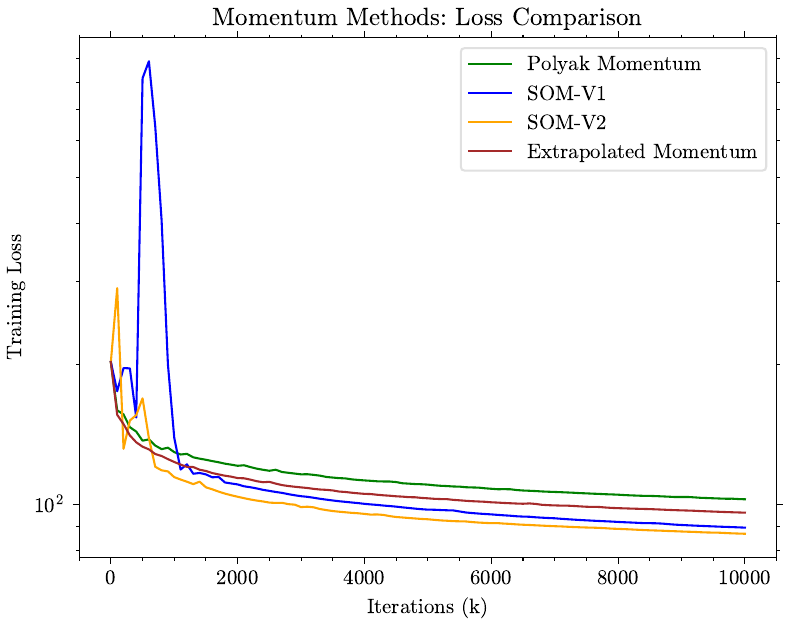}
    \includegraphics[width=0.47\textwidth]{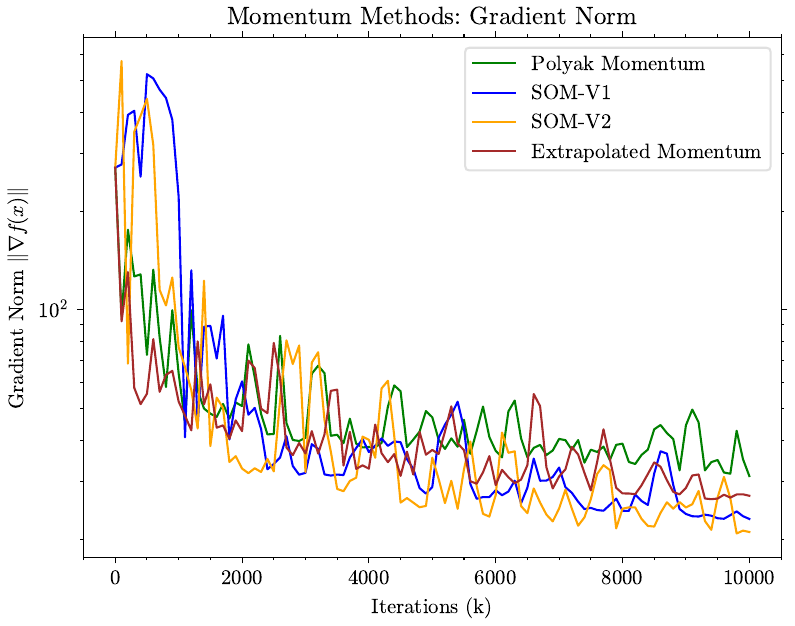}
    \caption{Comparison of training the MLP model on the \texttt{splice} dataset using LMO-based methods with various momentum updates  in the training loss (left) and gradient norm (right).
    }
    \label{fig:momentum_comparison_main_text}
\end{figure}
\begin{figure}[h!]
    \centering
    \includegraphics[width=0.47\textwidth]{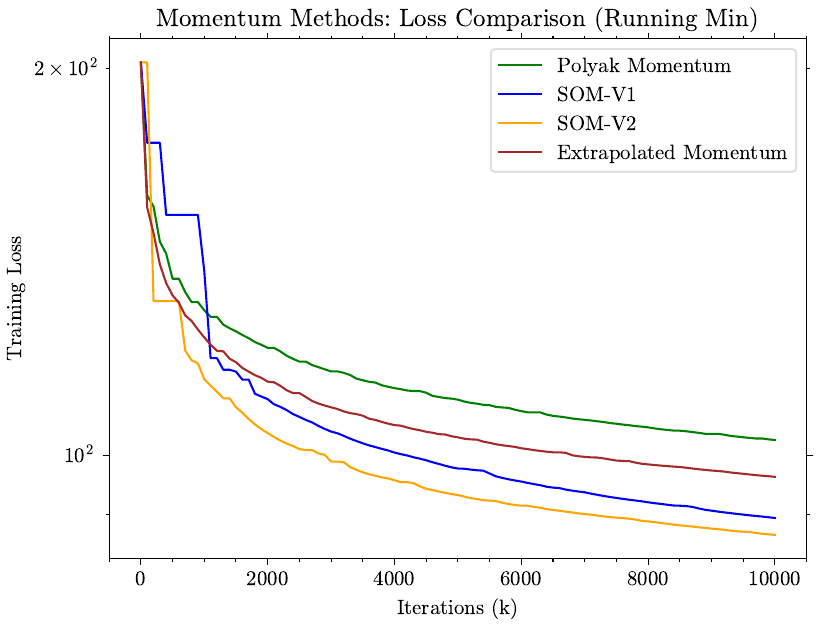}
    \includegraphics[width=0.47\textwidth]{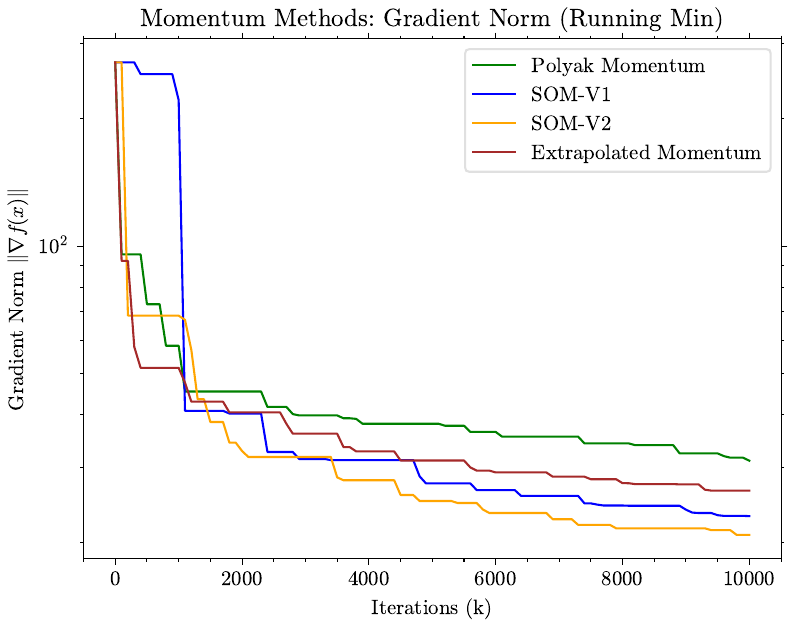}
    \caption{Comparison of training the MLP model on the \texttt{splice} dataset using LMO-based methods with second-order momentum (SOM) variants  in the training loss (left) and gradient norm (right).}
    \label{fig:hvp_comparison_main_text}
\end{figure}

\begin{figure}[h!]
    \centering
    \includegraphics[width=0.47\textwidth]{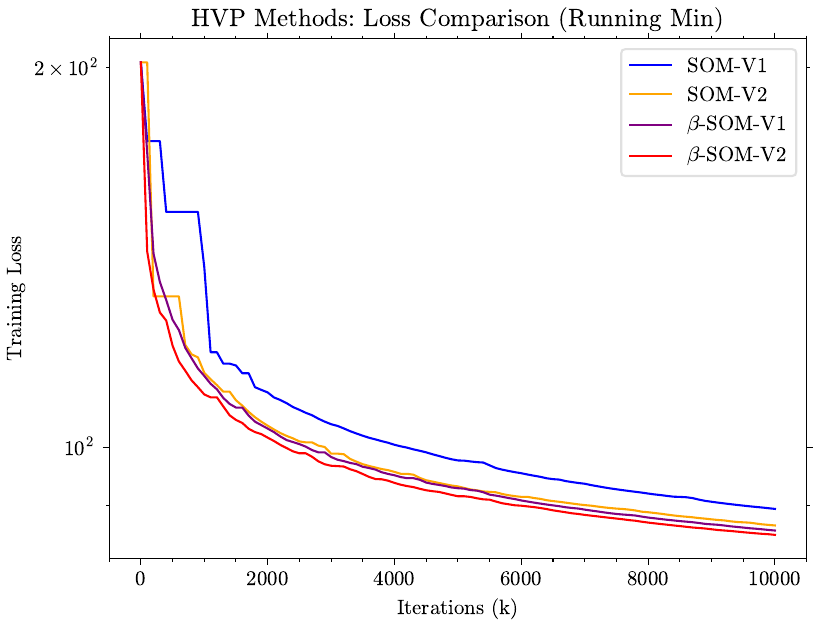}
    \includegraphics[width=0.47\textwidth]{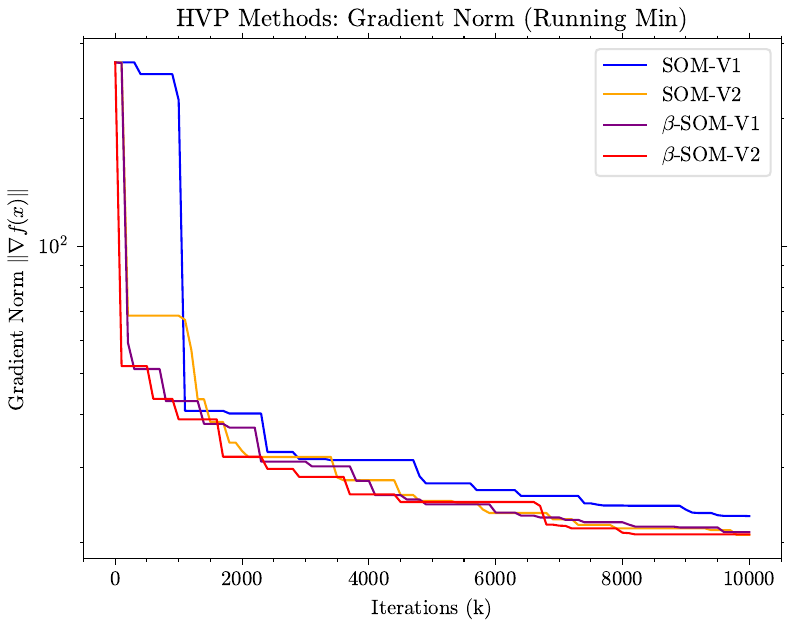}
    \caption{Performance of $\beta$-SOM-V2 (Variant 2 of second-order momentum with any $\beta_k$ scaling factor) when we vary the mini-batch size for training the MLP model on the \texttt{splice} dataset.}
    \label{fig:batch_comparison_main_text}
\end{figure}

\subsection{LSTM Training}
We benchmarked the algorithms for solving 
 a word-level language modeling task using a two-layer LSTM network. The model consists of an embedding layer with $200$ dimensions, followed by two stacked LSTM layers, each with $200$ hidden units. 
A final linear decoder layer produces logits over the vocabulary for each input token.
The experiments were conducted on the Penn Treebank (PTB) dataset, a standard dataset for evaluating language models. The dataset is split into training, validation, and test sets, with approximately $929,000$ tokens in the training set. The vocabulary is constructed from the unique words in the training data, with a total of $10,000$ tokens, including an end-of-sentence token \texttt{<eos>}. Tokenization is performed by splitting raw text on whitespace, converting each word into an index from the vocabulary. The model is trained using sequential batches, employing Truncated Backpropagation Through Time (BPTT) with a sequence length of $35$.
For training, we used a minibatch size of $20$ and the standard Cross-Entropy loss.

{\color{black}From~\Cref{fig:lstm_lmo_comparison_main_text}}, we  observe the superior performance of LMO-based methods using Variant 2 of second-order momentum and the scaling $\beta$ factor when we set the Euclidean norm (left) and the $\ell_{\infty}$-norm (right) in the LMO update.

\begin{figure}[h!]
    \centering
    \begin{subfigure}[b]{0.47\textwidth}
        \includegraphics[width=\textwidth]{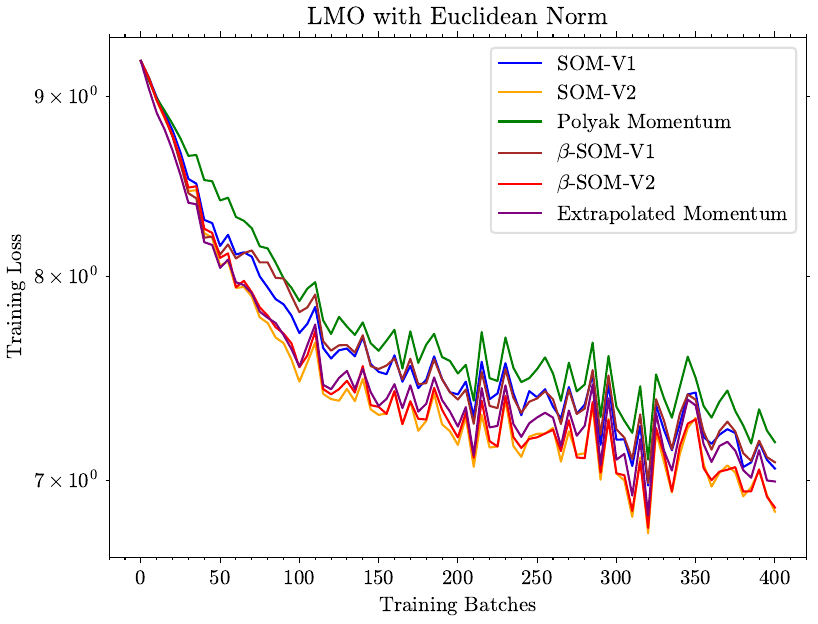}
        \caption{Training loss vs. training batches of LMO-based methods with respect to the Euclidean norm $\normeu{\cdot}$. Other momentum variants attain superier performance to Polyak Momentum.}
        \label{fig:lstm_lmo_loss_main_text}
    \end{subfigure}
    \hfill
    \begin{subfigure}[b]{0.47\textwidth}
        \includegraphics[width=\textwidth]{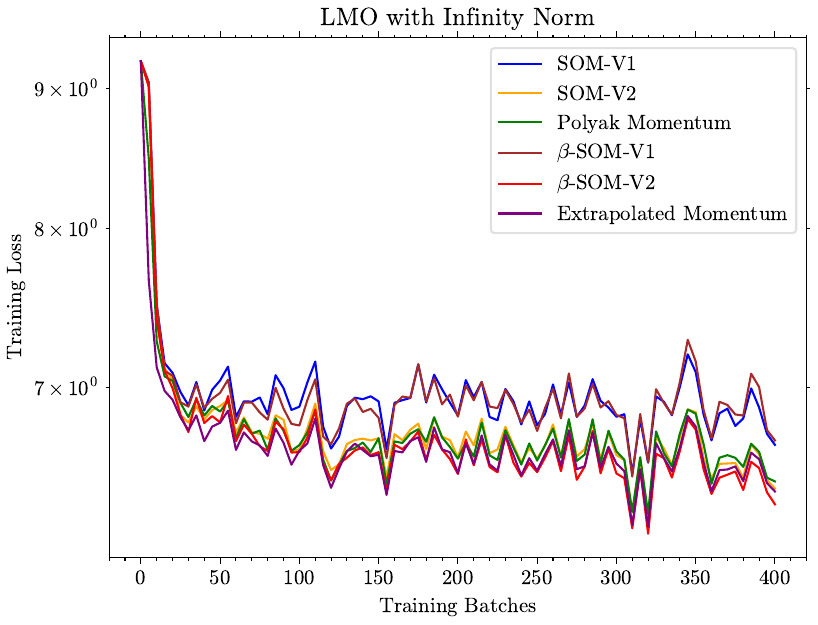}
        \caption{Training loss vs. training batches of LMO-based methods with respect to the $\ell_\infty$-norm $\norm{\cdot}_{\infty}$. Variant 2 of second-order momentum and extrapolated momentum outperforms Polyak momentum. 
        }
        \label{fig:lstm_linf_loss_main_text}
    \end{subfigure}
    \caption{Comparison of LMO-based methods using momentum variants for the word-level language modeling task using the LSTM network on the Penn Treebank (PTB) dataset.}
    \label{fig:lstm_lmo_comparison_main_text}
\end{figure}

\section{Conclusion}

We have proposed LMO-based methods that incorporate two variants of second-order momentum for minimizing nonconvex functions under relaxed smoothness and arbitrary norm settings. The proposed methods achieve a convergence rate of $\mathcal{O}(1/K^{1/3})$ in terms of the expected gradient norm, improving upon the $\mathcal{O}(1/K^{1/4})$ rate of existing LMO-based methods utilizing Polyak momentum in generalized smoothness settings. Notably, our theoretical results match the best-known rates for second-order momentum methods under standard smoothness and Euclidean norm settings.
Furthermore, we provide convergence guarantees for LMO-based methods with MVR and extrapolated momentum under generalized mean-squared smoothness and generalized second-order smoothness assumptions, respectively. 
Finally, our numerical experiments confirm the advantages of second-order momentum, consistently demonstrating superior convergence compared to Polyak and extrapolated momentum. Moreover, incorporating the adaptive scaling factor from MARS yields further improvements, resulting in faster convergence and higher precision.

\subsection*{Acknowledgments}
The research reported in this publication was supported by funding from King Abdullah University of Science and Technology (KAUST): i) KAUST Baseline Research Scheme, ii) CRG Grant ORFS-CRG12-2024-6460, and iii) Center of Excellence for Generative AI, under award number 5940.

\bibliography{iclr2026_conference}

\begin{thebibliography}{53}
\providecommand{\natexlab}[1]{#1}
\providecommand{\url}[1]{\texttt{#1}}
\expandafter\ifx\csname urlstyle\endcsname\relax
  \providecommand{\doi}[1]{doi: #1}\else
  \providecommand{\doi}{doi: \begingroup \urlstyle{rm}\Url}\fi

\bibitem[Arjevani et~al.(2023)Arjevani, Carmon, Duchi, Foster, Srebro, and
  Woodworth]{arjevani2023lower}
Yossi Arjevani, Yair Carmon, John~C Duchi, Dylan~J Foster, Nathan Srebro, and
  Blake Woodworth.
\newblock Lower bounds for non-convex stochastic optimization.
\newblock \emph{Mathematical Programming}, 199\penalty0 (1):\penalty0 165--214,
  2023.

\bibitem[Arnold et~al.(2019)Arnold, Manzagol, Babanezhad~Harikandeh,
  Mitliagkas, and Le~Roux]{arnold2019reducing}
S{\'e}bastien Arnold, Pierre-Antoine Manzagol, Reza Babanezhad~Harikandeh,
  Ioannis Mitliagkas, and Nicolas Le~Roux.
\newblock Reducing the variance in online optimization by transporting past
  gradients.
\newblock \emph{Advances in Neural Information Processing Systems}, 32, 2019.

\bibitem[Carmon et~al.(2018)Carmon, Duchi, Hinder, and
  Sidford]{carmon2018accelerated}
Yair Carmon, John~C Duchi, Oliver Hinder, and Aaron Sidford.
\newblock Accelerated methods for nonconvex optimization.
\newblock \emph{SIAM Journal on Optimization}, 28\penalty0 (2):\penalty0
  1751--1772, 2018.

\bibitem[Chang \& Lin(2011)Chang and Lin]{chang2011libsvm}
Chih-Chung Chang and Chih-Jen Lin.
\newblock {LIBSVM:} a library for support vector machines.
\newblock \emph{ACM transactions on intelligent systems and technology (TIST)},
  2\penalty0 (3):\penalty0 1--27, 2011.

\bibitem[Chen et~al.(2023)Chen, Zhou, Liang, and Lu]{chen2023generalized}
Ziyi Chen, Yi~Zhou, Yingbin Liang, and Zhaosong Lu.
\newblock Generalized-smooth nonconvex optimization is as efficient as smooth
  nonconvex optimization.
\newblock In \emph{International Conference on Machine Learning}, pp.\
  5396--5427. PMLR, 2023.

\bibitem[Chezhegov et~al.(2025)Chezhegov, Beznosikov, Horv{\'a}th, and
  Gorbunov]{chezhegov2025convergence}
Savelii Chezhegov, Aleksandr Beznosikov, Samuel Horv{\'a}th, and Eduard
  Gorbunov.
\newblock Convergence of clipped-{SGD} for convex $(l\_0, l\_1) $-smooth
  optimization with heavy-tailed noise.
\newblock \emph{arXiv preprint arXiv:2505.20817}, 2025.

\bibitem[Crawshaw et~al.(2022)Crawshaw, Liu, Orabona, Zhang, and
  Zhuang]{crawshaw2022robustness}
Michael Crawshaw, Mingrui Liu, Francesco Orabona, Wei Zhang, and Zhenxun
  Zhuang.
\newblock Robustness to unbounded smoothness of generalized signsgd.
\newblock \emph{Advances in neural information processing systems},
  35:\penalty0 9955--9968, 2022.

\bibitem[Cutkosky \& Mehta(2020)Cutkosky and Mehta]{cutkosky2020momentum}
Ashok Cutkosky and Harsh Mehta.
\newblock Momentum improves normalized {SGD}.
\newblock In \emph{International conference on machine learning}, pp.\
  2260--2268. PMLR, 2020.

\bibitem[Cutkosky \& Mehta(2021)Cutkosky and Mehta]{cutkosky2021high}
Ashok Cutkosky and Harsh Mehta.
\newblock High-probability bounds for non-convex stochastic optimization with
  heavy tails.
\newblock \emph{Advances in Neural Information Processing Systems},
  34:\penalty0 4883--4895, 2021.

\bibitem[Cutkosky \& Orabona(2019)Cutkosky and Orabona]{cutkosky2019momentum}
Ashok Cutkosky and Francesco Orabona.
\newblock Momentum-based variance reduction in non-convex {SGD}.
\newblock \emph{Advances in neural information processing systems}, 32, 2019.

\bibitem[Fatkhullin et~al.(2023)Fatkhullin, Tyurin, and
  Richt{\'a}rik]{fatkhullin2023momentum}
Ilyas Fatkhullin, Alexander Tyurin, and Peter Richt{\'a}rik.
\newblock Momentum provably improves error feedback!
\newblock \emph{Advances in Neural Information Processing Systems},
  36:\penalty0 76444--76495, 2023.

\bibitem[Faw et~al.(2023)Faw, Rout, Caramanis, and Shakkottai]{faw2023beyond}
Matthew Faw, Litu Rout, Constantine Caramanis, and Sanjay Shakkottai.
\newblock Beyond uniform smoothness: A stopped analysis of adaptive sgd.
\newblock In \emph{The Thirty Sixth Annual Conference on Learning Theory}, pp.\
   89--160. PMLR, 2023.

\bibitem[Ghadimi \& Lan(2013)Ghadimi and Lan]{ghadimi2013stochastic}
Saeed Ghadimi and Guanghui Lan.
\newblock Stochastic first-and zeroth-order methods for nonconvex stochastic
  programming.
\newblock \emph{SIAM journal on optimization}, 23\penalty0 (4):\penalty0
  2341--2368, 2013.

\bibitem[Gorbunov et~al.(2020)Gorbunov, Kovalev, Makarenko, and
  Richt{\'a}rik]{gorbunov2020linearly}
Eduard Gorbunov, Dmitry Kovalev, Dmitry Makarenko, and Peter Richt{\'a}rik.
\newblock Linearly converging error compensated {SGD}.
\newblock \emph{Advances in Neural Information Processing Systems},
  33:\penalty0 20889--20900, 2020.

\bibitem[Gorbunov et~al.(2024)Gorbunov, Tupitsa, Choudhury, Aliev,
  Richt{\'a}rik, Horv{\'a}th, and Tak{\'a}{\v{c}}]{gorbunov2024methods}
Eduard Gorbunov, Nazarii Tupitsa, Sayantan Choudhury, Alen Aliev, Peter
  Richt{\'a}rik, Samuel Horv{\'a}th, and Martin Tak{\'a}{\v{c}}.
\newblock Methods for convex $(l\_0, l\_1) $-smooth optimization: Clipping,
  acceleration, and adaptivity.
\newblock \emph{arXiv preprint arXiv:2409.14989}, 2024.

\bibitem[Gupta et~al.(2018)Gupta, Koren, and Singer]{gupta2018shampoo}
Vineet Gupta, Tomer Koren, and Yoram Singer.
\newblock Shampoo: Preconditioned stochastic tensor optimization.
\newblock In \emph{International Conference on Machine Learning}, pp.\
  1842--1850. PMLR, 2018.

\bibitem[Huang et~al.(2025)Huang, Luo, and Chen]{huang2025limuon}
Feihu Huang, Yuning Luo, and Songcan Chen.
\newblock Limuon: Light and fast muon optimizer for large models.
\newblock \emph{arXiv preprint arXiv:2509.14562}, 2025.

\bibitem[H{\"u}bler et~al.(2024{\natexlab{a}})H{\"u}bler, Fatkhullin, and
  He]{hubler2024gradient}
Florian H{\"u}bler, Ilyas Fatkhullin, and Niao He.
\newblock From gradient clipping to normalization for heavy tailed sgd.
\newblock \emph{arXiv preprint arXiv:2410.13849}, 2024{\natexlab{a}}.

\bibitem[H{\"u}bler et~al.(2024{\natexlab{b}})H{\"u}bler, Yang, Li, and
  He]{hubler2024parameter}
Florian H{\"u}bler, Junchi Yang, Xiang Li, and Niao He.
\newblock Parameter-agnostic optimization under relaxed smoothness.
\newblock In \emph{International Conference on Artificial Intelligence and
  Statistics}, pp.\  4861--4869. PMLR, 2024{\natexlab{b}}.

\bibitem[Jiang et~al.(2024)Jiang, Derezinski, and
  Mokhtari]{jiang2024stochastic}
Ruichen Jiang, Michal Derezinski, and Aryan Mokhtari.
\newblock Stochastic newton proximal extragradient method.
\newblock \emph{Advances in Neural Information Processing Systems},
  37:\penalty0 90818--90852, 2024.

\bibitem[Jiang et~al.(2025)Jiang, Yu, Yang, Yang, and Zhang]{jiang2025improved}
Wei Jiang, Dingzhi Yu, Sifan Yang, Wenhao Yang, and Lijun Zhang.
\newblock Improved analysis for sign-based methods with momentum updates.
\newblock \emph{arXiv preprint arXiv:2507.12091}, 2025.

\bibitem[Jordan et~al.(2024)Jordan, Jin, Boza, You, Cesista, Newhouse, and
  Bernstein]{jordan2024muon}
Keller Jordan, Yuchen Jin, Vlado Boza, Jiacheng You, Franz Cesista, Laker
  Newhouse, and Jeremy Bernstein.
\newblock Muon: An optimizer for hidden layers in neural networks.
\newblock \emph{Cited on}, pp.\ ~10, 2024.

\bibitem[Khirirat et~al.(2024)Khirirat, Sadiev, Riabinin, Gorbunov, and
  Richt{\'a}rik]{khirirat2024error}
Sarit Khirirat, Abdurakhmon Sadiev, Artem Riabinin, Eduard Gorbunov, and Peter
  Richt{\'a}rik.
\newblock Error feedback under $(l\_0, l\_1) $-smoothness: Normalization and
  momentum.
\newblock \emph{arXiv preprint arXiv:2410.16871}, 2024.

\bibitem[Koloskova et~al.(2023)Koloskova, Hendrikx, and
  Stich]{koloskova2023revisiting}
Anastasia Koloskova, Hadrien Hendrikx, and Sebastian~U Stich.
\newblock Revisiting gradient clipping: Stochastic bias and tight convergence
  guarantees.
\newblock In \emph{International Conference on Machine Learning}, pp.\
  17343--17363. PMLR, 2023.

\bibitem[Kovalev(2025)]{kovalev2025understanding}
Dmitry Kovalev.
\newblock Understanding gradient orthogonalization for deep learning via
  non-euclidean trust-region optimization.
\newblock \emph{arXiv preprint arXiv:2503.12645}, 2025.

\bibitem[Kovalev \& Borodich(2025)Kovalev and Borodich]{kovalev2025non}
Dmitry Kovalev and Ekaterina Borodich.
\newblock Non-euclidean sgd for structured optimization: Unified analysis and
  improved rates.
\newblock \emph{arXiv preprint arXiv:2511.11466}, 2025.

\bibitem[Li et~al.(2023)Li, Qian, Tian, Rakhlin, and Jadbabaie]{li2023convex}
Haochuan Li, Jian Qian, Yi~Tian, Alexander Rakhlin, and Ali Jadbabaie.
\newblock Convex and non-convex optimization under generalized smoothness.
\newblock \emph{Advances in Neural Information Processing Systems},
  36:\penalty0 40238--40271, 2023.

\bibitem[Li \& Hong(2025)Li and Hong]{li2025note}
Jiaxiang Li and Mingyi Hong.
\newblock A note on the convergence of muon.
\newblock \emph{arXiv preprint arXiv:2502.02900}, 2025.

\bibitem[Liu et~al.(2025{\natexlab{a}})Liu, Su, Yao, Jiang, Lai, Du, Qin, Xu,
  Lu, Yan, et~al.]{liu2025muon}
Jingyuan Liu, Jianlin Su, Xingcheng Yao, Zhejun Jiang, Guokun Lai, Yulun Du,
  Yidao Qin, Weixin Xu, Enzhe Lu, Junjie Yan, et~al.
\newblock Muon is scalable for llm training.
\newblock \emph{arXiv preprint arXiv:2502.16982}, 2025{\natexlab{a}}.

\bibitem[Liu et~al.(2020)Liu, Gao, and Yin]{liu2020improved}
Yanli Liu, Yuan Gao, and Wotao Yin.
\newblock An improved analysis of stochastic gradient descent with momentum.
\newblock \emph{Advances in Neural Information Processing Systems},
  33:\penalty0 18261--18271, 2020.

\bibitem[Liu et~al.(2025{\natexlab{b}})Liu, Ge, Pan, Kang, and
  Zhang]{liu2025theoretical}
Yuxing Liu, Yuze Ge, Rui Pan, An~Kang, and Tong Zhang.
\newblock Theoretical analysis on how learning rate warmup accelerates
  convergence.
\newblock \emph{arXiv preprint arXiv:2509.07972}, 2025{\natexlab{b}}.

\bibitem[Loshchilov \& Hutter(2017)Loshchilov and
  Hutter]{loshchilov2017decoupled}
Ilya Loshchilov and Frank Hutter.
\newblock Decoupled weight decay regularization.
\newblock \emph{arXiv preprint arXiv:1711.05101}, 2017.

\bibitem[Nesterov \& Polyak(2006)Nesterov and Polyak]{nesterov2006cubic}
Yurii Nesterov and Boris~T Polyak.
\newblock Cubic regularization of newton method and its global performance.
\newblock \emph{Mathematical programming}, 108\penalty0 (1):\penalty0 177--205,
  2006.

\bibitem[Paszke et~al.(2019)Paszke, Gross, Massa, Lerer, Bradbury, Chanan,
  Killeen, Lin, Gimelshein, Antiga, et~al.]{paszke2019pytorch}
Adam Paszke, Sam Gross, Francisco Massa, Adam Lerer, James Bradbury, Gregory
  Chanan, Trevor Killeen, Zeming Lin, Natalia Gimelshein, Luca Antiga, et~al.
\newblock Pytorch: An imperative style, high-performance deep learning library.
\newblock \emph{Advances in neural information processing systems}, 32, 2019.

\bibitem[Pearlmutter(1994)]{pearlmutter1994fast}
Barak~A Pearlmutter.
\newblock Fast exact multiplication by the hessian.
\newblock \emph{Neural computation}, 6\penalty0 (1):\penalty0 147--160, 1994.

\bibitem[Pethick et~al.(2025{\natexlab{a}})Pethick, Xie, Antonakopoulos, Zhu,
  Silveti-Falls, and Cevher]{pethick2025training}
Thomas Pethick, Wanyun Xie, Kimon Antonakopoulos, Zhenyu Zhu, Antonio
  Silveti-Falls, and Volkan Cevher.
\newblock Training deep learning models with norm-constrained {LMOs}.
\newblock \emph{arXiv preprint arXiv:2502.07529}, 2025{\natexlab{a}}.

\bibitem[Pethick et~al.(2025{\natexlab{b}})Pethick, Xie, Erdogan,
  Antonakopoulos, Silveti-Falls, and Cevher]{pethick2025generalized}
Thomas Pethick, Wanyun Xie, Mete Erdogan, Kimon Antonakopoulos, Tony
  Silveti-Falls, and Volkan Cevher.
\newblock Generalized gradient norm clipping \& non-euclidean $(l\_0, l\_1)
  $-smoothness.
\newblock \emph{arXiv preprint arXiv:2506.01913}, 2025{\natexlab{b}}.

\bibitem[Polyak(1964)]{polyak1964some}
Boris~T Polyak.
\newblock Some methods of speeding up the convergence of iteration methods.
\newblock \emph{Ussr computational mathematics and mathematical physics},
  4\penalty0 (5):\penalty0 1--17, 1964.

\bibitem[Riabinin et~al.(2025)Riabinin, Shulgin, Gruntkowska, and
  Richt{\'a}rik]{riabinin2025gluon}
Artem Riabinin, Egor Shulgin, Kaja Gruntkowska, and Peter Richt{\'a}rik.
\newblock Gluon: Making muon \& scion great again!(bridging theory and practice
  of lmo-based optimizers for llms).
\newblock \emph{arXiv preprint arXiv:2505.13416}, 2025.

\bibitem[Sadiev et~al.(2025)Sadiev, Richt{\'a}rik, and
  Fatkhullin]{sadiev2025second}
Abdurakhmon Sadiev, Peter Richt{\'a}rik, and Ilyas Fatkhullin.
\newblock Second-order optimization under heavy-tailed noise: Hessian clipping
  and sample complexity limits.
\newblock \emph{arXiv preprint arXiv:2510.10690}, 2025.

\bibitem[Salehkaleybar et~al.(2022)Salehkaleybar, Khorasani, Kiyavash, He, and
  Thiran]{salehkaleybar2022momentum}
Saber Salehkaleybar, Sadegh Khorasani, Negar Kiyavash, Niao He, and Patrick
  Thiran.
\newblock Momentum-based policy gradient with second-order information.
\newblock \emph{arXiv preprint arXiv:2205.08253}, 2022.

\bibitem[Tran \& Cutkosky(2022)Tran and Cutkosky]{tran2022better}
Hoang Tran and Ashok Cutkosky.
\newblock Better {SGD} using second-order momentum.
\newblock \emph{Advances in Neural Information Processing Systems},
  35:\penalty0 3530--3541, 2022.

\bibitem[Vankov et~al.(2024)Vankov, Rodomanov, Nedich, Sankar, and
  Stich]{vankov2024optimizing}
Daniil Vankov, Anton Rodomanov, Angelia Nedich, Lalitha Sankar, and Sebastian~U
  Stich.
\newblock Optimizing $(l\_0, l\_1) $-smooth functions by gradient methods.
\newblock \emph{arXiv preprint arXiv:2410.10800}, 2024.

\bibitem[Vyas et~al.(2024)Vyas, Morwani, Zhao, Kwun, Shapira, Brandfonbrener,
  Janson, and Kakade]{vyas2024soap}
Nikhil Vyas, Depen Morwani, Rosie Zhao, Mujin Kwun, Itai Shapira, David
  Brandfonbrener, Lucas Janson, and Sham Kakade.
\newblock Soap: Improving and stabilizing shampoo using adam.
\newblock \emph{arXiv preprint arXiv:2409.11321}, 2024.

\bibitem[Wang et~al.(2023)Wang, Zhang, Ma, and Chen]{wang2023convergence}
Bohan Wang, Huishuai Zhang, Zhiming Ma, and Wei Chen.
\newblock Convergence of adagrad for non-convex objectives: Simple proofs and
  relaxed assumptions.
\newblock In \emph{The Thirty Sixth Annual Conference on Learning Theory}, pp.\
   161--190. PMLR, 2023.

\bibitem[Xie et~al.(2024)Xie, Li, Zhang, Deng, Ge, and Ye]{xie2024trust}
Chenghan Xie, Chenxi Li, Chuwen Zhang, Qi~Deng, Dongdong Ge, and Yinyu Ye.
\newblock Trust region methods for nonconvex stochastic optimization beyond
  lipschitz smoothness.
\newblock In \emph{Proceedings of the AAAI Conference on Artificial
  Intelligence}, volume~38, pp.\  16049--16057, 2024.

\bibitem[Yan et~al.(2018)Yan, Yang, Li, Lin, and Yang]{yan2018unified}
Yan Yan, Tianbao Yang, Zhe Li, Qihang Lin, and Yi~Yang.
\newblock A unified analysis of stochastic momentum methods for deep learning.
\newblock \emph{arXiv preprint arXiv:1808.10396}, 2018.

\bibitem[Yu et~al.(2019)Yu, Jin, and Yang]{yu2019linear}
Hao Yu, Rong Jin, and Sen Yang.
\newblock On the linear speedup analysis of communication efficient momentum
  {SGD} for distributed non-convex optimization.
\newblock In \emph{International Conference on Machine Learning}, pp.\
  7184--7193. PMLR, 2019.

\bibitem[Yuan et~al.(2024)Yuan, Liu, Wu, Zhou, and Gu]{yuan2024mars}
Huizhuo Yuan, Yifeng Liu, Shuang Wu, Xun Zhou, and Quanquan Gu.
\newblock Mars: Unleashing the power of variance reduction for training large
  models.
\newblock \emph{arXiv preprint arXiv:2411.10438}, 2024.

\bibitem[Zhang et~al.(2020{\natexlab{a}})Zhang, Jin, Fang, and
  Wang]{zhang2020improved}
Bohang Zhang, Jikai Jin, Cong Fang, and Liwei Wang.
\newblock Improved analysis of clipping algorithms for non-convex optimization.
\newblock \emph{Advances in Neural Information Processing Systems},
  33:\penalty0 15511--15521, 2020{\natexlab{a}}.

\bibitem[Zhang et~al.(2019)Zhang, He, Sra, and Jadbabaie]{zhang2019gradient}
Jingzhao Zhang, Tianxing He, Suvrit Sra, and Ali Jadbabaie.
\newblock Why gradient clipping accelerates training: A theoretical
  justification for adaptivity.
\newblock \emph{arXiv preprint arXiv:1905.11881}, 2019.

\bibitem[Zhang et~al.(2020{\natexlab{b}})Zhang, Shen, Mokhtari, Hassani, and
  Karbasi]{zhang2020one}
Mingrui Zhang, Zebang Shen, Aryan Mokhtari, Hamed Hassani, and Amin Karbasi.
\newblock One sample stochastic frank-wolfe.
\newblock In \emph{International Conference on Artificial Intelligence and
  Statistics}, pp.\  4012--4023. PMLR, 2020{\natexlab{b}}.

\bibitem[Zhao et~al.(2021)Zhao, Xie, and Li]{zhao2021convergence}
Shen-Yi Zhao, Yin-Peng Xie, and Wu-Jun Li.
\newblock On the convergence and improvement of stochastic normalized gradient
  descent.
\newblock \emph{Science China Information Sciences}, 64\penalty0 (3):\penalty0
  132103, 2021.

\end{thebibliography}
\bibliographystyle{iclr2026_conference}

\newpage
\tableofcontents

\newpage 
\appendix

\section{Notations}

For random variables $u,v$, we use $\Exp{u}$ for the expectation of $u$, and $\ExpSub{v}{u}$ for the expectation of $u$ with respect to $v$.
For functions $f, g : \R^d \to \R$, we use $\nabla f$ and $\nabla^2 f$ to denote the gradient and Hessian of $f$, respectively. We write $f(x) = \cO(g(x))$ to indicate that there exists a constant $C > 0$ and a value $x_0 \in \R$ such that $f(x) \leq C \cdot g(x)$ for all $x \geq x_0$.
For vectors $x,y\in\R^d$, $\inp{x}{y} = \sum_{i=1}^d x_i y_i$, while $\norm{x}$,
 $\normstar{x}$, $\normeu{x}$, and $\norm{x}_{\infty}$ denote its arbitrary norm,  associated dual norm,  Euclidean norm, and  $\ell_{\infty}$-norm, respectively. 

\section{Useful Lemma from Relaxed Smoothness}

From Assumptions~\ref{assum:Relaxed_Lipschitz_Grad} and~\ref{assum:Relaxed_Lipschitz_Hessian}, we obtain the following lemma that is useful for our analysis. 
\begin{lemma}[Proposition 3.2 and Theorem 1 from \citep{chen2023generalized}]
\label{lemma:Relaxed_Lipschitz_Grad}
    Let $f$ satisfy Assumption~\ref{assum:Relaxed_Lipschitz_Grad}. Then, for all $x,y \in \R^d$,
    \begin{eqnarray*}
        \normstar{\nabla f(x) -\nabla f(y)} &\leq& (L_0 + L_1\normstar{\nabla f(y)})\exp\left(L_1\norm{x-y}\right)\norm{x-y};\\
        D_f(x,y) &\leq& \frac{1}{2}(L_0 + L_1\normstar{\nabla f(y)})\exp\left(L_1\norm{x-y}\right)\norm{x-y}^2,
    \end{eqnarray*}
    where $D_f(x,y) \eqdef f(x) - f(y) - \la\nabla f(y), x-y\ra$ denotes the Bregman divergence.
    Furthermore, if $f$ is twice-differentiable, then for any $x \in \R^d$,
    \begin{eqnarray*}
        \norm{\nabla^2 f(x)}_{\text{op}} & \eqdef & \sup_{u\neq 0}\frac{\normstar{\nabla^2f(x)u}}{\norm{u}} \leq L_0 + L_1 \normstar{\nabla f(x)}.
    \end{eqnarray*}
\end{lemma}
\begin{proof}
    The proof follows the same lines as Proposition 3.2 and Theorem 1 in \citep{chen2023generalized}, adapted to the setting of arbitrary norms. 
\end{proof}
\begin{lemma}\label{lemma:Relaxed_Grad_and_Hessian}
Let $f$ satisfy  Assumptions~\ref{assum:Relaxed_Lipschitz_Grad} and~\ref{assum:Relaxed_Lipschitz_Hessian}. Then, for all $x,y\in\R^d$, 
\begin{eqnarray*}
    \normstar{Z_f(x,y)} 
    &\leq & \frac{1}{2}(M_0 + M_1 \normstar{\nabla f(y)}) \norm{x-y}^2 \\
    && + \frac{1}{3} M_1 (L_0+L_1 \normstar{\nabla f(y)}) \exp(L_1 \norm{x-y}) \norm{x-y}^3,
\end{eqnarray*}
where we denote $Z_f(x,y) \eqdef \nabla f(x) - \nabla f(y) - \nabla^2 f(y)(x-y)$.
\end{lemma}
\begin{proof}
Since $f$ is twice differentiable, 
\begin{eqnarray*}
  Z_f(x,y) 
  & = & \int_{\tau=0}^1 \nabla^2 f(x + \tau(y-x)) (y-x) d\tau  - \nabla^2 f(y)(x-y) \\
  & = & \int_{\tau=0}^1 [\nabla^2 f(x + \tau(y-x)) (y-x)  - \nabla^2 f(y)(x-y)] d\tau.  
\end{eqnarray*}
Therefore,
\begin{eqnarray*}
    \normstar{Z_f(x,y) } 
    &\leq& \int_{\tau=0}^1 \normstar{\nabla^2 f(x + \tau(y-x)) (y-x)  - \nabla^2 f(y)(x-y)} d\tau. 
\end{eqnarray*}

Next, from~\Cref{assum:Relaxed_Lipschitz_Hessian}, 
\begin{eqnarray*}
    \normstar{Z_f(x,y) } 
    &\leq& \int_{\tau=0}^1 (M_0 + M_1 \underset{\theta \in [0,1]}{\sup} \normstar{ \nabla f(\theta \hat x + (1-\theta)y)}) \norm{\hat x-y} d\tau  \\
    &\leq& \frac{M_0 \norm{x-y}^2}{2}\\
    && + M_1 \norm{x-y}^2  \int_{\tau=0}^1 \tau  \underset{\theta \in [0,1]}{\sup} \normstar{ \nabla f(\theta \hat x + (1-\theta)y) - {\color{black} \nabla f(y)} + {\color{black}\nabla f(y)} }) d\tau \\
    &\leq& \frac{M_0 \norm{x-y}^2}{2} + M_1 \norm{x-y}^2  \int_{\tau=0}^1 \tau  \underset{\theta \in [0,1]}{\sup} \normstar{ \nabla f(\theta \hat x + (1-\theta)y) - {\color{black} \nabla f(y)} }) d\tau \\
    && + M_1 \norm{x-y}^2  \int_{\tau=0}^1 \tau  \underset{\theta \in [0,1]}{\sup} \normstar{ {\color{black} \nabla f(y)} } d\tau \\
    &=& \frac{M_0+M_1\norm{{\color{black}\nabla f(y)}}}{2}\norm{x-y}^2 \\
    &&+ M_1 \norm{x-y}^2  \int_{\tau=0}^1 \tau  \underset{\theta \in [0,1]}{\sup} \normstar{ \nabla f(\theta \hat x + (1-\theta)y) - {\color{black} \nabla f(y)} }) d\tau,
\end{eqnarray*}
where $\hat x = x + \tau(y-x)$.

Next, from \Cref{assum:Relaxed_Lipschitz_Grad},
\begin{eqnarray*}
\normstar{ \nabla f(\theta \hat x + (1-\theta)y) - {\color{black} \nabla f(y)} } 
& \leq & (L_0+L_1 \normstar{\nabla f(y)}) \exp( \tau\theta L_1 \norm{x-y}) \tau\theta \norm{x-y} \\
& \overset{\theta \leq 1}{\leq} & (L_0+L_1 \normstar{\nabla f(y)}) \exp( \tau L_1 \norm{x-y}) \tau \norm{x-y}.
\end{eqnarray*}

Therefore, by plugging this result into the upper-bound for $\normstar{\nabla f(x) - \nabla f(y) - \nabla^2 f(y)(x-y) }$, we complete the proof. 

\end{proof}

\newpage
\section{Proof of~\Cref{lemma:V1_relaxed_smoothness}  }

We prove~\Cref{lemma:V1_relaxed_smoothness} in the following steps.

\paragraph{Step 1) Proving the descent inequality.}

We follow the proof arguments from Lemma D.1. of \citet{pethick2025training}.
By~\Cref{lemma:Relaxed_Lipschitz_Grad}, and by the fact that $ x_{k+1}= x_k + \text{lmo}(m_k)$, 
\begin{eqnarray*}
    f(x_{k+1}) 
    & \leq & f(x_k) + \inp{\nabla f(x_k)}{\text{lmo}(m_k)} \\
    && + \frac{L_0 + L_1 \normstar{\nabla f(x_k)}}{2} \exp(L_1 \norm{x_{k+1}-x_k})\norm{\text{lmo}(m_k)}^2 \\
    & = &  f(x_k) + \inp{\nabla f(x_k) - m_k}{\text{lmo}(m_k)} + \inp{m_k}{\text{lmo}(m_k)} \\
    && + \frac{L_0 + L_1 \normstar{\nabla f(x_k)}}{2} \exp(L_1 \norm{x_{k+1}-x_k})\norm{\text{lmo}(m_k)}^2.
\end{eqnarray*}

By Cauchy-Schwartz inequality, 
\begin{eqnarray*}
    f(x_{k+1}) 
& \leq & f(x_k) + \normstar{\nabla f(x_k) - m_k}\norm{\text{lmo}(m_k)} + \inp{m_k}{\text{lmo}(m_k)} \\
    && + \frac{L_0 + L_1 \normstar{\nabla f(x_k)}}{2} \exp(L_1 \norm{x_{k+1}-x_k})\norm{\text{lmo}(m_k)}^2.
\end{eqnarray*}

By the fact that $\norm{x_{k+1} - x_k} = \norm{\text{lmo}(m_k)} \leq \eta_k$, 
\begin{eqnarray*}
    f(x_{k+1}) 
     \leq  f(x_k) + \eta_k\normstar{\nabla f(x_k) - m_k}  + \inp{m_k}{\text{lmo}(m_k)} + \frac{L_0 + L_1 \normstar{\nabla f(x_k)}}{2} \exp(L_1 \eta_k) \eta_k^2.
\end{eqnarray*}

Since for all $u \in \mathcal{X}$
\begin{eqnarray*}
    \normstar{u} = \underset{v: \norm{v} \leq 1}{\max} \inp{u}{v} = \max_{v: \norm{v} \leq \eta} \inp{u}{\frac{1}{\eta} v} = - \inp{u}{ \frac{1}{\eta} \text{lmo}(u)},
\end{eqnarray*}
we obtain
\begin{eqnarray*}
    \inp{m_k}{\text{lmo}(m_k)} = \eta_k \inp{m_k}{ \frac{1}{\eta_k} \text{lmo}(m_k)} = - \eta_k \normstar{m_k}.
\end{eqnarray*}
Therefore, 
\begin{eqnarray*}
    f(x_{k+1}) 
  \leq  f(x_k) + \eta_k \normstar{\nabla f(x_k) - m_k} - \eta_k \normstar{m_k}  + \frac{L_0 + L_1 \normstar{\nabla f(x_k)}}{2} \exp(L_1 \eta_k) \eta_k^2.
\end{eqnarray*}

By the triangle inequality, i.e. $\normstar{a} \geq \normstar{b} - \normstar{a-b}$ for $a,b\in\R^d$, 
\begin{eqnarray*}
    f(x_{k+1}) 
  \leq  f(x_k) + 2\eta_k \normstar{\nabla f(x_k) - m_k} - \eta_k \normstar{\nabla f(x_k)}  + \frac{L_0 + L_1 \normstar{\nabla f(x_k)}}{2} \exp(L_1 \eta_k) \eta_k^2.
\end{eqnarray*}

Finally, by summing the inequality over $k=0,1,\ldots,K$ and by re-arranging the terms, 
\begin{eqnarray}
 \sum_{k=0}^K \eta_k  \varphi_k \normstar{\nabla f(x_k)} &\overset{(a)}{=}&\sum_{k=0}^K \eta_k  \left( 1 - \exp(L_1 \eta_k) 
 \frac{L_1\eta_k}{2} \right)\normstar{\nabla f(x_k)} \notag\\ 
 &\leq&  \sum_{k=0}^K f(x_k) - f(x_{k+1}) + 2 \sum_{k=0}^K \eta_k \normstar{\nabla f(x_k)-m_k} + \frac{L_0}{2} \sum_{k=0}^K \exp(L_1 \eta_k) \eta_k^2 \notag\\
 &\overset{(b)}{\leq}&  (f(x_0)-f_{\inf}) + 2 \sum_{k=0}^K \eta_k \normstar{\nabla f(x_k)-m_k} + \frac{L_0}{2} \sum_{k=0}^K \exp(L_1 \eta_k) \eta_k^2 \notag\\
 &\overset{(c)}{=}& \Delta + 2 \sum_{k=0}^K \eta_k \normstar{\nabla f(x_k)-m_k} + \frac{L_0}{2} \sum_{k=0}^K \exp(L_1 \eta_k) \eta_k^2, \label{eq:descent_inequality}
\end{eqnarray}
where we reach  $(a)$ by defining $\varphi_k \eqdef\left( 1 - \exp(L_1 \eta_k) 
 \nicefrac{L_1\eta_k}{2} \right) $, $(b)$ by using the fact that $ {f(x) \geq f_{\inf}}$, and $(c)$ by denoting $\Delta \eqdef f(x_0) - f_{\inf}$.

\paragraph{Step 2) Bounding the error term.}
Next, we bound $\normstar{e_{k+1}}$, where $e_{k+1} \eqdef \nabla f(x_{k+1})-m_{k+1}$. 
From the definition of $e_k$,
\begin{eqnarray}
    e_{k+1} 
    & = & m_{k+1} - \nabla f(x_{k+1}) \notag \\
    & \overset{(a)}{=} & (1-\alpha_k) [m_{k} + \nabla^2 f_{\xi_k}(\hat x_{k+1}) (x_{k+1}-x_{k})  - \nabla f(x_{k+1})]\notag\\
    &&+ \alpha_k[\nabla f_{\xi_k}(x_{k+1})-\nabla f(x_{k+1})] \notag \\
    &\overset{(b)}{=} & (1-\alpha_k) e_{k} + (1-\alpha_k)W_{k+1} +  \alpha_kV_{k+1}, \label{eqn:SHARP_e}
\end{eqnarray}
where we reach $(a)$ by \eqref{eqn:second_order_momentum_update}, and $(b)$ by denoting  $W_{k+1} \eqdef  \nabla^2 f_{\xi_{k+1}}(\hat x_{k+1}) (x_{k+1}-x_{k}) - (\nabla f(x_{k+1}) - \nabla f(x_{k}))$ and $V_{k+1} \eqdef\nabla f_{\xi_{k+1}}(x_{k+1})-\nabla f(x_{k+1})$.

By recursively applying this inequality, 
\begin{eqnarray*}
    e_{k+1} = \prod_{t=0}^k (1-\alpha_t) e_0 +\sum_{t=0}^k \left( \prod_{j=t+1}^k (1-\alpha_j) \right) (1-\alpha_t)  W_{t+1} +  \sum_{t=0}^k \left(\prod_{j=t+1}^k (1-\alpha_j) \right) \alpha_t V_{t+1}.
\end{eqnarray*}

Next, by taking $\normstar{\cdot}$, and by taking the expectation,
\begin{eqnarray*}
    \Exp{\normstar{e_{k+1} } } 
    & \overset{(a)}{\leq} & \normstar{e_0}  \prod_{t=0}^k (1-\alpha_t) +\Exp{\normstar{\sum_{t=0}^k \left( \prod_{j=t+1}^k (1-\alpha_j)\right) (1-\alpha_t)  W_{t+1}}} \\
    && + \Exp{\normstar{\sum_{t=0}^k \left(\prod_{j=t+1}^k (1-\alpha_j) \right) \alpha_t V_{t+1}}} \\
    & \overset{(b)}{\leq}& \normstar{e_0}  \prod_{t=0}^k(1-\alpha_t) + \sqrt{\Exp{\normstar{\sum_{t=0}^k \left( \prod_{j=t+1}^k (1-\alpha_j)\right) (1-\alpha_t)  W_{t+1}}^2}} \\
    && + \sqrt{ \Exp{\normstar{\sum_{t=0}^k \left(\prod_{j=t+1}^k (1-\alpha_j) \right) \alpha_t V_{t+1}}^2} } \\
    &\overset{(c)}{\leq}& \normstar{e_0}  \prod_{t=0}^k(1-\alpha_t) + {\color{red}\bar{\rho}}\underbrace{\sqrt{\Exp{\normeu{\sum_{t=0}^k \left( \prod_{j=t+1}^k (1-\alpha_j)\right) (1-\alpha_t)  W_{t+1}}^2}}}_{\eqdef \circledOne} \\
    && + {\color{red}\bar{\rho}}\underbrace{\sqrt{ \Exp{\normeu{\sum_{t=0}^k \left(\prod_{j=t+1}^k (1-\alpha_j) \right) \alpha_t V_{t+1}}^2} }}_{\eqdef \circledTwo },
\end{eqnarray*}
where $(a)$ follows from the triangle inequality, $(b)$ results from \text{Jensen's inequality}, and $(c)$ obtains from \eqref{eqn:norm_diff}.

To bound  $\Exp{\normstar{e_{k+1} } }$, we must bound $\circledOne$ ~and $\circledTwo$. First, we bound $\circledTwo$. By~\Cref{assum:Hessian_boundedvariance_v2}, i.e. $\ExpSub{\xi}{\nabla f_\xi(x)}= \nabla f(x)$ and $\ExpSub{\xi}{\normeu{\nabla f_\xi(x)-\nabla f(x)}^2} \leq \sigma_g^2$, we can prove that 
$\ExpSub{\xi_k}{V_k}=0$, and that 
$   
\ExpSub{\xi_j}{\inp{ V_j}{  V_i}}  = \inp{\ExpSub{\xi_j}{ V_j}}{  V_i} = 0$  for $i,j \in \mathbb{N}$ and $i < j$.
Therefore, 
\begin{eqnarray*}
    \circledTwo & = & \sum_{t=0}^k \left(\prod_{j=t+1}^k (1-\alpha_j)^2 \right) \alpha_t^2 \Exp{\normeu{V_t}^2} \leq \sigma_g^2 \sum_{t=0}^k \left(\prod_{j=t+1}^k (1-\alpha_j)^2 \right) \alpha_t^2.
\end{eqnarray*}
Plugging this result into the main inequality yields
\begin{eqnarray*}
 \Exp{\normstar{e_{k+1} } } & \leq & \normstar{e_0}  \prod_{t=0}^k(1-\alpha_t) + {\color{red}\bar{\rho}}\underbrace{\sqrt{\Exp{\normeu{\sum_{t=0}^k \left( \prod_{j=t+1}^k (1-\alpha_j)\right) (1-\alpha_t)  W_{t+1}}^2}}}_{\eqdef \circledOne} \\
 && + {\color{red}\bar{\rho}} \sigma_g \sqrt{\sum_{t=0}^k \left(\prod_{j=t+1}^k (1-\alpha_j)^2 \right) \alpha_t^2}. 
\end{eqnarray*}

Second, we bound $\circledOne$. By ~\Cref{assum:Hessian_boundedvariance_v2}, i.e. $\ExpSub{\xi}{\nabla^2 f_{\xi}(x)}=\nabla^2 f(x)$, we can prove that
\begin{eqnarray*}
    \ExpSub{\xi_k,b_k}{W_k} 
    &=& \ExpSub{\xi_k,b_k}{ \nabla^2 f_{\xi_k}(b_k x_k + (1-b_k)x_{k-1})(x_k-x_{k-1})  - (\nabla f(x_k) - \nabla f(x_{k-1}))} \\
    & = &  \ExpSub{\xi_k}{ \int_{b=0}^1 \nabla^2 f_{\xi_k}(b x_k + (1-b)x_{k-1})(x_k-x_{k-1})  - (\nabla f(x_k) - \nabla f(x_{k-1}))} \\
    &=& \int_{b=0}^1 \nabla^2 f(b x_k + (1-b)x_{k-1})(x_k-x_{k-1}) db - (\nabla f(x_k) - \nabla f(x_{k-1}))   \\
    &=& 0,
\end{eqnarray*}
and that $ \ExpSub{\xi_j,b_j}{\inp{ W_j}{  W_i}}  = \inp{\ExpSub{\xi_j,b_j}{ W_j}}{  W_i} = 0$ for $i,j\in\mathbb{N}$ and $i<j$. Thus, 
\begin{eqnarray*}
 \Exp{\normstar{e_{k+1} } } & \leq & \normstar{e_0}  \prod_{t=0}^k(1-\alpha_t) + {\color{red}\bar{\rho}}\sqrt{\sum_{t=0}^k \left( \prod_{j=t+1}^k (1-\alpha_j)^2\right) (1-\alpha_t)^2  \Exp{\normeu{ W_{t+1}}^2}} \\
&& + {\color{red}\bar{\rho}} \sigma_g  \sqrt{\sum_{t=0}^k \left(\prod_{j=t+1}^k (1-\alpha_j)^2 \right) \alpha_t^2}. 
\end{eqnarray*}

\paragraph{Step 3) Bounding $\Exp{\normeu{ W_{k+1}}^2}$.} 
Next, we bound $\Exp{\normeu{ W_{k+1}}^2}$. 
By the fact that $W_{k+1} =  \nabla^2 f_{\xi_{k+1}}(\hat x_{k+1}) (x_{k+1}-x_{k}) - (\nabla f(x_{k+1}) - \nabla f(x_{k}))$, and that $\normeu{x+y+z}^2 \leq 3 \normeu{x}^2+3\normeu{y}^2+3\normeu{z}^2$ for $x,y,z\in\R^d$,
\begin{eqnarray*}
    \Exp{\normeu{ W_{k+1}}^2} 
    & \leq & 3 \Exp{\normeu{\nabla f(x_{k+1})-\nabla f(x_{k})}^2} + 3 \Exp{\normeu{\nabla^2 f(\hat x_{k+1})(x_{k+1}-x_{k})}^2} \\
    && + 3\Exp{ \normeu{(\nabla^2 f_{\xi_{k+1}}(\hat x_{k+1}) - \nabla^2 f(\hat x_{k+1}))(x_{k+1}-x_{k})}^2 }.
\end{eqnarray*}

Next, by \Cref{assum:Hessian_boundedvariance_v2}, 
\begin{eqnarray*}
    \Exp{\normeu{ W_{k+1}}^2} 
    & \leq & 3 \Exp{\normeu{\nabla f(x_{k+1})-\nabla f(x_{k})}^2} + 3 \Exp{\normeu{\nabla^2 f(\hat x_{k+1})(x_{k+1}-x_{k})}^2} \\
    && + 3 \sigma_H^2 \Exp{ \normeu{x_{k+1}-x_{k}}^2 } \\
    & \overset{(a)}{\leq} & \frac{3}{{\color{red} \ubar \rho }^2} \Exp{\normstar{\nabla f(x_{k+1})-\nabla f(x_{k})}^2} +  \frac{3}{{\color{red} \ubar \rho }^2} \Exp{\normstar{\nabla^2 f(\hat x_{k+1})(x_{k+1}-x_{k})}^2} \\
    && + \frac{3 \sigma_H^2}{{\color{blue}\ubar\theta}^2} \Exp{ \norm{x_{k+1}-x_{k}}^2 }\\
    & \leq & \frac{3}{{\color{red} \ubar \rho }^2} \Exp{\normstar{\nabla f(x_{k+1})-\nabla f(x_{k})}^2} +  \frac{3}{{\color{red} \ubar \rho}^2} \Exp{ \norm{\nabla^2 f(\hat x_{k+1})}_{\text{op}}^2 \norm{x_{k+1}-x_{k}}^2} \\
    && + \frac{3 \sigma_H^2}{{\color{blue}\ubar\theta}^2} \Exp{ \norm{x_{k+1}-x_{k}}^2 }
\end{eqnarray*}
where in $(a)$ we used \eqref{eqn:norm_diff}.

Next, applying \Cref{lemma:Relaxed_Lipschitz_Grad}, we obtain
\begin{eqnarray*}
    \Exp{\normeu{ W_{k+1}}^2} 
    & \leq & \frac{3}{{\color{red} \ubar \rho }^2} \Exp{ (L_0 + L_1 \normstar{\nabla f(x_{k})})^2 \exp(2 L_1 \norm{x_{k+1}-x_{k}}) \norm{x_{k+1} -x_{k}}^2}  \\
    && +  \frac{3}{{\color{red} \ubar \rho }^2} \Exp{ (L_0 + L_1 \normstar{\nabla f(\hat x_{k+1})})^2\norm{x_{k+1}-x_{k}}^2}  + \frac{3 \sigma_H^2}{{\color{blue}\ubar\theta}^2} \Exp{ \norm{x_{k+1}-x_{k}}^2 }. 
\end{eqnarray*}

By the fact that $\norm{x_{k+1}-x_{k}} \leq \eta_{k}$, 
\begin{eqnarray*}
\Exp{\normeu{ W_{k+1}}^2} 
& \leq & \frac{3}{{\color{red} \ubar \rho }^2}  (L_0 + L_1 \normstar{\nabla f(x_{k})})^2 \exp(2 L_1 \eta_{k}) \eta_{k}^2  \\
&& +  \frac{3}{{\color{red} \ubar \rho}^2} \eta_{k}^2 \Exp{ (L_0 + L_1 \normstar{\nabla f(\hat x_{k+1})})^2}  + \frac{3 \sigma_H^2}{{\color{blue}\ubar\theta}^2} \eta_{k}^2. 
\end{eqnarray*}

Next, since 
\begin{eqnarray*}
    (L_0+L_1\normstar{\nabla f(\hat x_{k+1})})^2 
    &\leq & (L_0 + L_1\normstar{\nabla f(x_{k})} + L_1\normstar{\nabla f(\hat x_{k+1}) - \nabla f(x_{k})})^2 \\
    &\leq & 2 (L_0 + L_1 \normstar{\nabla f(x_{k})})^2 + 2 L_1^2 \normstar{\nabla f(\hat x_{k+1}) - \nabla f(x_{k})}^2 \\
    &\leq &  2 (L_0 + L_1 \normstar{\nabla f(x_{k})})^2  \\
    &&+ 2 L_1^2 (L_0 + L_1\normstar{\nabla f(x_{k})})^2 \exp(2L_1 \norm{\hat x_{k+1} - x_{k}}) \norm{\hat x_{k+1}-x_{k}}^2,
\end{eqnarray*}
we obtain
\begin{eqnarray*}
    \Exp{\normeu{ W_{k+1}}^2}  &\leq&  \frac{3}{{\color{red} \ubar \rho }^2}  (L_0 + L_1 \normstar{\nabla f(x_{k})})^2 (\exp(2 L_1 \eta_{k}) + 2)\eta_{k}^2  + \frac{3 \sigma_H^2}{{\color{blue}\ubar\theta}^2} \eta_{k}^2 \\
    && +  \frac{6}{{\color{red} \ubar \rho}^2} L_1^2\eta_{k}^2  (L_0 + L_1 \normstar{\nabla f(x_{k})})^2 \Exp{ \exp(2L_1 \norm{\hat x_{k+1} - x_{k}})\norm{\hat x_{k+1} -x_{k}}^2}   \\
    & \overset{(a)}{\leq} & \frac{3}{{\color{red} \ubar \rho}^2}  (L_0 + L_1 \normstar{\nabla f(x_{k})})^2 (\exp(2 L_1 \eta_{k}) + 2)\eta_{k}^2 + \frac{3 \sigma_H^2}{{\color{blue}\ubar\theta}^2} \eta_{k}^2 \\
    && +  \frac{6}{{\color{red} \ubar \rho}^2} L_1^2\eta_{k}^2  (L_0 + L_1 \norm{\nabla f(x_{k})})^2 \Exp{ \exp(2L_1 b_k \norm{x_{k+1} - x_{k}})b_k^2\norm{x_{k+1} -x_{k}}^2} ,
\end{eqnarray*}
where $(a)$ results from the definition of $\hat x_{k+1}$. 
Next, by the fact that $\norm{x_{k+1}-x_{k}} \leq \eta_{k}$,
\begin{eqnarray*}
    \Exp{\normeu{ W_{k+1}}^2} &\leq& \frac{3}{{\color{red} \ubar \rho}^2}  (L_0 + L_1 \normstar{\nabla f(x_{k})})^2 (\exp(2 L_1 \eta_{k}) + 2)\eta_{k}^2  \\
    && +  \frac{6}{{\color{red} \ubar \rho}^2} L_1^2\eta_{k}^2  (L_0 + L_1 \norm{\nabla f(x_{k})})^2 \Exp{ \exp(2L_1 b_{k+1} \eta_{k})b_{k+1}^2\eta_{k}^2}  + \frac{3 \sigma_H^2}{{\color{blue}\ubar\theta}^2} \eta_{k}^2 \\
    & = &  \frac{3}{{\color{red} \ubar \rho}^2}  (L_0 + L_1 \normstar{\nabla f(x_{k})})^2 (\exp(2 L_1 \eta_{k}) + 2)\eta_{k}^2  \\
    && +  \frac{6}{{\color{red} \ubar \rho}^2} L_1^2\eta_{k}^2  (L_0 + L_1 \norm{\nabla f(x_{k})})^2  \eta_{k}^2 \int_{0}^1 \exp(2L_1 b \eta_{k})b^2 db  \\
    && + \frac{3 \sigma_H^2}{{\color{blue}\ubar\theta}^2} \eta_{k}^2. 
\end{eqnarray*}

Since 
\begin{eqnarray*}
    \int \exp(a_k z) z^2 dz = \frac{1}{a_k} z^2 \exp(a_k z) - \frac{2}{a_k}\int z \exp(a_k z) dz,
\end{eqnarray*}
we obtain $\int_{z=0}^1 \exp(a_k z) z^2 dz \leq \frac{1}{a_k} \exp(a_k)$. 
Therefore,
\begin{eqnarray*}
    \Exp{\normeu{ W_{k+1}}^2}  
& \leq  & \frac{3}{{\color{red} \ubar \rho}^2}  (L_0 + L_1 \normstar{\nabla f(x_{k})})^2 (\exp(2 L_1 \eta_{k}) + 2)\eta_{k}^2  \\
&& +  \frac{6}{{\color{red} \ubar \rho}^2} L_1^2\eta_{k}^2  (L_0 + L_1 \norm{\nabla f(x_{k})})^2  \eta_{k}^2 \frac{1}{2L_1 \eta_{k}} \exp(2L_1 \eta_{k})  + \frac{3 \sigma_H^2}{{\color{blue}\ubar\theta}^2} \eta_{k}^2.
\end{eqnarray*}

By re-arranging the terms, and by the fact that $(a+b)^2 \leq 2 a^2 + 2b^2$ for $a,b\in\R$,
\begin{eqnarray*}
    \Exp{\normeu{ W_{k+1}}^2}  
    & \leq &  \frac{3}{{\color{red} \ubar \rho}^2}  (L_0 + L_1 \normstar{\nabla f(x_{k})})^2 \left(\exp(2 L_1 \eta_{k}) + \eta_{k}L_1\exp(2 L_1 \eta_{k}) + 2\right)\eta_{k}^2  \\
    && + \frac{3 \sigma_H^2}{{\color{blue}\ubar\theta}^2} \eta_{k}^2 \\
    & \leq & \frac{6}{{\color{red} \ubar \rho}^2}  L_1^2 \normstar{\nabla f(x_{k})}^2 \left(\exp(2 L_1 \eta_{k}) + \eta_{k}L_1\exp(2 L_1 \eta_{k}) + 2\right)\eta_{k}^2   \\
    && +  \frac{6}{{\color{red} \ubar \rho}^2}  L_0^2 \left((\exp(2 L_1 \eta_{k}) + \eta_{k}L_1\exp(2 L_1 \eta_{k}) + 2\right)\eta_{k}^2 + \frac{3 \sigma_H^2}{{\color{blue}\ubar\theta}^2} \eta_{k}^2.
\end{eqnarray*}

\paragraph{Step 4) Plugging $\Exp{\normeu{W_k}^2}$ back into the upper-bound for $\Exp{\normstar{e_{k+1}}}$.}
By plugging $\Exp{\normeu{W_k}^2}$ back into the upper-bound for $\Exp{\normstar{e_{k+1}}}$, 
\begin{eqnarray*}
 \Exp{\normstar{e_{k+1} } } & \leq &   \prod_{t=0}^k(1-\alpha_t) \Exp{\normstar{e_0}} + {\color{red}\bar{\rho}}\sqrt{\sum_{t=0}^k \left( \prod_{j=t+1}^k (1-\alpha_j)^2\right) (1-\alpha_t)^2  (A_{t+1}+ B_{t+1})} \\
&& + {\color{red}\bar{\rho}} \sigma_g  \sqrt{\sum_{t=0}^k \left(\prod_{j=t+1}^k (1-\alpha_j)^2 \right) \alpha_t^2},
\end{eqnarray*}
where $A_{k+1} = \frac{6}{{\color{red} \ubar \rho}^2}  L_1^2  (\exp(2 L_1 \eta_{k}) + \eta_{k}L_1\exp(2 L_1 \eta_{k}) + 2)\eta_{k}^2 \Exp{\normstar{\nabla f(x_{k})}^2}$ and $B_{k+1} = \frac{6}{{\color{red} \ubar \rho}^2}  L_0^2 (\exp(2 L_1 \eta_{k}) + \eta_{k}L_1\exp(2 L_1 \eta_{k}) + 2)\eta_{k}^2 + \frac{3 \sigma_H^2}{{\color{blue}\ubar\theta}^2} \eta_{k}^2$.

\paragraph{Step 5) Deriving the convergence bound under constant tuning parameters.}

If $\eta_k=\eta$ and $\alpha_k=\alpha$, then 
\begin{eqnarray*}
 \Exp{\normstar{e_{k+1} } } 
& \leq &  (1-\alpha)^{k+1}\Exp{\normstar{e_0}}  + {\color{red}\bar\rho}\sqrt{ \sum_{t=0}^{k} (1-\alpha)^{2(k-t+1)} (A_{t+1} + B )}  \\
&& + {\color{red}\bar\rho} \sigma_g \sqrt{\sum_{t=0}^{k} (1-\alpha)^{2(k-t)} \alpha^2} \\
& \leq &   (1-\alpha)^{k+1}\Exp{\normstar{e_0}}   + {\color{red}\bar\rho}\sqrt{ \sum_{t=0}^{k} (1-\alpha)^{2(k-t+1)} A_{t+1}} \\
&&+  {\color{red}\bar\rho}\sqrt{ \sum_{t=0}^{k} (1-\alpha)^{2(k-t+1)} B}  + {\color{red}\bar\rho} \sigma_g \sqrt{\sum_{t=0}^{k} (1-\alpha)^{2(k-t)} \alpha^2},  
\end{eqnarray*}
where $A_{k+1} = c\eta^2 \normstar{\nabla f(x_{k})}^2$, $B = \frac{6}{{\color{red} \ubar \rho}^2}  L_0^2 (\exp(2 L_1 \eta) +  L_1 \eta\exp(2 L_1 \eta) + 2)\eta^2 + \frac{3 \sigma_H^2}{{\color{blue}\ubar\theta}^2} \eta^2$, and $c= \frac{6}{{\color{red} \ubar \rho}^2}  L_1^2  (\exp(2 L_1 \eta) + L_1\eta \exp(2 L_1 \eta) + 2)$.

Next, since 
\begin{eqnarray*}
    \sum_{t=0}^{k-1} (1-\alpha)^{2(k-t+1)}
    & \leq  & \sum_{j=0}^\infty ((1-\alpha)^2)^j = \frac{1}{1-(1-\alpha)^2}= \frac{1}{\alpha(2-\alpha)}  \overset{\alpha \in [0,1]}{\leq}  \frac{1}{\alpha}, \\
    &\text{and}&\\
     \sum_{t=0}^{k-1} (1-\alpha)^{2(k-t)} \alpha^2  & \leq &  \alpha^2 \sum_{j=0}^\infty ((1-\alpha)^2)^j  = \frac{\alpha^2}{1-(1-\alpha)^2} = \frac{\alpha}{2-\alpha} \overset{\alpha \in [0,1]}{\leq} \alpha,
\end{eqnarray*}
we obtain
\begin{eqnarray*}
    \Exp{\normstar{e_{k+1} } } 
    &\overset{(a)}{\leq} &  (1-\alpha)^{k+1}\normstar{e_0}   + {\color{red}\bar\rho}\sqrt{ \sum_{t=0}^{k} (1-\alpha)^{2(k-t)} (1-\alpha)^2 A_{t+1}} +  {\color{red}\bar\rho} \frac{\eta}{\sqrt{\alpha}} \hat B \\
    && + {\color{red}\bar\rho} \sqrt{\alpha} \sigma_g \\
    & \overset{(b)}{\leq} &  (1-\alpha)^{k+1}\normstar{e_0}   + {\color{red}\bar\rho}\sqrt{ \sum_{t=0}^{k} (1-\alpha)^{2(k-t+1)}  c \eta^2 \normstar{\nabla f(x_{t})}^2} +  {\color{red}\bar\rho} \frac{\eta}{\sqrt{\alpha}} \hat B \\
    && + {\color{red}\bar\rho} \sqrt{\alpha} \sigma_g \\
    & \leq &  (1-\alpha)^{k}\normstar{e_0}   + {\color{red}\bar\rho}\hat{c} \eta \sum_{t=0}^{k} (1-\alpha)^{(k-t+1)}  \normstar{\nabla f(x_{t})} +  {\color{red}\bar\rho} \frac{\eta}{\sqrt{\alpha}} \hat B \\
    && + {\color{red}\bar\rho} \sqrt{\alpha} \sigma_g,
\end{eqnarray*}
where we reach $(a)$ by denoting $\hat{c} = \frac{3}{{\color{red} \ubar \rho}}  L_1 (\exp( L_1 \eta) + \sqrt{L_1 \eta} \exp( L_1 \eta) + 2) $ and  $\hat B = \frac{3}{{\color{red} \ubar \rho}}  L_0 (\exp( L_1 \eta) + \sqrt{L_1 \eta} \exp( L_1 \eta) + 2) + \frac{2\sigma_H}{{\color{blue}\ubar\theta}}$, and $(b)$ by using the condition that $\alpha \in [0,1]$  and the definition of $ A_{t+1}$.

Plugging the above result into the main descent inequality with $\eta_k=\eta$ and $\alpha_k = \alpha$ and denoting $\varphi = \left( 1 - \exp(L_1 \eta) \frac{L_1\eta}{2} \right)$, we obtain
\begin{eqnarray*}
    \sum_{k=0}^K \eta\varphi  \Exp{ \normstar{\nabla f(x_k)} }   
    & \leq&  \Delta + 2 \sum_{k=0}^K \eta \Exp{\normstar{e_k}}  + \frac{L_0}{2} \sum_{k=0}^K \exp(L_1 \eta) \eta^2 \\
    & \leq & \Delta + 2 \eta \sum_{k=0}^K (1-\alpha)^{k}\Exp{\normstar{e_0}} + \frac{L_0}{2}  \exp(L_1 \eta) \eta^2 (K+1) \\
    && + 2 {\color{red}\bar\rho}\hat{c} \eta^2 \sum_{k=0}^K \sum_{t=0}^{k-1} (1-\alpha)^{(k-t)}  \Exp{ \normstar{\nabla f(x_{t})} } \\
    && + 2 {\color{red}\bar\rho} \frac{\eta^2}{\sqrt{\alpha}} \hat B (K+1) + 2  {\color{red}\bar\rho}  \eta \sqrt{\alpha} \sigma_g (K+1)   .
\end{eqnarray*}

By the fact that $\sum_{k=0}^K (1-\alpha)^k \leq \sum_{k=0}^\infty (1-\alpha)^k = \frac{1}{\alpha}$, 
\begin{eqnarray*}
    \sum_{k=0}^K \eta\varphi  \Exp{ \normstar{\nabla f(x_k)} } &\leq& \Delta + \frac{2 \eta \Exp{\normstar{e_0}}}{\alpha}   + 2 {\color{red}\bar\rho}\hat{c} \eta^2 \sum_{k=0}^K \sum_{t=0}^{k-1} (1-\alpha)^{(k-t)}  \Exp{ \normstar{\nabla f(x_{t})} }  \\
    && + 2 {\color{red}\bar\rho} \frac{\eta^2}{\sqrt{\alpha}} \hat B (K+1) + 2  {\color{red}\bar\rho}  \eta \sqrt{\alpha} \sigma_g (K+1)    + \frac{L_0}{2}  \exp(L_1 \eta) \eta^2 (K+1).
\end{eqnarray*}

Next, since 
\begin{eqnarray*}
    \sum_{k=0}^K \sum_{t=0}^{k-1} (1-\alpha)^{k-t} \Exp{\normstar{\nabla f(x_{t})}} 
    & \leq & \sum_{k=0}^K \sum_{t=0}^{k} (1-\alpha)^{k-t} \Exp{\normstar{\nabla f(x_{t})}} \\
    & = & \sum_{t=0}^K (\sum_{k=t}^K (1-\alpha)^{k-t}) \Exp{\normstar{\nabla f(x_{t})}} \\
    & \leq & \sum_{t=0}^K (\sum_{k=0}^{\infty} (1-\alpha)^{k}) \Exp{\normstar{\nabla f(x_{t})}} \\
    & = & \frac{1}{\alpha} \sum_{k=0}^K \Exp{\normstar{\nabla f(x_{k})}},
\end{eqnarray*}
we obtain
\begin{eqnarray*}
    \sum_{k=0}^K \eta( \varphi - 2 {\color{red}\bar\rho}\hat{c} \frac{\eta}{\alpha}) \Exp{ \normstar{\nabla f(x_k)} } &\leq& \Delta + \frac{2 \eta \Exp{\normstar{e_0}}}{\alpha}     + 2 {\color{red}\bar\rho} \frac{\eta^2}{\sqrt{\alpha}} \hat B (K+1)\\
    && + 2  {\color{red}\bar\rho}  \eta \sqrt{\alpha} \sigma_g (K+1)    + \frac{L_0}{2}  \exp(L_1 \eta) \eta^2 (K+1).
\end{eqnarray*}

\paragraph{Step 6) Choosing tuning parameters.}

If $\eta \leq  \frac{\alpha}{80 L_1} \left(\frac{{\color{red}\bar\rho}}{{\color{red}\ubar\rho}}\right)^{-1}$, then 
\begin{eqnarray*}
    \frac{\eta}{2}\sum_{k=0}^K  \Exp{ \normstar{\nabla f(x_k)} } &\leq& \Delta + \frac{2 \eta \Exp{\normstar{e_0}}}{\alpha}     + 2 {\color{red}\bar\rho} \frac{\eta^2}{\sqrt{\alpha}} \hat B (K+1) \\
    && + 2  {\color{red}\bar\rho}  \eta \sqrt{\alpha} \sigma_g (K+1)    + \frac{L_0}{2} \exp(L_1 \eta) \eta^2 (K+1).
\end{eqnarray*}
Therefore, 
\begin{eqnarray*}
\min_{k \in \{0,1,\ldots,K\}} \Exp{\normstar{\nabla f(x_k)}}
& \leq & \frac{1}{K+1}\sum_{k=0}^K \Exp{\normstar{\nabla f(x_k)}}  \\
& \leq & \frac{2 \Delta}{\eta (K+1)} +  \frac{4 \Exp{\normstar{e_0}}}{\alpha (K+1)}  + 4 {\color{red}\bar\rho} \frac{\eta}{\sqrt{\alpha}} \hat B \\
&& + 4 {\color{red}\bar\rho}   \sqrt{\alpha} \sigma_g     + L_0 \exp(L_1 \eta) \eta.
\end{eqnarray*}

If $\eta = \frac{\hat \eta}{(K+1)^{\nicefrac{2}{3}}}$ with $\hat \eta =\frac{1}{80 L_1} \left(\frac{{\color{red}\bar\rho}}{{\color{red}\ubar\rho}}\right)^{-1}$, and $\alpha = \frac{1}{(K+1)^{\nicefrac{2}{3}}}$, then 
\begin{eqnarray*}
\min_{k \in \{0,1,\ldots,K\}} \Exp{\normstar{\nabla f(x_k)}}
& \leq & \frac{2 \Delta}{\hat \eta (K+1)^{\nicefrac{1}{3}}} +  \frac{4 \Exp{\normstar{e_0}}}{(K+1)^{\nicefrac{1}{3}}}  +  \frac{4 {\color{red}\bar\rho} \hat\eta \hat B}{(K+1)^{1/3}}  \\
&& + 4 {\color{red}\bar\rho}   \frac{1}{(K+1)^{\nicefrac{1}{3}}} \sigma_g     + L_0 \exp(L_1 \hat\eta) \frac{\hat \eta}{(K+1)^{\nicefrac{2}{3}}}.
\end{eqnarray*}

\newpage

\section{Proof of~\Cref{thm:V2_relaxed_smoothness}}

We prove the result in the following steps. 

\paragraph{Step 1) Proving the descent inequality.} 
By following the proof arguments in Step 1) of~\Cref{lemma:V1_relaxed_smoothness} (see \eqref{eq:descent_inequality}), we have 
\begin{eqnarray*}
 \sum_{k=0}^K \eta_k\varphi_k  \normstar{\nabla f(x_k)}    &\leq&  \Delta + 2 \sum_{k=0}^K \eta_k \normstar{\nabla f(x_k)-m_k}  + \frac{L_0}{2} \sum_{k=0}^K \exp(L_1 \eta_k) \eta_k^2,
\end{eqnarray*}
where $\varphi_k \eqdef \left( 1 - \exp(L_1 \eta_k) 
 \nicefrac{L_1\eta_k}{2} \right)$ and $\Delta \eqdef f(x_0) - f_{\inf}$.

\paragraph{Step 2) Bounding the error term.}
Next, we bound $\normstar{e_k}$, where $e_k \eqdef  m_k - \nabla f(x_k)$. 
From the definition of $e_k$,
\begin{eqnarray*}
    e_{k+1} 
    & = & m_{k+1} - \nabla f(x_{k+1}) \\
    & = & (1-\alpha_k)e_k + (1-\alpha_k)Z_f(x_k;x_{k+1})  + (1-\alpha_k) \varepsilon^H_{k+1}  + \alpha_k \varepsilon^g_{k+1} ,
\end{eqnarray*}
where $Z_f(x_k,x_{k+1}) = \nabla f(x_k) - \nabla f(x_{k+1}) + \nabla^2 f(x_{k+1})(x_{k+1}-x_k)$, $\varepsilon^H_{k+1} = (\nabla^2 f_{\xi_{k+1}}(x_{k+1})-\nabla^2 f(x_{k+1}))(x_{k+1}-x_k)$ and $\varepsilon^g_{k+1} = \nabla f_{\xi_{k+1}}(x_{k+1}) - \nabla f(x_{k+1})$. 

By recursively applying the above inequality, 
\begin{eqnarray*}
    e_{k+1} 
    & =& \prod_{t=0}^k (1-\alpha_t) e_0 +\sum_{t=0}^k \left( \prod_{j=t+1}^k (1-\alpha_j) \right) (1-\alpha_t)  Z_f(x_t,x_{t+1}) \\
    && +\sum_{t=0}^k \left( \prod_{j=t+1}^k (1-\alpha_j) \right) (1-\alpha_t)  \varepsilon^H_{t+1}+  \sum_{t=0}^k \left(\prod_{j=t+1}^k (1-\alpha_j) \right) \alpha_t \varepsilon^g_{t+1}.
\end{eqnarray*}
Therefore, 
\begin{eqnarray*}
    \Exp{\normstar{e_{k+1} }}
    & \leq & \Exp{\normstar{\prod_{t=0}^k (1-\alpha_t) e_0}} + \underbrace{\Exp{\normstar{\sum_{t=0}^k \left( \prod_{j=t+1}^k (1-\alpha_j) \right) (1-\alpha_t))  Z_f(x_t,x_{t+1})}}}_{\eqdef \circledThree} \\
    && + \underbrace{\Exp{\normstar{\sum_{t=0}^k \left( \prod_{j=t+1}^k (1-\alpha_j) \right) (1-\alpha_t))  \varepsilon^H_{t+1} }}}_{\eqdef \circledFour} +   \underbrace{\Exp{\normstar{\sum_{t=0}^k \left(\prod_{j=t+1}^k (1-\alpha_j) \right) \alpha_t \varepsilon^g_{t+1}}}}_{\eqdef \circledFive}.
\end{eqnarray*}

Next, we bound $\circledFive$ by using~\Cref{assum:Hessian_boundedvariance_v2} and by following the proof arguments for bounding $\circledOne$ in Step 2) of the convergence proof for~\Cref{lemma:V1_relaxed_smoothness}. Then, we have 
\begin{eqnarray*}
    \circledFive
    & \leq & {\color{red}\bar\rho}   \Exp{\normeu{\sum_{t=0}^k \left(\prod_{j=t+1}^k (1-\alpha_j) \right) \alpha_t \varepsilon^g_{t+1}}} \\
    & \leq & {\color{red}\bar\rho}  \sqrt{  \Exp{\normeu{\sum_{t=0}^k \left(\prod_{j=t+1}^k (1-\alpha_j) \right) \alpha_t \varepsilon^g_{t+1}}^2} } \\
    & \leq & {\color{red}\bar\rho} \sigma_g \sqrt{ \sum_{t=0}^k \left(\prod_{j=t+1}^k (1-\alpha_j)^2 \right) \alpha_t^2}, 
\end{eqnarray*}
Therefore, 
\begin{eqnarray*}
    \Exp{\normstar{e_{k+1} }}
    & \leq & \prod_{t=0}^k (1-\alpha_t)\Exp{\normstar{e_0}} + ~\circledThree~ +~\circledFour~ +   {\color{red}\bar\rho} \sigma_g \sqrt{ \sum_{t=0}^k \left(\prod_{j=t+1}^k (1-\alpha_j)^2 \right) \alpha_t^2}.
\end{eqnarray*}

Next, we bound $\circledFour$. 
\begin{eqnarray*}
    \circledFour  
    & \leq & {\color{red}\bar\rho} \Exp{\normeu{\sum_{t=0}^k \left( \prod_{j=t+1}^k (1-\alpha_j) \right) (1-\alpha_t))  \varepsilon^H_{t+1} }} \\
    & \leq & {\color{red}\bar\rho}  \sqrt{ \Exp{\normeu{\sum_{t=0}^k \left( \prod_{j=t+1}^k (1-\alpha_j) \right) (1-\alpha_t))  \varepsilon^H_{t+1} }^2} }.
\end{eqnarray*}
From the definition of $\varepsilon_{t+1}^H$ and from~\Cref{assum:Hessian_boundedvariance_v2}, 
\begin{eqnarray*}
    \circledFour
  & \leq & {\color{red}\bar\rho}  \sqrt{  \sum_{t=0}^k \left( \prod_{j=t+1}^k (1-\alpha_j)^2 \right) (1-\alpha_t)^2) \Exp{ \normeu{  \varepsilon^H_{t+1} }^2} } \\
   & \leq & {\color{red}\bar\rho}  \sqrt{  \sum_{t=0}^k \left( \prod_{j=t+1}^k (1-\alpha_j)^2 \right) (1-\alpha_t)^2) \sigma_H^2 \normeu{x_{t+1}-x_t}^2 } \\
   & \leq & \frac{{\color{red}\bar\rho}}{\color{blue} \ubar{\theta}}  \sqrt{  \sum_{t=0}^k \left( \prod_{j=t+1}^k (1-\alpha_j)^2 \right) (1-\alpha_t)^2) \sigma_H^2 \norm{x_{t+1}-x_t}^2 } \\
   & \leq & \frac{{\color{red}\bar\rho}}{\color{blue} \ubar{\theta}} \eta \sigma_H  \sqrt{  \sum_{t=0}^k \left( \prod_{j=t+1}^k (1-\alpha_j)^2 \right) (1-\alpha_t)^2  }.
\end{eqnarray*}
Therefore, 
\begin{eqnarray*}
    \Exp{\normstar{e_{k+1} }}
    & \leq & \prod_{t=0}^k (1-\alpha_t)\Exp{\normstar{ e_0}}  + \frac{{\color{red}\bar\rho}}{\color{blue} \ubar{\theta}} \eta \sigma_H  \sqrt{  \sum_{t=0}^k \left( \prod_{j=t+1}^k (1-\alpha_j)^2 \right) (1-\alpha_t)^2  } \\
    && + ~\circledThree~ +   {\color{red}\bar\rho} \sigma_g \sqrt{ \sum_{t=0}^k \left(\prod_{j=t+1}^k (1-\alpha_j)^2 \right) \alpha_t^2}.
\end{eqnarray*}

Next, we bound $\circledThree$: 
\begin{eqnarray*}
    \circledThree &\leq&  \sum_{t=0}^k \left( \prod_{j=t+1}^k (1-\alpha_j) \right) (1-\alpha_t)) \Exp{ \normstar{ Z_f(x_t,x_{t+1}) } }.
\end{eqnarray*}

From~\Cref{lemma:Relaxed_Grad_and_Hessian},
\begin{eqnarray*}
   \normstar{ Z_f(x_t,x_{t+1}) } & \leq &   \frac{1}{2}(M_0 + M_1 \normstar{\nabla f(x_{t+1})}) \norm{x_t-x_{t+1}}^2 \\
   && + \frac{1}{3} M_1 (L_0+L_1 \normstar{\nabla f(x_{t+1})}) \exp(L_1 \norm{x_t-x_{t+1}}) \norm{x_t-x_{t+1}}^3 \\
   & \leq &  \frac{1}{2}(M_0 + M_1 \normstar{\nabla f(x_{t+1})}) \eta_t^2 + \frac{1}{3} M_1 (L_0+L_1 \normstar{\nabla f(x_{t+1})}) \exp(L_1 \eta_t) \eta_t^3 \\
   & = & \left(\frac{M_0}{2} + \frac{M_1}{3}L_0 \eta_t\exp(L_1\eta_t)\right) \eta_t^2 + \left(\frac{M_1}{2} + \frac{M_1}{3} L_1 \eta_t\exp(L_1 \eta_t) \right) \eta_t^2 \normstar{\nabla f(x_{t+1})}.
\end{eqnarray*}
Therefore, 
\begin{eqnarray*}
\circledThree &\leq&  \sum_{t=0}^k \left( \prod_{j=t+1}^k (1-\alpha_j) \right) (1-\alpha_t)) \left(\frac{M_0}{2} + \frac{M_1}{3}L_0\eta_t \exp(L_1\eta_t)\right) \eta_t^2  \\
&& +  \sum_{t=0}^k \left( \prod_{j=t+1}^k (1-\alpha_j) \right) (1-\alpha_t))\left(\frac{M_1}{2} + \frac{M_1}{3}L_1 \eta_t\exp(L_1 \eta_t) \right) \eta_t^2 \normstar{\nabla f(x_{t+1})}.
\end{eqnarray*}
Plugging $\circledThree$ into the upper-bound for $\Exp{\normstar{e_{k+1} }}$ yields
\begin{eqnarray*}
    \Exp{\normstar{e_{k+1} }}
    & \leq & \prod_{t=0}^k (1-\alpha_t)\Exp{\normstar{ e_0}} \\
    &&+ \sum_{t=0}^k \left( \prod_{j=t+1}^k (1-\alpha_j) \right) (1-\alpha_t) \left(\frac{M_0}{2} + \frac{M_1}{3}L_0\eta_t \exp(L_1\eta_t)\right) \eta_t^2  \\
    && +  \sum_{t=0}^k \left( \prod_{j=t+1}^k (1-\alpha_j) \right) (1-\alpha_t)\left(\frac{M_1}{2} + \frac{M_1}{3} L_1 \eta_t\exp(L_1 \eta_t) \right) \eta_t^2 \normstar{\nabla f(x_{t+1})} \\
    && + \frac{{\color{red}\bar\rho}}{\color{blue} \ubar{\theta}} \eta \sigma_H  \sqrt{  \sum_{t=0}^k \left( \prod_{j=t+1}^k (1-\alpha_j)^2 \right) (1-\alpha_t)^2  } +   {\color{red}\bar\rho} \sigma_g \sqrt{ \sum_{t=0}^k \left(\prod_{j=t+1}^k (1-\alpha_j)^2 \right) \alpha_t^2}.
\end{eqnarray*}

If $\eta_k = \eta$ and $\alpha_k = \alpha$, then 
\begin{eqnarray*}
    \Exp{\normstar{e_{k+1} }}
    & \leq & (1-\alpha)^{k+1} \normstar{e_0} + \left(\frac{M_0}{2} + \frac{M_1L_0\eta \exp(L_1\eta)}{3}\right) \sum_{t=0}^k (1-\alpha)^{k-t+1} \eta^2   \\
    && + \left(\frac{M_1}{2} + \frac{M_1L_1 \exp(L_1 \eta) \eta}{3}\right) \sum_{t=0}^k (1-\alpha)^{k-t+1}  \eta^2 \normstar{\nabla f(x_{t+1})} \\
    && + \frac{{\color{red}\bar\rho}}{\color{blue} \ubar{\theta}} \eta \sigma_H  \sqrt{  \sum_{t=0}^k (1-\alpha)^{2(k-t+1)}  } +   {\color{red}\bar\rho} \sigma_g \sqrt{ \sum_{t=0}^k (1-\alpha)^{2(k-t)} \alpha^2}.
\end{eqnarray*}

Next, since 
\begin{eqnarray*}
    \sum_{t=0}^{k-1} (1-\alpha)^{2(k-t+1)}
    & \leq & \sum_{j=0}^\infty ((1-\alpha)^2)^j  = \frac{1}{\alpha(2-\alpha)}  \overset{\alpha \in [0,1]}{\leq} \frac{1}{\alpha}, \\
    &\text{and}&\\
     \sum_{t=0}^{k-1} (1-\alpha)^{2(k-t)} \alpha^2  & \leq &  \alpha^2 \sum_{j=0}^\infty ((1-\alpha)^2)^j  = \frac{\alpha^2}{1-(1-\alpha)^2} = \frac{\alpha}{2-\alpha} 
     \overset{\alpha \in [0,1]}{\leq} \alpha,
\end{eqnarray*}
we obtain
\begin{eqnarray*}
    \Exp{\normstar{e_{k+1} }}
    & \leq & (1-\alpha)^{k+1} \Exp{\normstar{e_0}} + \left(\frac{M_0}{2} + \frac{M_1}{3}L_0\eta \exp(L_1\eta)\right) \frac{\eta^2}{\alpha}   \\
    && + \left(\frac{M_1}{2} + \frac{M_1}{3}L_1  \eta\exp(L_1 \eta)\right) \sum_{t=0}^k (1-\alpha)^{k-t+1}  \eta^2 \normstar{\nabla f(x_{t+1})} \\
    && + \frac{{\color{red}\bar\rho}}{\color{blue} \ubar{\theta}} \frac{\eta}{\sqrt{\alpha}} \sigma_H  +   {\color{red}\bar\rho} \sqrt{\alpha}\sigma_g.
\end{eqnarray*}

Therefore, recalling $\varphi =\eta( 1 - \exp(L_1 \eta) 
 \nicefrac{L_1\eta}{2} ) $, we have 
\begin{eqnarray*}   
    \sum_{k=0}^K  \eta \varphi \Exp{\normstar{\nabla f(x_k)}}   &\leq&  \Delta + \frac{L_0}{2} \sum_{k=0}^K \exp(L_1 \eta) \eta^2 + 2  \eta \sum_{k=0}^K \Exp{\normstar{e_k}}   \\
    &\leq& \Delta  + \frac{L_0}{2}  \exp(L_1 \eta) \eta^2(K+1)  + 2  \eta \sum_{k=0}^K (1-\alpha)^{k} \Exp{\normstar{e_0}}\\
    &&+ \eta  \left(M_0 + \frac{2M_1}{3}L_0\eta \exp(L_1\eta)\right) \frac{\eta^2}{\alpha} (K+1) \\
    && + \eta \left(M_1 + \frac{2M_1}{3} L_1 \eta\exp(L_1 \eta) \right)  \sum_{k=0}^K \sum_{t=0}^{k-1} (1-\alpha)^{k-t}  \eta^2 \normstar{\nabla f(x_{t+1})} \\
    && + 2\eta\left(\frac{{\color{red}\bar\rho}}{\color{blue} \ubar{\theta}} \frac{\eta}{\sqrt{\alpha}} \sigma_H  +   {\color{red}\bar\rho} \sqrt{\alpha}\sigma_g\right)(K+1),
\end{eqnarray*}
where $\Delta = \Exp{f(x_0)-f_{\inf}}$.

Next, since 
\begin{eqnarray*}
    \sum_{k=0}^K (1-\alpha)^k  \leq \frac{1}{\alpha}, &\text{and} & \sum_{k=0}^K \sum_{t=0}^{k-1} (1-\alpha)^{k-t} \normstar{\nabla f(x_{t+1})}
    \leq \frac{1}{\alpha} \sum_{k=0}^K \normstar{\nabla f(x_k)},
\end{eqnarray*}
we obtain
\begin{eqnarray*}
    \sum_{k=0}^K \eta\vartheta  \Exp{\normstar{\nabla f(x_k)}}
    & \leq& \Delta  + (K+1)\frac{L_0}{2}  \exp(L_1 \eta) \eta^2  + 2  \frac{\eta}{\alpha}  \Exp{\normstar{e_0}}\\
    && + \eta (K+1) \left(M_0 + \frac{2M_1}{3}L_0 \eta\exp(L_1\eta)\right) \frac{\eta^2}{\alpha}  \\
    && + 2\eta\left(\frac{{\color{red}\bar\rho}}{\color{blue} \ubar{\theta}} \frac{\eta}{\sqrt{\alpha}} \sigma_H  +   {\color{red}\bar\rho} \sqrt{\alpha}\sigma_g\right)(K+1),
\end{eqnarray*}
where $\vartheta \eqdef \left( 1 - \exp(L_1 \eta) 
 \nicefrac{L_1\eta}{2}   -  \left(M_1 + \nicefrac{2M_1}{3} \cdot L_1 \eta \exp(L_1 \eta) \right) \frac{\eta^2}{\alpha}  \right)$. 
 
 If $\eta \leq \frac{\alpha}{3}\min\left\{ \frac{1}{L_1} , \frac{1}{\sqrt{M_1 } }  \right\} $, then $\eta \leq \min\left\{ \frac{1}{3L_1} ,  \frac{1}{3\sqrt{M_1}} \right\}$ and 
\begin{eqnarray*}
    \frac{\eta}{2}\sum_{k=0}^K   \Exp{\normstar{\nabla f(x_k)}}   
    &\leq& \Delta  + (K+1)\frac{L_0}{2}  \exp(L_1 \eta) \eta^2  + 2  \frac{\eta}{\alpha}  \Exp{\normstar{e_0}} \\
    &&+ \eta (K+1) \left(M_0 + \frac{2M_1}{3}L_0\eta \exp(L_1\eta)\right) \frac{\eta^2}{\alpha}  \\
    && + 2\eta\left(\frac{{\color{red}\bar\rho}}{\color{blue} \ubar{\theta}} \frac{\eta}{\sqrt{\alpha}} \sigma_H  +   {\color{red}\bar\rho} \sqrt{\alpha}\sigma_g\right)(K+1).
\end{eqnarray*}
Therefore, 
\begin{eqnarray*}
    \frac{1}{K+1}\sum_{k=0}^K   \Exp{\normstar{\nabla f(x_k)}} 
    & \leq & \frac{2\Delta}{\eta(K+1)}  +L_0 \exp(L_1 \eta) \eta  +  \frac{4\Exp{\normstar{e_0}}}{\alpha (K+1)}  \\
    && + 2 \left(M_0 + \frac{2M_1}{3}L_0\eta \exp(L_1\eta)\right) \frac{\eta^2}{\alpha}  \\
    && + 4\left(\frac{{\color{red}\bar\rho}}{\color{blue} \ubar{\theta}} \frac{\eta}{\sqrt{\alpha}} \sigma_H  +   {\color{red}\bar\rho} \sqrt{\alpha}\sigma_g\right).
\end{eqnarray*}

If $\eta = \hat \eta \alpha$ with $\hat \eta = \min\left\{ \frac{1}{5L_1} , \frac{1}{3\sqrt{M_1}}  \right\} $ and $\alpha = \frac{1}{(K+1)^{\nicefrac{2}{3}}}$, then 
\begin{eqnarray*}
    \frac{1}{K+1}\sum_{k=0}^K   \Exp{\normstar{\nabla f(x_k)}}   
    &\leq& \frac{2\Delta}{\hat\eta(K+1)^{\nicefrac{1}{3}}}  + \frac{\hat\eta L_0 \exp(L_1 \hat \eta)}{(K+1)^{\nicefrac{2}{3}}}  + \frac{4 \Exp{\normstar{e_0}}}{ (K+1)^{\nicefrac{1}{3}}} \\
    && + 2\left(M_0 + \frac{2M_1}{3} L_0\hat\eta \exp(L_1\hat \eta)\right) \frac{\hat \eta^2}{(K+1)^{\nicefrac{2}{3}}}  \\
    && + 4\frac{{\color{red}\bar\rho}}{\color{blue} \ubar{\theta}} \frac{\hat \eta \sigma_H }{(K+1)^{\nicefrac{1}{3}}}  +   4\frac{{\color{red}\bar\rho} \sigma_g}{(K+1)^{\nicefrac{1}{3}}}.
\end{eqnarray*}

Finally, by the fact that $\min_{k \in \{0,1,\ldots,K\}} \Exp{\normstar{\nabla f(x_k)}}
\leq \frac{1}{K+1}\sum_{k=0}^K \Exp{\normstar{\nabla f(x_k)}}$, we obtain the final result.

\newpage
\section{LMO-based Methods with Extrapolated Momentum}\label{app:IGT_relaxed_smoothness}

In this section, we present the convergence of LMO-based methods in~\eqref{eqn:LMO_momentum} that leverage extrapolated momentum in~\eqref{eqn:extrapolated_m}. 

\begin{theorem}
    \label{thm:IGT_relaxed_smoothness}
    Consider the problem of minimizing $f(x)=\ExpSub{\xi \sim \cD}{f_{\xi}(x)}$. Let $f$ be twice differentiable, and let   Assumptions~\ref{assum:Hessian_boundedvariance_v2},~\ref{assum:lower_bound},~\ref{assum:Relaxed_Lipschitz_Grad}, and~\ref{assum:Relaxed_Lipschitz_Hessian} hold. 
    Then, the iterates $\{x_k\}$ generated by LMO-based methods in~\eqref{eqn:LMO_momentum} that leverage extrapolated momentum in~\eqref{eqn:extrapolated_m}  
    \begin{eqnarray*}
        \alpha_k =  \alpha = \frac{1}{(K+1)^{\nicefrac{4}{7}}}, \quad \text{and} \quad  \eta_k= \eta = \frac{\hat \eta}{(K+1)^{\nicefrac{5}{7}}},
    \end{eqnarray*}
    where $\hat \eta = \min\left\{\frac{1}{3L_1}, \frac{1}{3\sqrt{M_1}}\right\}$  satisfy  
    \begin{eqnarray*}
        \min_{k \in \{0,1,\ldots,K\}} \Exp{\normstar{\nabla f(x_k)}} 
        &\leq&  \frac{2(f(x_0)-f_{\inf})}{\hat\eta(K+1)^{\nicefrac{2}{7}}}  + \frac{\hat\eta L_0 \exp(L_1 \hat \eta)}{(K+1)^{\nicefrac{5}{7}}}  + \frac{4 \left(\Exp{\normstar{e_0}}+ {\color{red}\bar\rho} \sigma_g\right)}{ (K+1)^{\nicefrac{2}{7}}} \\
        && + 2\left(M_0 + \frac{2M_1}{3} L_0 \exp(L_1\hat \eta)\right) \frac{\hat \eta^2}{(K+1)^{\nicefrac{2}{7}}}.
    \end{eqnarray*}

\end{theorem}

\Cref{thm:IGT_relaxed_smoothness} establishes the $\cO(1/K^{2/7})$ convergence of LMO-based methods using extrapolated momentum under relaxed smoothness with respect to the arbitrary norm, which matches the known rate obtained by \citet[Corollary 5]{kovalev2025understanding} under traditional smoothness with respect to the arbitrary norm. 

\subsection{Proof of \Cref{thm:IGT_relaxed_smoothness}}

We prove the result in the following steps. 

\paragraph{Step 1) Proving the descent inequality.} 
By following the proof arguments in Step 1) of~\Cref{lemma:V1_relaxed_smoothness} (see \eqref{eq:descent_inequality}), we have 
\begin{eqnarray*}
    \sum_{k=0}^K \eta_k\varphi_k  \normstar{\nabla f(x_k)}    &\leq&  \Delta + 2 \sum_{k=0}^K \eta_k \normstar{\nabla f(x_k)-m_k}  + \frac{L_0}{2} \sum_{k=0}^K \exp(L_1 \eta_k) \eta_k^2,
\end{eqnarray*}
where $\varphi_k \eqdef \left( 1 - \exp(L_1 \eta_k) 
 \nicefrac{L_1\eta_k}{2} \right)$ and $\Delta \eqdef f(x_0) - f_{\inf}$.

\paragraph{Step 2) Bounding the error term.}
Next, we bound $\normstar{e_k}$, where $e_k \eqdef  m_k - \nabla f(x_k)$. 
From the definition of $e_k$,
\begin{eqnarray*}
    e_{k+1} 
    & = & m_{k+1} - \nabla f(x_{k+1}) \\
    & = & (1-\alpha_k)e_k  +\alpha_k\left(\nabla f(y_{k+1}) - \nabla f(x_{k+1}) -\nabla^2 f(x_{k+1})(y_{k+1}-x_{k+1})\right)\\
    && +(1-\alpha_k)\left(\nabla f(x_k) - \nabla f(x_{k+1}) +\frac{\alpha_k}{1-\alpha_k}\nabla^2 f(x_{k+1})(y_{k+1}-x_{k+1})\right)\\
    && + \alpha_k\left(\nabla f_{\xi_{k+1}}(x_{k+1}) - \nabla f(x_{k+1})\right)\\
    & = & (1-\alpha_k)e_k + (1-\alpha_k)Z_f(x_k;x_{k+1})  + \alpha_k Z_f(y_{k+1},x_{k+1})  + \alpha_k \varepsilon^g_{k+1} ,
\end{eqnarray*}
where $Z_f(x,y) = \nabla f(x) - \nabla f(y) + \nabla^2 f(y)(x-y)$ and $\varepsilon^g_{k+1} = \nabla f_{\xi_{k+1}}(x_{k+1}) - \nabla f(x_{k+1})$. 

By recursively applying the above inequality, 
\begin{eqnarray*}
    e_{k+1} 
    & =& \prod_{t=0}^k (1-\alpha_t) e_0 +\sum_{t=0}^k \left( \prod_{j=t+1}^k (1-\alpha_j) \right) (1-\alpha_t)  Z_f(x_t,x_{t+1}) \\
    && +\sum_{t=0}^k \left( \prod_{j=t+1}^k (1-\alpha_j) \right) \alpha_t Z_f(y_{t+1},x_{t+1})+  \sum_{t=0}^k \left(\prod_{j=t+1}^k (1-\alpha_j) \right) \alpha_t \varepsilon^g_{t+1}.
\end{eqnarray*}
Therefore, 
\begin{eqnarray*}
    \Exp{\normstar{e_{k+1} }}
    & \leq & \Exp{\normstar{\prod_{t=0}^k (1-\alpha_t) e_0}} + \underbrace{\Exp{\normstar{\sum_{t=0}^k \left( \prod_{j=t+1}^k (1-\alpha_j) \right) (1-\alpha_t)  Z_f(x_t,x_{t+1})}}}_{\eqdef \circledSix} \\
    && + \underbrace{\Exp{\normstar{\sum_{t=0}^k \left( \prod_{j=t+1}^k (1-\alpha_j) \right) \alpha_t  Z_f(y_{t+1},x_{t+1}) }}}_{\eqdef \circledSeven} +   \underbrace{\Exp{\normstar{\sum_{t=0}^k \left(\prod_{j=t+1}^k (1-\alpha_j) \right) \alpha_t \varepsilon^g_{t+1}}}}_{\eqdef \circledEight}.
\end{eqnarray*}

Next, we bound $\circledEight$ by using~\Cref{assum:Hessian_boundedvariance_v2} and by following the proof arguments for bounding $\circledOne$ in Step 2) of the convergence proof for~\Cref{lemma:V1_relaxed_smoothness}. Then, we have 
\begin{eqnarray*}
    \circledEight
    & \leq & {\color{red}\bar\rho}   \Exp{\normeu{\sum_{t=0}^k \left(\prod_{j=t+1}^k (1-\alpha_j) \right) \alpha_t \varepsilon^g_{t+1}}} \\
    & \leq & {\color{red}\bar\rho}  \sqrt{  \Exp{\normeu{\sum_{t=0}^k \left(\prod_{j=t+1}^k (1-\alpha_j) \right) \alpha_t \varepsilon^g_{t+1}}^2} } \\
    & \leq & {\color{red}\bar\rho} \sigma_g \sqrt{ \sum_{t=0}^k \left(\prod_{j=t+1}^k (1-\alpha_j)^2 \right) \alpha_t^2}, 
\end{eqnarray*}
Therefore, 
\begin{eqnarray*}
    \Exp{\normstar{e_{k+1} }}
    & \leq & \prod_{t=0}^k (1-\alpha_t)\Exp{\normstar{e_0}} + ~\circledSix~ +~\circledSeven~ +   {\color{red}\bar\rho} \sigma_g \sqrt{ \sum_{t=0}^k \left(\prod_{j=t+1}^k (1-\alpha_j)^2 \right) \alpha_t^2}.
\end{eqnarray*}

Next, we bound $\circledSix$ and $\circledSeven$: 
\begin{eqnarray*}
    \circledSix &\leq&  \sum_{t=0}^k \left( \prod_{j=t+1}^k (1-\alpha_j) \right) (1-\alpha_t) \Exp{ \normstar{ Z_f(x_t,x_{t+1}) } };\\
    \circledSeven &\leq&  \sum_{t=0}^k \left( \prod_{j=t+1}^k (1-\alpha_j) \right) \alpha_t \Exp{ \normstar{ Z_f(y_{t+1},x_{t+1}) } }
\end{eqnarray*}
Since $(1-\alpha_t) +\alpha_t \theta_t^2 = \frac{1-\alpha_t}{\alpha_t} = \theta_t$,
\begin{eqnarray*}
   \normstar{ Z_f(x_t,x_{t+1}) } & \leq &   \frac{1}{2}(M_0 + M_1 \normstar{\nabla f(x_{t+1})}) \norm{x_t-x_{t+1}}^2 \\
   && + \frac{1}{3} M_1 (L_0+L_1 \normstar{\nabla f(x_{t+1})}) \exp(L_1 \norm{x_t-x_{t+1}}) \norm{x_t-x_{t+1}}^3 \\
   & \leq &  \frac{1}{2}(M_0 + M_1 \normstar{\nabla f(x_{t+1})}) \eta_t^2 + \frac{1}{3} M_1 (L_0+L_1 \normstar{\nabla f(x_{t+1})}) \exp(L_1 \eta_t) \eta_t^3 \\
   & = & \left(\frac{M_0}{2} + \frac{M_1}{3}L_0\eta_t \exp(L_1\eta_t)\right) \eta_t^2\\
   && + \left(\frac{M_1}{2} + \frac{M_1}{3} L_1 \eta_t\exp(L_1 \eta_t) \right) \eta_t^2 \normstar{\nabla f(x_{t+1})};
\end{eqnarray*} and 
\begin{eqnarray*}
   \normstar{ Z_f(y_{t+1},x_{t+1}) } & \leq &   \frac{1}{2}(M_0 + M_1 \normstar{\nabla f(x_{t+1})}) \norm{y_{t+1}-x_{t+1}}^2 \\
   && + \frac{1}{3} M_1 (L_0+L_1 \normstar{\nabla f(x_{t+1})}) \exp(L_1 \norm{y_{t+1}-x_{t+1}}) \norm{x_t-x_{t+1}}^3 \\
   & \leq &  \frac{1}{2}(M_0 + M_1 \normstar{\nabla f(x_{t+1})}) \theta_t^2\eta_t^2 + \frac{1}{3} M_1 (L_0+L_1 \normstar{\nabla f(x_{t+1})}) \exp(L_1 \theta_t\eta_t) \theta^3_t\eta_t^3 \\
   & = & \left(\frac{M_0}{2} + \frac{M_1}{3}L_0\theta_t\eta_t \exp(L_1\theta_t\eta_t)\right) \theta_t^2\eta_t^2\\
   && + \left(\frac{M_1}{2} + \frac{M_1}{3} L_1 \theta_t\eta_t\exp(L_1 \theta_t\eta_t) \right) \theta_t^2\eta_t^2 \normstar{\nabla f(x_{t+1})},
\end{eqnarray*}
we obtain
\begin{eqnarray*}
    \circledSix + \circledSeven &\leq&  \sum_{t=0}^k \left( \prod_{j=t+1}^k (1-\alpha_j) \right)  \left(\frac{M_0}{2} + \frac{M_1}{3}\cdot\frac{L_0 \eta_t}{\alpha_t}\exp\left(\frac{L_1 \eta_t}{\alpha_t}\right)\right) \theta_t\eta_t^2  \\
    && +  \sum_{t=0}^k \left( \prod_{j=t+1}^k (1-\alpha_j) \right) \left(\frac{M_1}{2} + \frac{M_1}{3}\cdot\frac{L_1 \eta_t}{\alpha_t}\exp\left(\frac{L_1 \eta_t}{\alpha_t}\right) \right) \theta_t\eta_t^2 \normstar{\nabla f(x_{t+1})}.
\end{eqnarray*}

Therefore, 
\begin{eqnarray*}
    \Exp{\normstar{e_{k+1} }}
    & \leq & \prod_{t=0}^k (1-\alpha_t)\Exp{\normstar{ e_0}}   +   {\color{red}\bar\rho} \sigma_g \sqrt{ \sum_{t=0}^k \left(\prod_{j=t+1}^k (1-\alpha_j)^2 \right) \alpha_t^2} \\
    &&+ \sum_{t=0}^k \left( \prod_{j=t+1}^k (1-\alpha_j) \right)  \left(\frac{M_0}{2} + \frac{M_1}{3}\cdot\frac{L_0 \eta_t}{\alpha_t} \exp\left(\frac{L_1 \eta_t}{\alpha_t}\right)\right) \theta_t\eta_t^2  \\
    && +  \sum_{t=0}^k \left( \prod_{j=t+1}^k (1-\alpha_j) \right) \left(\frac{M_1}{2} + \frac{M_1}{3} \frac{L_1 \eta_t}{\alpha_t}\exp\left(\frac{L_1 \eta_t}{\alpha_t}\right) \right) \theta_t \eta_t^2 \normstar{\nabla f(x_{t+1})}.
\end{eqnarray*}

If $\eta_k = \eta$ and $\alpha_k = \alpha$, then 
\begin{eqnarray*}
    \Exp{\normstar{e_{k+1} }}
    & \leq & (1-\alpha)^{k+1} \normstar{e_0} + \left(\frac{M_0}{2} + \frac{M_1}{3} \cdot\frac{L_0 \eta}{\alpha} \exp\left(\frac{L_1 \eta}{\alpha}\right) \right) \sum_{t=0}^k (1-\alpha)^{k-t} \theta\eta^2   \\
    && + \left(\frac{M_1}{2} + \frac{M_1}{3} \frac{L_1 \eta}{\alpha}\exp\left(\frac{L_1 \eta}{\alpha}\right) \right) \sum_{t=0}^k (1-\alpha)^{k-t}  \theta\eta^2 \normstar{\nabla f(x_{t+1})} \\
    && +   {\color{red}\bar\rho} \sigma_g \sqrt{ \sum_{t=0}^k (1-\alpha)^{2(k-t)} \alpha^2}.
\end{eqnarray*}

Next, since 
\begin{eqnarray*}
    \sum_{t=0}^{k-1} (1-\alpha)^{k-t}
    & \leq & \sum_{j=0}^\infty (1-\alpha)^j  =  \frac{1}{\alpha}, \\
    &\text{and}&\\
     \sum_{t=0}^{k-1} (1-\alpha)^{2(k-t)} \alpha^2  & \leq &  \alpha^2 \sum_{j=0}^\infty ((1-\alpha)^2)^j  = \frac{\alpha^2}{1-(1-\alpha)^2} = \frac{\alpha}{2-\alpha} 
     \overset{\alpha \in [0,1]}{\leq} \alpha,
\end{eqnarray*}
we obtain
\begin{eqnarray*}
    \Exp{\normstar{e_{k+1} }}
    & \leq & (1-\alpha)^{k+1} \normstar{e_0} + \left(\frac{M_0}{2} + \frac{M_1}{3}\cdot\frac{L_0 \eta}{\alpha} \exp\left(\frac{L_1\eta}{\alpha}\right)\right) \frac{\eta^2}{\alpha^2} +    {\color{red}\bar\rho} \sqrt{\alpha}\sigma_g  \\
    && + \left(\frac{M_1}{2} + \frac{M_1}{3}\frac{L_1\eta}{\alpha}\exp\left(\frac{L_1\eta}{\alpha}\right)\right) \sum_{t=0}^k (1-\alpha)^{k-t}  \theta\eta^2 \normstar{\nabla f(x_{t+1})} .
\end{eqnarray*}

Therefore, recalling $\varphi =\eta( 1 - \exp(L_1 \eta) 
 \nicefrac{L_1\eta}{2} ) $, we have 
\begin{eqnarray*}   
    \sum_{k=0}^K  \eta \varphi \Exp{\normstar{\nabla f(x_k)}}   &\leq&  \Delta + \frac{L_0}{2} \sum_{k=0}^K \exp(L_1 \eta) \eta^2 + 2  \eta \sum_{k=0}^K \Exp{\normstar{e_k}}   \\
    &\leq& \Delta  + \frac{L_0}{2}  \exp(L_1 \eta) \eta^2(K+1)  + 2  \eta \sum_{k=0}^K (1-\alpha)^{k} \normstar{e_0}\\
    &&+ \eta  \left(M_0 + \frac{2M_1}{3}\cdot\frac{L_0 \eta}{\alpha}\exp\left(\frac{L_1\eta}{\alpha}\right)\right) \frac{\eta^2}{\alpha^2} (K+1) \\
    && + \eta \left(M_1 + \frac{2M_1}{3}\frac{L_1\eta}{\alpha}\exp\left(\frac{L_1\eta}{\alpha}\right)\right)  \sum_{k=0}^K \sum_{t=0}^{k-1} (1-\alpha)^{k-t}  \theta\eta^2 \normstar{\nabla f(x_{t+1})} \\
    && + 2\eta{\color{red}\bar\rho} \sqrt{\alpha}\sigma_g(K+1),
\end{eqnarray*}
where $\Delta = \Exp{f(x_0)-f_{\inf}}$.

Next, since 
\begin{eqnarray*}
    \sum_{k=0}^K (1-\alpha)^k  \leq \frac{1}{\alpha}, &\text{and} & \sum_{k=0}^K \sum_{t=0}^{k-1} (1-\alpha)^{k-t} \normstar{\nabla f(x_{t+1})}
    \leq \frac{1}{\alpha} \sum_{k=0}^K \normstar{\nabla f(x_k)},
\end{eqnarray*}
we obtain
\begin{eqnarray*}
    \sum_{k=0}^K \eta\vartheta  \Exp{\normstar{\nabla f(x_k)}}
    & \leq& \Delta  + \frac{L_0\eta^2}{2}  \exp(L_1 \eta)  (K+1) + 2  \frac{\eta}{\alpha}  \Exp{\normstar{e_0}}\\
    && + \left(M_0 + \frac{2M_1}{3}\cdot\frac{L_0 \eta}{\alpha} \exp\left(\frac{L_1\eta}{\alpha}\right)\right) \frac{\eta^3}{\alpha^2}(K+1)  \\
    && +   {\color{red}\bar\rho} \sqrt{\alpha}\sigma_g(K+1),
\end{eqnarray*}
where $\vartheta \eqdef \left( 1 - \exp(L_1 \eta) 
 \frac{L_1\eta}{2}   -  \left(M_1 + \frac{2M_1}{3} \cdot\frac{L_1 \eta}{\alpha} \exp\left(\frac{L_1 \eta}{\alpha}\right) \right) \frac{\eta^2}{\alpha^2}  \right)$.

 If $\eta \leq \min\left\{ \frac{1}{3L_1} , \frac{1}{3\sqrt{M_1}}  \right\} \alpha$, then $\eta \leq \min\left\{ \frac{1}{3L_1} ,  \frac{1}{3\sqrt{M_1}} \right\}$ and 
\begin{eqnarray*}
    \frac{\eta}{2}\sum_{k=0}^K   \Exp{\normstar{\nabla f(x_k)}}   
    &\leq& \Delta  + (K+1)\frac{L_0}{2}  \exp(L_1 \eta) \eta^2  + 2  \frac{\eta}{\alpha}  \Exp{\normstar{e_0}} + 2\eta   {\color{red}\bar\rho} \sqrt{\alpha}\sigma_g(K+1)\\
    &&+ \left(M_0 + \frac{2M_1}{3}\cdot\frac{L_0 \eta}{\alpha} \exp\left(\frac{L_1\eta}{\alpha}\right)\right) \frac{\eta^3}{\alpha^2}(K+1).
\end{eqnarray*}

Therefore, 
\begin{eqnarray*}
    \frac{1}{K+1}\sum_{k=0}^K   \Exp{\normstar{\nabla f(x_k)}} 
    & \leq & \frac{2\Delta}{\eta(K+1)}  +L_0 \exp(L_1 \eta) \eta  +  \frac{4\Exp{\normstar{e_0}}}{\alpha (K+1)}  + 4   {\color{red}\bar\rho} \sqrt{\alpha}\sigma_g\\
    && + 2 \left(M_0 + \frac{2M_1}{3}\cdot\frac{L_0 \eta}{\alpha} \exp\left(\frac{L_1\eta}{\alpha}\right)\right) \frac{\eta^2}{\alpha^2}.
\end{eqnarray*}

If $\eta = \frac{\hat \eta}{(K+1)^{\nicefrac{5}{7}}} $ with $\hat \eta = \min\left\{ \frac{1}{3L_1} , \frac{1}{3\sqrt{M_1}}  \right\} $ and $\alpha = \frac{1}{(K+1)^{\nicefrac{4}{7}}}$, then 
\begin{eqnarray*}
    \frac{1}{K+1}\sum_{k=0}^K   \Exp{\normstar{\nabla f(x_k)}}   
    &\leq& \frac{2\Delta}{\hat\eta(K+1)^{\nicefrac{2}{7}}}  + \frac{\hat\eta L_0 \exp(L_1 \hat \eta)}{(K+1)^{\nicefrac{5}{7}}}  + \frac{4\Exp{\normstar{e_0}}}{ (K+1)^{\nicefrac{2}{7}}} +   4\frac{{\color{red}\bar\rho} \sigma_g}{(K+1)^{\nicefrac{2}{7}}}\\
    && + 2\left(M_0 + \frac{2M_1}{3} L_0\hat \eta \exp(L_1\hat \eta)\right) \frac{\hat \eta^2}{(K+1)^{\nicefrac{2}{7}}}.
\end{eqnarray*}

Finally, by the fact that $\min_{k \in \{0,1,\ldots,K\}} \Exp{\normstar{\nabla f(x_k)}}
\leq \frac{1}{K+1}\sum_{k=0}^K \Exp{\normstar{\nabla f(x_k)}}$, we obtain the final result.

\newpage
\section{LMO-based Methods with Momentum Variance Reduction}\label{app:MVR_relaxed_smoothness}

This section establishes the convergence properties of the LMO-based algorithms described in~\eqref{eqn:LMO_momentum}, utilizing the momentum-based variance reduction update given by \eqref{eq:MARS}.

To facilitate the convergence analysis of MVR, STORM, and MARS, we introduce an additional assumption. We refer to this condition as the \textit{generalized mean-squared smoothness assumption} \citep{chen2023generalized}, as it generalizes the standard mean-squared smoothness assumption proposed by \citet{arjevani2023lower}.

\begin{assumption}
\label{assum:Relaxed_Mean_Squared_Lipschitz_Grad}
    The stochastic gradients of $f$ are mean-squared symmetrically $(\cL_0, \cL_1)$-Lipschitz continuous with respect to the norm $\|\cdot\|$, i.e., for all $x, y \in \R^d$,
    \begin{eqnarray*}
        \ExpSub{\xi}{\normstar{\nabla f_{\xi}(x)-\nabla f_{\xi}(y)}^2}
        &\leq& \ExpSub{\xi}{\left(\cL_0+\cL_1\sup_{\theta \in [0,1]}\normstar{\nabla f_{\xi}(\theta x + (1-\theta)y)}\right)^2} \norm{x-y}^2.
    \end{eqnarray*}
\end{assumption}

As shown by \citet{chen2023generalized}, \Cref{assum:Relaxed_Mean_Squared_Lipschitz_Grad} is a stronger condition than the generalized smoothness in \Cref{assum:Relaxed_Lipschitz_Grad}.

\begin{lemma}[Proposition 5.3 from \citep{chen2023generalized}]
\label{lemma:Relaxed_Mean_Squared_Lipschitz_Grad}
    Let $f$ satisfy Assumption~\ref{assum:Relaxed_Mean_Squared_Lipschitz_Grad}. Then, for all $x,y \in \R^d$,
    \begin{eqnarray*}
        \ExpSub{\xi}{\normstar{\nabla f(x) -\nabla f(y)}^2} &\leq& 2\left(\cL_0^2 + 2\cL_1^2\ExpSub{\xi}{\normstar{\nabla f_{\xi}(y)}^2}\right)\exp\left(12\cL_1^2\norm{x-y}^2\right)\norm{x-y}^2.
    \end{eqnarray*}
\end{lemma}

\begin{proof}
    The proof adapts the arguments of Proposition 5.3 in \citep{chen2023generalized} to the setting of arbitrary norms.
\end{proof}

Now, we provide the main convergence theorem.

\begin{theorem}
    \label{thm:MVR_relaxed_smoothness}
    Consider the problem of minimizing $f(x)=\ExpSub{\xi \sim \cD}{f_{\xi}(x)}$. Suppose Assumptions~\ref{assum:Hessian_boundedvariance_v2},~\ref{assum:lower_bound},~\ref{assum:Relaxed_Lipschitz_Grad}, and~\ref{assum:Relaxed_Mean_Squared_Lipschitz_Grad} hold.
    Then, the iterates $\{x_k\}$ generated by the LMO-based method in~\eqref{eqn:LMO_momentum} using the momentum variance reduction in~\eqref{eq:MARS} with  $\beta_k = 1-\alpha_k$, and  step sizes
    \begin{eqnarray*}
        \alpha_k = \alpha = \frac{1}{(K+1)^{\nicefrac{2}{3}}} \quad \text{and} \quad \eta_k= \eta = \frac{\hat \eta}{(K+1)^{\nicefrac{2}{3}}},
    \end{eqnarray*}
    where $\hat \eta = \frac{1}{3}\min\left\{\frac{1}{L_1}, \frac{1}{11\cL_1}\left(\frac{\color{red}\bar \rho}{\color{red} \ubar \rho}\right)^{-1}\right\}$, satisfy
    \begin{eqnarray*}
        \min_{k \in \{0,1,\ldots,K\}} \Exp{\normstar{\nabla f(x_k)}}
        &\leq& \frac{2(f(x_0) - f_{\inf})}{\hat\eta(K+1)^{\nicefrac{1}{3}}}  + \frac{4\left( \Exp{ \normstar{e_0}} + {\color{red} \bar \rho} \sigma_g\right)}{ (K+1)^{\nicefrac{1}{3}}} +  \frac{\hat\eta L_0 \exp(L_1 \hat \eta)}{(K+1)^{\nicefrac{2}{3}}} \\
        && + 8 \frac{\color{red}\bar \rho}{\color{red} \ubar \rho} \left(\cL_0 + 2\cL_1 {\color{red} \bar \rho} \sigma_g\right)\exp\left(6 \cL^2_1\hat \eta^2\right)\frac{\hat \eta}{(K+1)^{\nicefrac{1}{3}}}
    \end{eqnarray*}
\end{theorem}

\subsection{Proof of \Cref{thm:MVR_relaxed_smoothness}}

We prove the result in the following steps. 

\paragraph{Step 1) Proving the descent inequality.} 
By following the proof arguments in Step 1) of~\Cref{lemma:V1_relaxed_smoothness} (see \eqref{eq:descent_inequality}), we have 
\begin{eqnarray*}
    \sum_{k=0}^K \eta_k\varphi_k  \normstar{\nabla f(x_k)}    &\leq&  \Delta + 2 \sum_{k=0}^K \eta_k \normstar{\nabla f(x_k)-m_k}  + \frac{L_0}{2} \sum_{k=0}^K \exp(L_1 \eta_k) \eta_k^2,
\end{eqnarray*}
where $\varphi_k \eqdef \left( 1 - \exp(L_1 \eta_k) 
 \nicefrac{L_1\eta_k}{2} \right)$ and $\Delta \eqdef f(x_0) - f_{\inf}$.

\paragraph{Step 2) Bounding the error term.}
Next, we bound $\normstar{e_k}$, where $e_k \eqdef  m_k - \nabla f(x_k)$. 
From the definition of $e_k$,
\begin{eqnarray*}
    e_{k+1} 
    & = & m_{k+1} - \nabla f(x_{k+1}) \\
    & = & (1-\alpha_k)\left(m_k - \nabla f(x_k)\right) + (1-\alpha_k)\left[(\nabla f_{\xi_{k+1}}(x_{k+1}) - \nabla f_{\xi_{k+1}}(x_{k})) - (\nabla f(x_{k+1}) - \nabla f(x_{k}))\right]\\
    && +\alpha_k(\nabla f_{\xi_{k+1}}(x_{k+1}) - \nabla f(x_{k+1}))\\
    &=& (1-\alpha_k)e_k + (1-\alpha_k)\hat S_{k+1} + \alpha_k \varepsilon^g_{k+1},
\end{eqnarray*}
 where we denote $\hat S_{t+1} = (\nabla f_{\xi_{t+1}}(x_{t+1}) - \nabla f_{\xi_{t+1}}(x_{t})) - (\nabla f(x_{t+1}) - \nabla f(x_{t})) $ and $\varepsilon^g_{t+1} = \nabla f_{\xi_{t+1}}(x_{t+1}) - \nabla f(x_{t+1})$ for any $t \geq 0$.

By recursively applying the above inequality, 
\begin{equation*}
    e_{k+1} = \prod^{k}_{t=0}(1-\alpha_t)e_0 + \sum^k_{t=0} \left(\prod^k_{j = t+1}(1-\alpha_j)\right)(1-\alpha_t)\hat S_{t+1} + \sum^k_{t=0} \left(\prod^k_{j = t+1}(1-\alpha_j)\right)\alpha_t\varepsilon^g_{t+1}.
\end{equation*}
Therefore, 
\begin{eqnarray*}
    \Exp{\normstar{e_{k+1}}} &\leq& \Exp{\normstar{\prod^{k}_{t=0}(1-\alpha_t)e_0}} + \underbrace{\Exp{\normstar{\sum^k_{t=0} \left(\prod^k_{j = t+1}(1-\alpha_j)\right)(1-\alpha_t)\hat S_{t+1}}}}_{\eqdef \circledNine}\\
    && + \underbrace{\Exp{\normstar{\sum^k_{t=0} \left(\prod^k_{j = t+1}(1-\alpha_j)\right)\alpha_t\varepsilon^g_{t+1}}}}_{\eqdef \circledTen}
\end{eqnarray*}

Next, we bound $\circledTen$ by using~\Cref{assum:Hessian_boundedvariance_v2} and by following the proof arguments for bounding $\circledOne$ in Step 2) of the convergence proof for~\Cref{lemma:V1_relaxed_smoothness}. Then, we have 
\begin{eqnarray*}
    \circledTen
    & \leq & {\color{red}\bar\rho}   \Exp{\normeu{\sum_{t=0}^k \left(\prod_{j=t+1}^k (1-\alpha_j) \right) \alpha_t \varepsilon^g_{t+1}}} \\
    & \leq & {\color{red}\bar\rho}  \sqrt{  \Exp{\normeu{\sum_{t=0}^k \left(\prod_{j=t+1}^k (1-\alpha_j) \right) \alpha_t \varepsilon^g_{t+1}}^2} } \\
    & \leq & {\color{red}\bar\rho} \sigma_g \sqrt{ \sum_{t=0}^k \left(\prod_{j=t+1}^k (1-\alpha_j)^2 \right) \alpha_t^2}, 
\end{eqnarray*}
Therefore, 
\begin{eqnarray*}
    \Exp{\normstar{e_{k+1} }}
    & \leq & \prod_{t=0}^k (1-\alpha_t)\Exp{\normstar{e_0}} + ~\circledNine~  +   {\color{red}\bar\rho} \sigma_g \sqrt{ \sum_{t=0}^k \left(\prod_{j=t+1}^k (1-\alpha_j)^2 \right) \alpha_t^2}.
\end{eqnarray*}

Next, we bound $\circledNine$. 
\begin{eqnarray*}
    \circledNine &\leq& {\color{red}\bar\rho} \Exp{\norm{\sum^k_{t=0} \left(\prod^k_{j = t+1}(1-\alpha_j)\right)(1-\alpha_t)\hat S_{t+1}}_2}\\
    &=& {\color{red}\bar\rho} \Exp{\ExpSub{k+1}{\norm{\sum^{k-1}_{t=0} \left(\prod^k_{j = t+1}(1-\alpha_j)\right)(1-\alpha_t)\hat S_{t+1} + (1-\alpha_k)\hat S_{k+1}}_2}}.
\end{eqnarray*}
By Jensen's inequality, we have 
\begin{eqnarray*}
    \circledNine &\leq&  {\color{red}\bar\rho} \Exp{\left(\ExpSub{k+1}{\norm{\sum^{k-1}_{t=0} \left(\prod^k_{j = t+1}(1-\alpha_j)\right)(1-\alpha_t)\hat S_{t+1} + (1-\alpha_k)\hat S_{k+1}}_2^2}\right)^{\nicefrac{1}{2}}}\\
    &=& {\color{red}\bar\rho} \Exp{\left(\norm{\sum^{k-1}_{t=0} \left(\prod^k_{j = t+1}(1-\alpha_j)\right)(1-\alpha_t)\hat S_{t+1}}^2_2 + (1-\alpha_k)^2\ExpSub{k+1}{\norm{\hat S_{k+1}}_2^2}\right)^{\nicefrac{1}{2}}}.
\end{eqnarray*}
To continue, we need to bound $\ExpSub{t+1}{\norm{\hat S_{t+1}}_2^2}$ for any $t \geq 0$. By \Cref{lemma:Relaxed_Mean_Squared_Lipschitz_Grad}
\begin{eqnarray*}
    \ExpSub{t+1}{\norm{\hat S_{t+1}}_2^2} &\leq& \ExpSub{t+1}{\norm{\nabla f_{\xi_{t+1}}(x_{t+1}) - \nabla f_{\xi_{t+1}}(x_{t})}_2^2} \\
    &\leq&\frac{1}{\color{red} \ubar{\rho}^2}\ExpSub{t+1}{\normstar{\nabla f_{\xi_{t+1}}(x_{t+1}) - \nabla f_{\xi_{t+1}}(x_{t})}^2}\\
    &\leq& \frac{2}{\color{red} \ubar{\rho}^2}\left(\cL_0^2 +2 \cL_1^2 \ExpSub{t+1}{\normstar{\nabla f_{\xi_{t+1}}(x_t)}^2}\right)\exp\left(12 \cL^2_1\norm{x_{t+1} - x_t}^2\right)\norm{x_{t+1} - x_t}^2\\
    &\leq& \frac{2}{\color{red} \ubar{\rho}^2}\left(\cL_0^2 +2 \cL_1^2 \ExpSub{t+1}{\normstar{\nabla f_{\xi_{t+1}}(x_t)}^2}\right)\exp\left(12 \cL^2_1\eta_t^2\right)\eta_t^2\\
    &\leq& \frac{2}{\color{red} \ubar{\rho}^2}\left(\cL_0^2 +4 \cL_1^2 \ExpSub{t+1}{\normstar{\nabla f_{\xi_{t+1}}(x_t) - \nabla f_{\xi_{t+1}}(x_t)}^2}\right)\exp\left(12 \cL^2_1\eta_t^2\right)\eta_t^2\\
    && + \frac{8}{\color{red} \ubar{\rho}^2}  \eta_t^2\cL_1^2 \exp\left(12 \cL^2_1\eta_t^2\right)\normstar{\nabla f(x_t)}^2\\
    &\leq& \underbrace{\left(\frac{2}{\color{red} \ubar{\rho}^2}\eta_t^2\cL_0^2 +\frac{8{\color{red}\bar\rho^2}}{\color{red} \ubar{\rho}^2}\eta_t^2\cL_1^2  \sigma^2_g \right)\exp\left(12 \cL^2_1\eta_t^2\right) }_{\eqdef A_{t}^2} + \underbrace{\frac{8}{\color{red} \ubar{\rho}^2}  \eta_t^2\cL_1^2 \exp\left(12 \cL^2_1\eta_t^2\right)\normstar{\nabla f(x_t)}^2}_{\eqdef B_{t}^2}.
\end{eqnarray*}
Therefore,
\begin{eqnarray*}
    \circledNine &\leq&  {\color{red}\bar\rho} \Exp{\left(\ExpSub{k+1}{\norm{\sum^{k-1}_{t=0} \left(\prod^k_{j = t+1}(1-\alpha_j)\right)(1-\alpha_t)\hat S_{t+1} + (1-\alpha_k)\hat S_{k+1}}_2^2}\right)^{\nicefrac{1}{2}}}\\
    &=& {\color{red}\bar\rho} \Exp{\left(\norm{\sum^{k-1}_{t=0} \left(\prod^k_{j = t+1}(1-\alpha_j)\right)(1-\alpha_t)\hat S_{t+1}}^2_2 + (1-\alpha_k)^2A_k^2 \right)^{\nicefrac{1}{2}}} + {\color{red}\bar\rho}(1-\alpha_k) \Exp{B_k}\\
    &=& {\color{red}\bar\rho} \Exp{\ExpSub{k}{\left(\norm{\sum^{k-2}_{t=0} \left(\prod^k_{j = t+1}(1-\alpha_j)\right)(1-\alpha_t)\hat S_{t+1} + (1-\alpha_{k-1})(1-\alpha_k)\hat S_{k} }^2_2 + (1-\alpha_k)^2A_k^2 \right)^{\nicefrac{1}{2}}}} \\
    && + {\color{red}\bar\rho}(1-\alpha_k) \Exp{B_k}.
\end{eqnarray*}
By Jensen's inequality, we have 
\begin{eqnarray*}
    \circledNine &\leq&  {\color{red}\bar\rho}\Exp{\left( \ExpSub{k}{\norm{\sum^{k-2}_{t=0} \left(\prod^k_{j = t+1}(1-\alpha_j)\right)(1-\alpha_t)\hat S_{t+1} + (1-\alpha_{k-1})(1-\alpha_k)\hat S_{k} }^2_2} + (1-\alpha_k)^2A_k^2 \right)^{\nicefrac{1}{2}}}\\
    && + {\color{red}\bar\rho}(1-\alpha_k) \Exp{B_k}\\
    &=& {\color{red}\bar\rho}\Exp{\left( \norm{\sum^{k-2}_{t=0} \left(\prod^k_{j = t+1}(1-\alpha_j)\right)(1-\alpha_t)\hat S_{t+1}}^2_2 + (1-\alpha_{k-1})^2(1-\alpha_k)^2\ExpSub{k}{\norm{\hat S_{k} }^2_2} + (1-\alpha_k)^2A_k^2 \right)^{\nicefrac{1}{2}}}\\
    && + {\color{red}\bar\rho}(1-\alpha_k) \Exp{B_k}\\
    &\leq& {\color{red}\bar\rho}\Exp{\left( \norm{\sum^{k-2}_{t=0} \left(\prod^k_{j = t+1}(1-\alpha_j)\right)(1-\alpha_t)\hat S_{t+1}}^2_2 + (1-\alpha_{k-1})^2(1-\alpha_k)^2A^2_{k-1} + (1-\alpha_k)^2A_k^2 \right)^{\nicefrac{1}{2}}}\\
    && + {\color{red}\bar\rho}(1-\alpha_k) \Exp{B_k} + {\color{red}\bar\rho}(1-\alpha_{k-1})(1-\alpha_k)\Exp{B_{k-1}}.
\end{eqnarray*}
Expanding the sum under the norm as before yields
\begin{eqnarray*}
    \circledNine &\leq& {\color{red}\bar\rho}\left( \sum^{k}_{t=0} \left(\prod^k_{j = t}(1-\alpha_j)^2\right)A^2_{t} \right)^{\nicefrac{1}{2}} + \sum^{k}_{t=0} \left(\prod^k_{j = t}(1-\alpha_j)\right)\Exp{B_{t}}\\
    &\leq& \left(2\cL_0 +4\cL_1  {\color{red}\bar\rho}\sigma_g \right)\frac{{\color{red}\bar\rho}}{\color{red} \ubar{\rho}}\left( \sum^{k}_{t=0} \left(\prod^k_{j = t}(1-\alpha_j)^2\right)\eta^2_t\exp\left(12 \cL^2_1\eta_t^2\right) \right)^{\nicefrac{1}{2}} \\
    && + 4\frac{{\color{red}\bar\rho}}{\color{red} \ubar{\rho}}\sum^{k}_{t=0} \left(\prod^k_{j = t}(1-\alpha_j)\right)  \eta_t\cL_1 \exp\left(6 \cL^2_1\eta_t^2\right)\Exp{\normstar{\nabla f(x_t)}}.
\end{eqnarray*}

Therefore,
\begin{eqnarray*}
    \Exp{\normstar{e_{k+1} }}
    & \leq & \prod_{t=0}^k (1-\alpha_t)\Exp{\normstar{e_0}}   +   {\color{red}\bar\rho} \sigma_g \sqrt{ \sum_{t=0}^k \left(\prod_{j=t+1}^k (1-\alpha_j)^2 \right) \alpha_t^2}\\
    && + \left(2\cL_0 +4\cL_1  {\color{red}\bar\rho}\sigma_g \right)\frac{{\color{red}\bar\rho}}{\color{red} \ubar{\rho}}\left( \sum^{k}_{t=0} \left(\prod^k_{j = t}(1-\alpha_j)^2\right)\eta^2_t\exp\left(12 \cL^2_1\eta_t^2\right) \right)^{\nicefrac{1}{2}} \\
    && + 4\frac{{\color{red}\bar\rho}}{\color{red} \ubar{\rho}}\sum^{k}_{t=0} \left(\prod^k_{j = t}(1-\alpha_j)\right)  \eta_t\cL_1 \exp\left(6 \cL^2_1\eta_t^2\right)\Exp{\normstar{\nabla f(x_t)}}.
\end{eqnarray*}

If $\eta_k = \eta$ and $\alpha_k = \alpha$, then 
\begin{eqnarray*}
    \Exp{\normstar{e_{k+1} }}
    & \leq & (1-\alpha)^{k+1}\Exp{\normstar{e_0}}   +   {\color{red}\bar\rho} \sigma_g \sqrt{ \sum_{t=0}^k  (1-\alpha)^{2(k-t)}\alpha^2}\\
    && + 2\frac{{\color{red}\bar\rho}}{\color{red} \ubar{\rho}}\left(\cL_0 +2\cL_1  {\color{red}\bar\rho}\sigma_g \right)\eta\exp\left(6 \cL^2_1\eta^2\right)\left( \sum^{k}_{t=0} (1-\alpha)^{2(k-t+1)} \right)^{\nicefrac{1}{2}} \\
    && + 4\frac{{\color{red}\bar\rho}}{\color{red} \ubar{\rho}}\cL_1 \exp\left(6 \cL^2_1\eta^2\right)\sum^{k}_{t=0} (1-\alpha)^{k-t+1}\eta\Exp{\normstar{\nabla f(x_t)}}.
\end{eqnarray*}

Next, since 
\begin{eqnarray*}
    \sum_{t=0}^{k-1} (1-\alpha)^{2(k-t+1)}
    & \leq & \sum_{j=0}^\infty (1-\alpha)^{2j}  =  \frac{1}{1-(1-\alpha)^2} \leq \frac{1}{\alpha}, \\
    &\text{and}&\\
     \sum_{t=0}^{k-1} (1-\alpha)^{2(k-t)} \alpha^2  & \leq &  \alpha^2 \sum_{j=0}^\infty ((1-\alpha)^2)^j  = \frac{\alpha^2}{1-(1-\alpha)^2} = \frac{\alpha}{2-\alpha} 
     \overset{\alpha \in [0,1]}{\leq} \alpha,
\end{eqnarray*}
we obtain
\begin{eqnarray*}
    \Exp{\normstar{e_{k+1} }}
    & \leq & (1-\alpha)^{k+1}\Exp{\normstar{e_0}}   +   {\color{red}\bar\rho} \sigma_g \sqrt{\alpha} + 2\frac{{\color{red}\bar\rho}}{\color{red} \ubar{\rho}}\left(\cL_0 +2\cL_1  {\color{red}\bar\rho}\sigma_g \right)\exp\left(6 \cL^2_1\eta^2\right) \frac{\eta}{\sqrt{\alpha}} \\
    && + 4\frac{{\color{red}\bar\rho}}{\color{red} \ubar{\rho}}\cL_1 \exp\left(6 \cL^2_1\eta^2\right)\sum^{k}_{t=0} (1-\alpha)^{k-t+1}\eta\Exp{\normstar{\nabla f(x_t)}}.
\end{eqnarray*}

Therefore, recalling $\varphi =\eta( 1 - \exp(L_1 \eta) 
 \nicefrac{L_1\eta}{2} ) $, we have 
\begin{eqnarray*}   
    \sum_{k=0}^K  \eta \varphi \Exp{\normstar{\nabla f(x_k)}}   &\leq&  \Delta + \frac{L_0}{2} \sum_{k=0}^K \exp(L_1 \eta) \eta^2 + 2  \eta \sum_{k=0}^K \Exp{\normstar{e_k}}   \\
    &\leq& \Delta  + \frac{L_0}{2}  \exp(L_1 \eta) \eta^2(K+1)  + 2  \eta \sum_{k=0}^K (1-\alpha)^{k} \Exp{\normstar{e_0}}\\
    &&+  4\eta\frac{{\color{red}\bar\rho}}{\color{red} \ubar{\rho}}\left(\cL_0 +2\cL_1  {\color{red}\bar\rho}\sigma_g \right)\exp\left(6 \cL^2_1\eta^2\right) \frac{\eta}{\sqrt{\alpha}} (K+1) \\
    && + 8\eta \frac{{\color{red}\bar\rho}}{\color{red} \ubar{\rho}}\cL_1 \exp\left(6 \cL^2_1\eta^2\right) \sum_{k=0}^K \sum_{t=0}^{k-1} (1-\alpha)^{k-t+1}  \eta \normstar{\nabla f(x_{t+1})} \\
    && + 2\eta{\color{red}\bar\rho} \sqrt{\alpha}\sigma_g(K+1),
\end{eqnarray*}
where $\Delta = \Exp{f(x_0)-f_{\inf}}$.

Next, since 
\begin{eqnarray*}
    \sum_{k=0}^K (1-\alpha)^k  \leq \frac{1}{\alpha}, &\text{and} & \sum_{k=0}^K \sum_{t=0}^{k-1} (1-\alpha)^{k-t+1} \normstar{\nabla f(x_{t+1})}
    \leq \frac{1}{\alpha} \sum_{k=0}^K \normstar{\nabla f(x_k)},
\end{eqnarray*}
we obtain
\begin{eqnarray*}
    \sum_{k=0}^K \eta\vartheta  \Exp{\normstar{\nabla f(x_k)}}
    & \leq& \Delta  + \frac{L_0\eta^2}{2}  \exp(L_1 \eta)  (K+1) + 2  \frac{\eta}{\alpha} \Exp{ \normstar{e_0}}\\
    && + 4\eta\frac{{\color{red}\bar\rho}}{\color{red} \ubar{\rho}}\left(\cL_0 + 2\cL_1  {\color{red}\bar\rho}\sigma_g \right)\exp\left(6 \cL^2_1\eta^2\right) \frac{\eta}{\sqrt{\alpha}} (K+1) \\
    && +  2 {\color{red}\bar\rho} \sqrt{\alpha}\sigma_g(K+1),
\end{eqnarray*}
where $\vartheta \eqdef \left( 1 - \exp(L_1 \eta) 
 \frac{L_1\eta}{2}   -  8 \frac{{\color{red}\bar\rho}}{\color{red} \ubar{\rho}}\cL_1 \exp\left(6 \cL^2_1\eta^2\right)\frac{\eta}{\alpha}  \right)$.

 If $\eta \leq \min\left\{ \frac{1}{3L_1} , \frac{1}{33\cL_1}\left(\frac{{\color{red}\bar\rho}}{\color{red} \ubar{\rho}}\right)^{-1}  \right\} \alpha$, then $\eta \leq \min\left\{ \frac{1}{3L_1} , \frac{1}{33\cL_1}\left(\frac{{\color{red}\bar\rho}}{\color{red} \ubar{\rho}}\right)^{-1}  \right\}$ and 
\begin{eqnarray*}
    \frac{\eta}{2}\sum_{k=0}^K   \Exp{\normstar{\nabla f(x_k)}}   
    &\leq& \Delta  + (K+1)\frac{L_0}{2}  \exp(L_1 \eta) \eta^2  + 2  \frac{\eta}{\alpha}  \Exp{\normstar{e_0}} + 2\eta   {\color{red}\bar\rho} \sqrt{\alpha}\sigma_g(K+1)\\
    &&+ 4\eta\frac{{\color{red}\bar\rho}}{\color{red} \ubar{\rho}}\left(\cL_0 + 2\cL_1  {\color{red}\bar\rho}\sigma_g \right)\exp\left(6 \cL^2_1\eta^2\right) \frac{\eta}{\sqrt{\alpha}} (K+1).
\end{eqnarray*}

Therefore, 
\begin{eqnarray*}
    \frac{1}{K+1}\sum_{k=0}^K   \Exp{\normstar{\nabla f(x_k)}} 
    & \leq & \frac{2\Delta}{\eta(K+1)}  +L_0 \exp(L_1 \eta) \eta  +  \frac{4\Exp{\normstar{e_0}}}{\alpha (K+1)}  + 4   {\color{red}\bar\rho} \sqrt{\alpha}\sigma_g\\
    && +  8\frac{{\color{red}\bar\rho}}{\color{red} \ubar{\rho}}\left(\cL_0 + 2\cL_1  {\color{red}\bar\rho}\sigma_g \right)\exp\left(6 \cL^2_1\eta^2\right) \frac{\eta}{\sqrt{\alpha}} .
\end{eqnarray*}

If $\eta = \frac{\hat \eta}{(K+1)^{\nicefrac{2}{3}}} $ with $\hat \eta = \min\left\{ \frac{1}{3L_1} , \frac{1}{33\cL_1}\left(\frac{{\color{red}\bar\rho}}{\color{red} \ubar{\rho}}\right)^{-1}  \right\} $ and $\alpha = \frac{1}{(K+1)^{\nicefrac{2}{3}}}$, then 
\begin{eqnarray*}
    \frac{1}{K+1}\sum_{k=0}^K   \Exp{\normstar{\nabla f(x_k)}}   
    &\leq& \frac{2\Delta}{\hat\eta(K+1)^{\nicefrac{1}{3}}}  + \frac{\hat\eta L_0 \exp(L_1 \hat \eta)}{(K+1)^{\nicefrac{2}{3}}}  + \frac{4 \Exp{\normstar{e_0}}}{ (K+1)^{\nicefrac{1}{3}}} +   4\frac{{\color{red}\bar\rho} \sigma_g}{(K+1)^{\nicefrac{1}{3}}}\\
    && +  8\frac{{\color{red}\bar\rho}}{\color{red} \ubar{\rho}}\left(\cL_0 + 2\cL_1  {\color{red}\bar\rho}\sigma_g \right)\exp\left(6 \cL^2_1\hat\eta^2\right) \frac{\hat\eta}{(K+1)^{\nicefrac{1}{3}}} .
\end{eqnarray*}

Finally, by the fact that $\min_{k \in \{0,1,\ldots,K\}} \Exp{\normstar{\nabla f(x_k)}}
\leq \frac{1}{K+1}\sum_{k=0}^K \Exp{\normstar{\nabla f(x_k)}}$, we obtain the final result.

\newpage
\section{Experiments}\label{sec:exp}
To ensure a fair and reproducible comparison, the initialization and hyperparameter schedules were standardized across all experiments.

\subsection{Initialization}
For each experiment, a single initial parameter vector, $x_0$, was generated by drawing from a normal (Gaussian) distribution, $x_0 \sim \mathcal{N}(0, 1)$. The random seed was fixed to ensure that this exact same starting point was used for every algorithm evaluated in that experiment. This "far start" initialization is designed to test the robustness and convergence capabilities of the optimizers from a non-trivial region of the parameter space.
\subsection{Stepsize and Momentum Schedules}
The learning rate $\eta_k$ and the momentum parameter $\alpha_k$ are decayed at each iteration $k$. The schedules are chosen based on the theoretical underpinnings of each class of algorithm. Let $\eta_0$ be a pre-defined initial learning rate.

\begin{itemize}
    \item \textbf{For Polyak Momentum:} This first-order method uses its standard theoretically-backed schedule:
    \begin{align*}
        \alpha_k = \frac{1}{\sqrt{k+1}}, \quad
        \eta_k = \frac{\eta_0}{(k+1)^{3/4}}.
    \end{align*}

    \item \textbf{For Extrapolated Momentum:} This method uses a distinct schedule designed for its update rule:
    \begin{align*}
        \alpha_k = \frac{1}{(k+1)^{4/7}}, \quad
        \eta_k = \frac{\eta_0}{(k+1)^{5/7}}.
    \end{align*}

    \item \textbf{For all Second-Order Momentum Variants (SOM-V1, SOM-V2, $\beta$-SOM-V1, and $\beta$-SOM-V2):} These methods share a common schedule for their primary learning rate and momentum parameters:
    \begin{align*}
        \alpha_k = \frac{1}{(k+1)^{2/3}}, \quad
        \eta_k = \frac{\eta_0}{(k+1)^{2/3}}.
    \end{align*}
\end{itemize}
These schedules ensure that the step sizes and momentum contributions diminish over time, a necessary condition for convergence in stochastic optimization.
\subsection{Nonconvex Logistic Regression}
We evaluate the performance of several stochastic optimization algorithms on a composite non-convex problem. The objective is to benchmark their convergence speed and stability on a logistic regression task augmented with a non-convex regularizer.

The model is a standard logistic regression classifier. The experiment is conducted on the \texttt{splice} dataset from the libsvm library, which contains $1000$ training samples and $60$ features.

The optimization objective is to minimize a composite function $f(x)$, defined as:
\begin{equation*}
    \min_{x \in \mathbb{R}^d} f(x) = \mathcal{L}(x) + R(x)
\end{equation*}
where $x \in \mathbb{R}^{60}$ is the vector of model parameters. The loss term $\mathcal{L}(x)$ is the standard mean logistic loss:
\begin{equation*}
    \mathcal{L}(x) = \frac{1}{N} \sum_{i=1}^{N} \log(1 + \exp(-y_i a_i^T x))
\end{equation*}
The regularization term $R(x)$ is the non-convex Welsch regularizer, defined as:
\begin{equation*}
    R(x) = \lambda \sum_{j=1}^{d} \frac{x_j^2}{1 + x_j^2}
\end{equation*}
The regularization hyperparameter $\lambda$ is set to $0.01$.

Figure~\ref{fig:main_comparison_lr} displays the convergence behavior of the three algorithms over $20000$ iterations. The training loss plot shows that the second-order methods containing Hessian information (Variant 1, Variant 2) outperform the first-order Polyak Momentum. Among these, Variant 2 achieves the fastest convergence and reaches the lowest final loss value.

The gradient norm plot corroborates these findings. While all methods exhibit stochastic oscillations, the overall trend for the second-order methods is a more rapid and consistent decrease in the gradient norm. The Polyak Momentum method converges to a region with a substantially higher gradient norm, indicating a less optimal solution.

\begin{figure}[h!]
    \centering
    \includegraphics[width=0.47\textwidth]{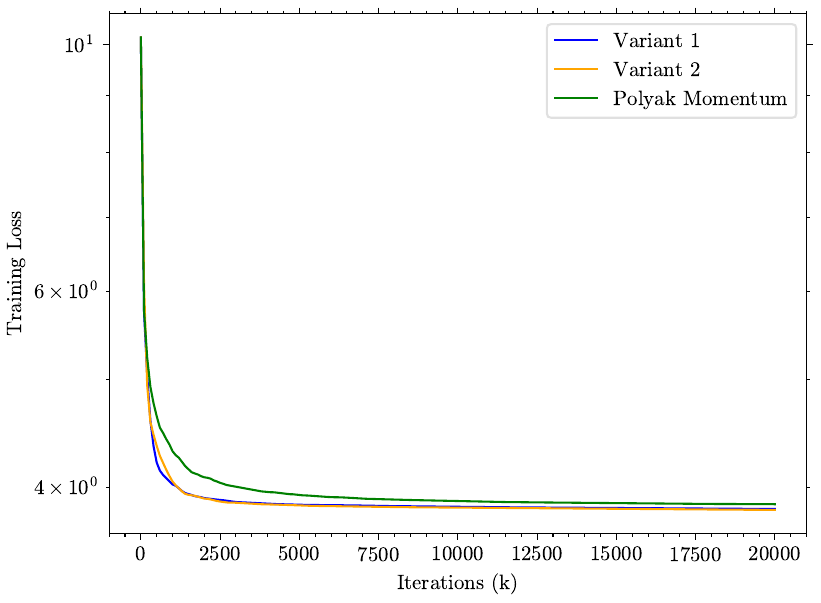}
    \includegraphics[width=0.47\textwidth]{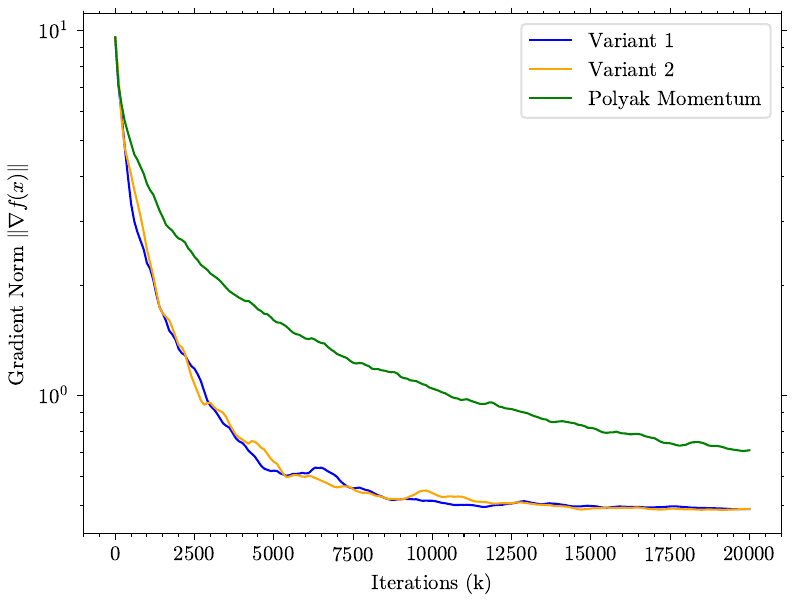}
    \caption{Main algorithm comparison for logistic regression. Training loss vs. iterations. Gradient norm vs. iterations.}
    \label{fig:main_comparison_lr}
\end{figure}
\subsection{Multi-Layer Perceptron}
We investigate the performance of several stochastic optimization algorithms on a non-convex binary classification problem. Our goal is to compare their convergence properties.

The model is a Multi-Layer Perceptron (MLP) with two hidden layers, implemented in PyTorch. The experiment is conducted on the \texttt{splice} dataset, obtained from the libsvm library. It consists of 1000 training samples, each with 60 features.

The optimization objective is to minimize a composite function $f(x)$, which includes a standard loss term and a non-convex regularizer:
\begin{equation*}
    \min_{x \in \mathbb{R}^d} f(x) = \mathcal{L}(x) + R(x)
\end{equation*}
where $x$ represents the flattened vector of all model parameters.

The loss term, $\mathcal{L}(x)$, is the mean Binary Cross-Entropy with Logits loss, calculated over the entire training dataset.

The regularization term, $R(x)$, is the non-convex Welsch regularizer, chosen to create a more challenging optimization landscape. It is defined as:
\begin{equation*}
    R(x) = \lambda \sum_{i=1}^{d} \frac{x_i^2}{1 + x_i^2}
\end{equation*}
where $\lambda$ is the regularization hyperparameter, set to $0.01$.

We set the theoretical learning rates for the algorithms considered.

The results are presented in two separate comparisons. Figure~\ref{fig:momentum_comparison} compares four momentum-based algorithms by plotting their running minimum performance. Among these, SOM-V2 achieves the lowest final training loss. Extrapolated Momentum and SOM-V1 perform similarly, while Polyak Momentum converges to a slightly higher loss. The running minimum plot of the gradient norm clarifies the convergence trend for each method by smoothing out stochastic noise.

\begin{figure}[h!]
    \centering
    \includegraphics[width=0.47\textwidth]{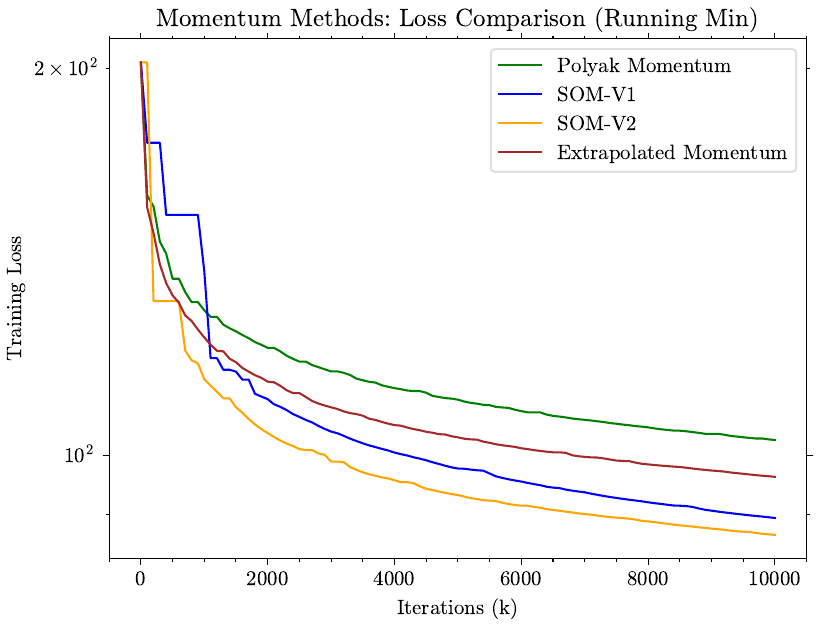}
    \includegraphics[width=0.47\textwidth]{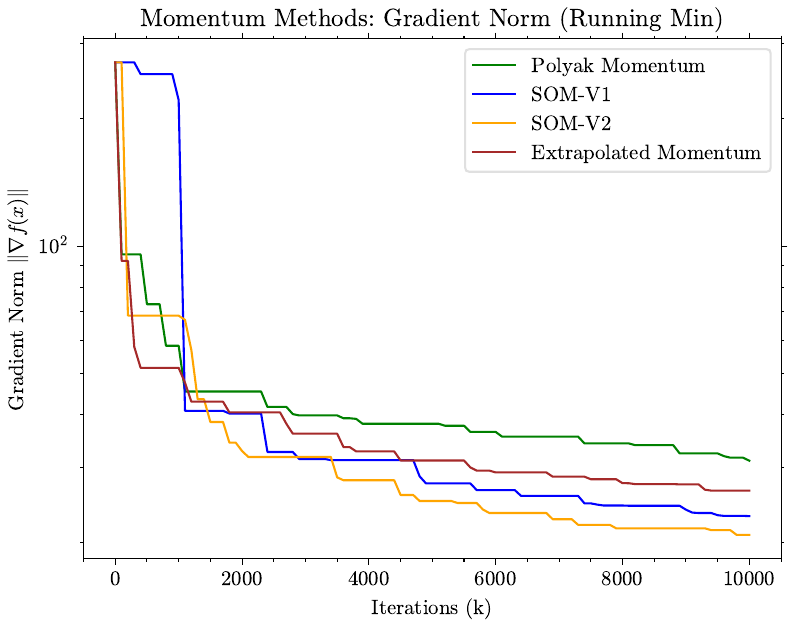}
    \caption{Comparison of momentum-based methods showing the running minimum (best value found so far). Left: Training loss vs. iterations. Right: Gradient norm vs. iterations.}
    \label{fig:momentum_comparison}
\end{figure}

Figure~\ref{fig:hvp_comparison} compares the standard second-order momentum methods (SOM-V1 and SOM-V2) with their $\beta$-HVP counterparts. The results clearly show the benefit of the additional Hessian term, as both $\beta$-SOM-V1 and $\beta$-SOM-V2 outperform their respective base variants. $\beta$-SOM-V2 demonstrates the strongest overall performance, converging to the lowest training loss of all tested algorithms. The running minimum gradient norm plot also makes it clear that the $\beta$-SOM methods, particularly $\beta$-SOM-V2, offer a more consistent decrease in gradient magnitude compared to the standard SOM-V1 and SOM-V2 methods.

\begin{figure}[h!]
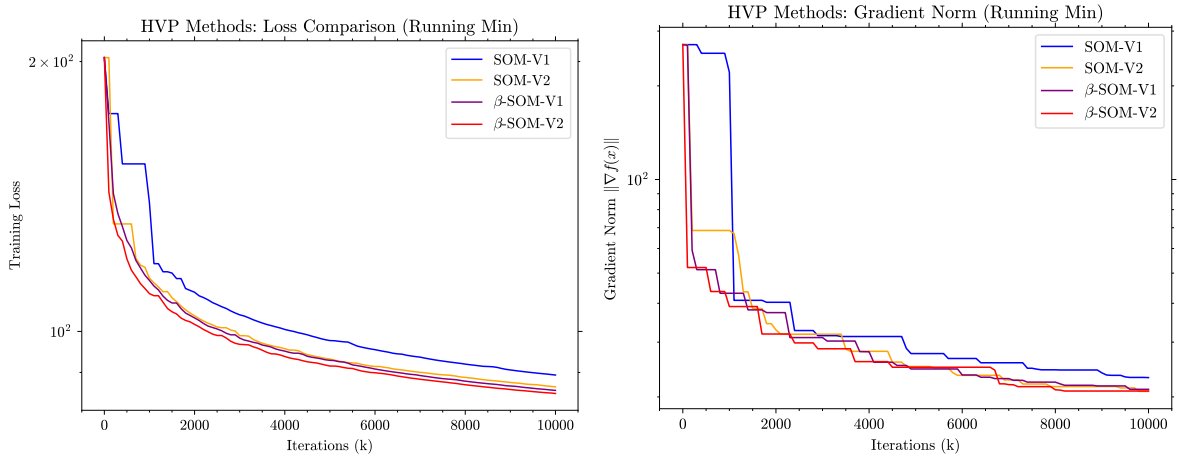

    \centering
    \includegraphics[width=0.47\textwidth]{plots/mlp_hvp_running_min_loss.pdf}
    \includegraphics[width=0.47\textwidth]{plots/mlp_hvp_running_min_grad.pdf}
    \caption{Comparison of HVP-based methods showing the running minimum (best value found so far). Left: Training loss vs. iterations. Right: Gradient norm vs. iterations.}
    \label{fig:hvp_comparison}
\end{figure}

\subsection{Recurrent Neural Network Training}
The model is a two-layer LSTM network designed for word-level language modeling. The architecture consists of an embedding layer ($200$ dimensions), a two-layer LSTM core with hidden units of size $200,$ and a final linear decoder layer to produce logits over the vocabulary.

The experiments are conducted on the Penn Treebank (PTB) dataset, a standard and widely-used benchmark for evaluating language models.

Key statistics and preprocessing steps are as follows:
\begin{itemize}
    \item \textbf{Corpus Size:} The dataset is split into training, validation, and testing sets. The training portion, used in our experiment, contains approximately $929,000$ tokens.
    \item \textbf{Vocabulary:} A dictionary is constructed from the unique words present in the training data. An end-of-sentence token, \texttt{<eos>}, is added to each sentence, resulting in a total vocabulary size of $10,000$ unique tokens.
    \item \textbf{Tokenization:} The raw text is tokenized by splitting on whitespace, and each word is converted into its corresponding integer index from the vocabulary.
    \item \textbf{Batching:} For training, the entire sequence of token IDs is reshaped into a fixed number of parallel streams (a batch size of $20$ in our case). The model is then trained on sequential chunks of this data using Truncated Backpropagation Through Time (BPTT) with a sequence length of $35.$
\end{itemize}

The objective is to minimize the standard Cross-Entropy Loss. The parameter update follows the LMO-based rule:
\begin{equation*}
    x_{k+1} = x_k + \text{lmo}(m_k)
\end{equation*}
where the Linear Minimization Oracle, $\text{lmo}(m_k)$, is defined as:
\begin{equation*}
    \text{lmo}(m_k) := \underset{\norm{v} \le \eta_k}{\argmin} \langle m_k, v \rangle.
\end{equation*}

We compare six algorithms: Polyak Momentum, SOM-V1, SOM-V2, our proposed $\beta$-SOM-V1 and $\beta$-SOM-V2 variants, and Extrapolated Momentum. The hyperparameter schedules for $\alpha_k$ and $\eta_k$ are set to their theoretical values for each respective algorithm.
\subsection{The case when $\norm{\cdot}$ is a Euclidean norm}
For this experiment, we use the Euclidean ($\ell_2$) norm. The results of the LMO-based training, presented in Figure~\ref{fig:lstm_lmo_loss}, reveal a significant divergence in performance between the first-order and more advanced momentum methods.
\begin{figure}[h!]
    \centering
    \includegraphics[width=0.47\textwidth]{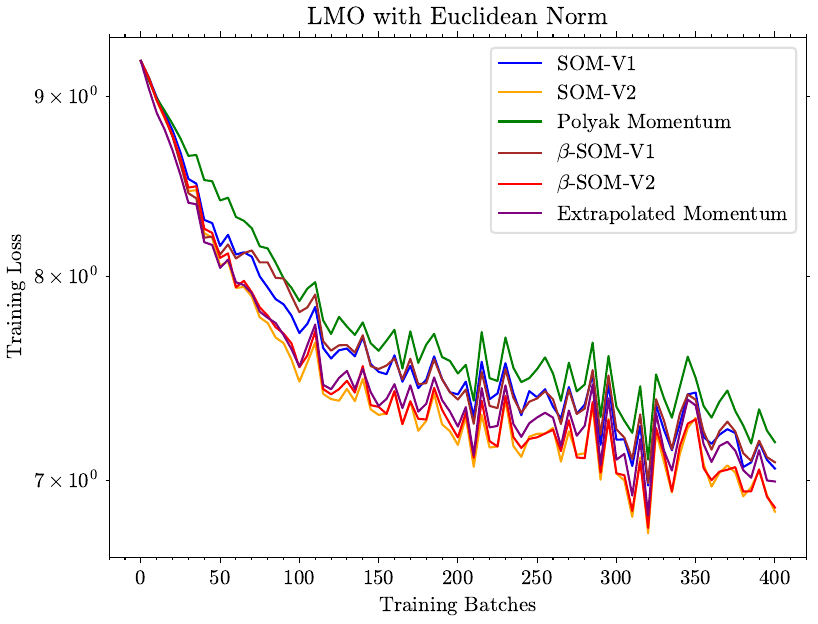}
    \caption{Training loss vs. training batches for the LSTM model using an LMO update with the Euclidean norm. The plot highlights the stronger performance of the other five methods over Polyak Momentum.}
    \label{fig:lstm_lmo_loss}
\end{figure}
The simple momentum update rule is less efficient at navigating the loss landscape for this task. In contrast, the other five algorithms, which all incorporate either second-order (HVP) or extrapolated gradient information, form a tight cluster of high-performing methods. They converge faster and to a lower loss value than Polyak Momentum. Within this cluster, SOM-V2 and $\beta$-SOM-V2 often achieve a marginally lower loss, suggesting that calculating the curvature information at the current point $x_k$ (rather than the interpolated point used by SOM-V1 and $\beta$-SOM-V1) may offer a slight advantage. The strong performance of these five methods indicates that the LMO update rule is a highly effective stabilization technique when paired with momentum strategies that utilize more sophisticated directional information.
\subsection{The case when $\norm{\cdot}$ is $\ell_{\infty}$-norm}
To further investigate the effect of the update geometry, we conducted a parallel experiment where the LMO is constrained by the Infinity norm ($\ell_\infty$). The resulting training loss curves are presented in Figure~\ref{fig:lstm_linf_loss}.
\begin{figure}[h!]
    \centering
    \includegraphics[width=0.47\textwidth]{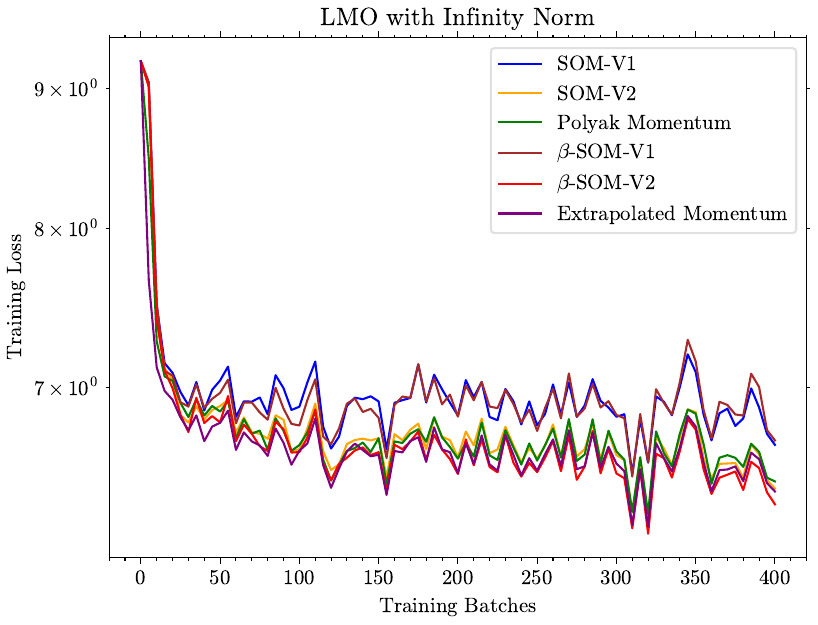}
    \caption{Training loss vs. training batches for the LSTM model using an LMO update with the Infinity norm. The optimization process is successful but visibly less stable than with the Euclidean norm.}
    \label{fig:lstm_linf_loss}
\end{figure}
The most immediate observation is that while all algorithms still successfully guide the training process toward a lower loss, the convergence path is significantly less stable. The loss curves for all methods are characterized by high-frequency, large-amplitude oscillations.

This instability is an expected consequence of the $\ell_\infty$ LMO. Unlike the smooth, normalized direction vector produced by the $\ell_2$ norm, the $\ell_\infty$ LMO generates a more aggressive update where every component of the step vector is pushed to its maximum value. This results in a more chaotic exploration of the parameter space, leading to the observed volatility in the training loss.

An important consequence of this instability is that the clear performance hierarchy observed in the Euclidean experiment has vanished. The loss curves for all six algorithms are tightly intertwined, and no single method demonstrates a consistent advantage. This suggests that the high variance and non-smooth nature of the $\ell_\infty$ update step dominate the more subtle directional corrections offered by the second-order and extrapolated momentum terms.

While the final loss values appear comparable to those achieved with the $\ell_2$ norm, the convergence path is significantly less stable, making the Euclidean norm the more reliable and predictable choice for this particular language modeling task.

\section*{LLM Use Acknowledgment}

In this paper, we used large language models (LLMs) to assist with grammar and wording during the preparation of the manuscript. We did not use LLMs to derive convergence theorems, generate empirical plots, or search for citations. This usage is in accordance with two primary LLM-related policies.

\end{document}